\documentclass[aps, pra, amsmath, amssymb, reperint, superscriptaddress, showkeys, tightenlines, twocolumn]{revtex4-2}

\usepackage{amssymb}
\usepackage{amsmath}
\usepackage[T1]{fontenc}
\usepackage{hyperref}
\usepackage{graphicx} 
\usepackage{quantikz} 
\usepackage{tikz} 
\usetikzlibrary {arrows.meta} 
\usetikzlibrary{external} 
\usepackage{xcolor}
\usepackage{caption}
\usepackage{subcaption}
\usepackage{tabularray}
\usepackage[export]{adjustbox} 
\usepackage[ruled,linesnumbered,noend]{algorithm2e} 
\SetKwComment{Comment}{/* }{ */} 
\usepackage{pict2e} 
\usepackage{relsize} 
\usepackage{comment} 
\usepackage{lipsum} 
\usepackage[normalem]{ulem} 
\usepackage{comment} 

\usepackage{amsthm} 
\theoremstyle{definition}
\newtheorem{theorem}{Theorem}[section] 
\newtheorem{lemma}[theorem]{Lemma}
\newtheorem{corollary}[theorem]{Corollary}

\makeatletter
\newcommand{\bigcomp}{%
  \DOTSB
  \mathop{\vphantom{\sum}\mathpalette\bigcomp@\relax}%
  \slimits@
}
\newcommand{\bigcomp@}[2]{%
  \begingroup\m@th
  \sbox\z@{$#1\sum$}%
  \setlength{\unitlength}{0.9\dimexpr\ht\z@+\dp\z@}%
  \vcenter{\hbox{%
    \begin{picture}(1,1)
    \bigcomp@linethickness{#1}
    \put(0.5,0.5){\circle{1}}
    \end{picture}%
  }}%
  \endgroup
}
\newcommand{\bigcomp@linethickness}[1]{%
  \linethickness{%
      \ifx#1\displaystyle 2\fontdimen8\textfont\else
      \ifx#1\textstyle 1.65\fontdimen8\textfont\else
      \ifx#1\scriptstyle 1.65\fontdimen8\scriptfont\else
      1.65\fontdimen8\scriptscriptfont\fi\fi\fi 3
  }%
}

\def\ps@pprintTitle{%
  \let\@oddhead\@empty
  \let\@evenhead\@empty
  \let\@oddfoot\@empty
  \let\@evenfoot\@empty
}

\makeatother

\begin{document}

\title{Local Equivalence Classes of Distance-Hereditary Graphs using Split Decompositions}

\author{Nicholas Connolly}
 \email{nicholas.connolly@oist.jp}
\affiliation{
Okinawa Institute of Science and Technology, 1919-1 Tancha, Onna-son, Kunigami-gun, 904-0495, Okinawa, Japan
}

\author{Shin Nishio}
 \email{shin.nishio@keio.jp}
\affiliation{
Graduate School of Science and Technology, Keio University, Yokohama, 223-8522, Kanagawa, Japan}
\affiliation{
Department of Physics \& Astronomy, University College London, London, WC1E 6BT, United Kingdom
}
\author{Kae Nemoto}
 \email{kae.nemoto@oist.jp}
\affiliation{
Okinawa Institute of Science and Technology, 1919-1 Tancha, Onna-son, Kunigami-gun, 904-0495, Okinawa, Japan
}

\begin{abstract}
Local complement is a graph operation formalized by Bouchet which replaces the neighborhood of a chosen vertex with its edge-complement.
This operation induces an equivalence relation on graphs; determining the size of the resulting equivalence classes is a challenging problem in general.
Bouchet obtained formulas only for paths and cycles, and brute-force methods are limited to very small graphs.
In this work, we extend these results by deriving explicit formulas for several broad families of distance-hereditary graphs, including complete multipartite graphs, clique-stars, and repeater graphs.
Our approach uses a technique known as split decomposition to establish upper bounds on equivalence class sizes, and we prove these bounds are tight through a combinatorial enumeration of the graphs' decomposed structure up to symmetry.

\begin{description}
\item[Keywords]
local complement, graph local equivalence, distance-hereditary graphs,\\split decomposition, complete multipartite, repeater graphs, clique-stars, graph orbit 
\end{description}
\end{abstract}

\maketitle

\section{Introduction}
\label{sect:introduction}

\subsection{Motivation}

Local complement is an operation on graphs, first introduced by Bouchet~\cite{bouchet1993recognizing}, which replaces the neighborhood of a chosen vertex with its edge-complement.
This operation defines an equivalence relation on graphs, and two graphs related in this way are said to be locally equivalent.
In this sense, graphs can be partitioned into disjoint equivalence classes, often referred to as local complement (LC) orbits~\cite{cabello2011optimal, adcock2020mapping}.
Determining the number of distinct graphs that appear in the LC orbit of a given graph is an interesting mathematical problem, and in recent years it has also received attention from the physics community because of its links to quantum information theory~\cite{hein2004multiparty,van2004graphical,dahlberg2018transforming,adcock2020mapping}.

The size of the LC orbit grows exponentially with respect to the number of vertices, which makes the task of enumeration a challenging one in general; Dahlberg et al. proved that counting the size of the orbit is \#P-complete~\cite{dahlberg2020counting}.
Bouchet obtained explicit formulas in special cases, such as path graphs and cycle graphs~\cite{bouchet1993recognizing}, but no such general formula is known for arbitrary graphs.
Other efforts at enumeration have relied on brute-force numerical searches~\cite{adcock2020mapping}, though these have been restricted to small graphs with up to nine vertices because of the rapid growth of the orbit size.
Nevertheless, when one considers graph families with a large amount of symmetry, the problem of counting locally equivalent graphs becomes more manageable.

In this work, we introduce a general approach for describing LC orbits based on the \textit{split decomposition}~\cite{cunningham1982decomposition,cunningham1980combinatorial}, a technique developed by Cunningham to break a graph into simpler irreducible components known as \textit{quotient graphs}.
Bouchet previously established that locally equivalent graphs share the same splits~\cite{bouchet1987reducing,bouchet1989connectivity}.
Building on this result, we show that there is a natural relationship between the LC orbit of a graph and the LC orbits of its quotient graphs, and that this relationship can be exploited to perform explicit orbit counts.
This method is especially effective for \textit{distance-hereditary} graphs~\cite{howorka1977characterization}, which decompose into trivial quotient graphs and therefore lend themselves to simple enumeration.

Several well-known graph families with highly symmetric structures are distance-hereditary, including complete graphs, complete multipartite graphs, and so-called repeater graphs.
By applying our method, we derive explicit formulas for the LC orbit size in a number of these families using exhaustive enumeration of the possible combinations of quotient graphs up to symmetry.
We also provide explicit sequences of local complements which transform one representative of an orbit into another.
In addition, we analyze quantities such as the edge count and vertex degrees across each equivalence class, features which are of particular interest for optimization questions in physics.

\subsection{Connection with Quantum Information Theory}

Beyond its mathematical appeal, this problem is motivated by the role graphs play in quantum information theory.
Certain quantum states, known as graph states~\cite{hein2004multiparty}, can be represented directly by simple graphs, and these states form a central resource for measurement-based quantum computation~\cite{raussendorf2001one,raussendorf2003measurement} and fusion-based quantum computation~\cite{bartolucci2023fusion}. Importantly, within the class of graph states, single-qubit Clifford operations admit a purely graph-theoretic description in terms of local complements (also called local complementations)~\cite{van2004graphical}.

Another important aspect is that graph states are a subclass of stabilizer states~\cite{gottesman1997stabilizer}, which have been extensively studied from the viewpoints of quantum error-correcting codes~\cite{cross2008codeword, schlingemann2001stabilizer} and computational complexity~\cite{bravyi2016improved}.  Moreover, every stabilizer state is local-Clifford equivalent to some graph state, i.~e., it can be transformed into a graph state by a sequence of single-qubit Clifford operations~\cite{schlingemann2001stabilizer}.

In this way, many physical questions can be rephrased as questions about locally equivalent graphs~\cite{van2004efficient, van2004graphical}. For a detailed overview of graph states and their applications, see~\cite{hein2006entanglement}.

The study of graph states has grown rapidly in recent years, with a wide range of results reported in the literature~\cite{cabello2011optimal,dahlberg2018transforming,dahlberg2022complexity,dahlberg2021entanglement,dahlberg2020transforming,adcock2020mapping,kumabe2024complexity,sharma2025minimizing,li2025generalized}. 
Much of this work is motivated by practical concerns, such as minimizing the experimental resources required to prepare an equivalent graph state.
This often reduces to the problem of identifying a graph in the LC orbit that is optimal with respect to some parameter, such as the number of edges or the maximum degree.
While heuristic algorithms can be used to explore the LC orbit, it is difficult to be certain that the solutions obtained are indeed optimal without a complete characterization of the equivalence class~\cite{sharma2025minimizing}.
Furthermore, even when local equivalence is known in principle, it may not be straightforward to identify the specific sequence of local complements that converts one graph into another~\cite{adcock2020mapping, concha2025mathsf}.

The results we present provide a means of addressing these questions.
Although we restrict our attention to distance-hereditary graphs, such graphs are already of considerable importance in communication theory, where the distance-hereditary property is closely related to the robustness of networks against loss of vertices~\cite{di1997routing, cicerone2003k}.
Photon loss error is a common phenomenon in optical systems, and hence using distance-hereditary (especially tree graph) networks can guard against this~\cite{varnava2006loss, zhan2020deterministic}.
Indeed, one of the families we analyze in detail—repeater graphs—was first introduced in the context of photonic quantum systems~\cite{azuma2015all}.
For these reasons, we hope that our characterizations of locally equivalent graphs will be of interest not only in mathematics but also in quantum information.

\subsection{Structure of this Paper}

Section~\ref{sect:background} begins with an introduction to the special graph families that will be studied here.
In Section~\ref{sect:local_complements} we define the local complement operation and describe the resulting equivalence classes, fixing our notation and outlining several equivalent ways to view these concepts.
Section~\ref{sect:split_decompositions} introduces the \textit{split decomposition} and presents the quotient-augmented strong split tree (QASST), which will be our main tool of analysis.
Section~\ref{sect:distance_hereditary_QASST} describes the properties of \textit{distance-hereditary} graphs in terms of the QASST.

In Sections~\ref{sect:LC_equivalence_classes_with_the_QASST} and ~\ref{sect:explicit_formulas_for_local_equivalence_classes}, we show how to establish upper and lower bounds on the size of LC orbits for distance-hereditary graphs, and we prove that these bounds coincide in certain important families.
We work out the case of complete bipartite graphs in detail, both because they are mathematically interesting and because they illustrate our enumeration method clearly. 
We also present general formulas for the orbit sizes of more complicated families, including complete $k$-partite graphs, clique-stars, and multi-leaf repeater graphs.
The enumeration details in these cases involve many combinatorial arguments, so we provide a summary in the main text and leave the full details to the Appendix.

These constitute the principal contributions of the paper.
Section~\ref{sect:conclusion} concludes with a brief summary of our results.

\section{Summary of Special Graph Classes}
\label{sect:background}

A basic review of some of the core concepts and notations in graph theory is included in \ref{app:graph_theory_fundamentals}.
We will be primarily interested in the study of several special families of distance-hereditary graphs.
A graph is called \textit{distance-hereditary}~\cite{howorka1977characterization} provided the connected induced subgraphs preserve distance between vertices.

In this section, we define these special graphs of interest and introduce a choice of notation to represent each.
An example illustrating each these types of graphs is shown in Figure~\ref{fig:Graph_Classes}.

\begin{figure*}[t]
\centering
{\small
\begin{tabular}{ccccccccc}
\includegraphics[page=1,scale=0.42]{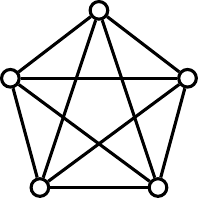}&\includegraphics[page=2,scale=0.42]{Figures/Special_Graph_Classes.pdf}&\includegraphics[page=3,scale=0.42]{Figures/Special_Graph_Classes.pdf}&\includegraphics[page=4,scale=0.42]{Figures/Special_Graph_Classes.pdf}&\includegraphics[page=5,scale=0.42]{Figures/Special_Graph_Classes.pdf}&\includegraphics[page=6,scale=0.42]{Figures/Special_Graph_Classes.pdf}&\includegraphics[page=7,scale=0.42]{Figures/Special_Graph_Classes.pdf}&\includegraphics[page=8,scale=0.42]{Figures/Special_Graph_Classes.pdf}&\includegraphics[page=9,scale=0.42]{Figures/Special_Graph_Classes.pdf}\\
$K_5$&$S_5$&$P_6$&$C_5$&$K_{3,2}$&$K_{2,2,2}$&$CS^1_{2,2,2}$&$R_3$&$MR_{4,3,2}$
\end{tabular}
}
\caption{Examples of special classes of graphs, all distance-hereditary except for the cycle.}
\label{fig:Graph_Classes}
\end{figure*}

\begin{itemize}
\item \textbf{Complete graph} $K_n$: the graph consisting of $n$ vertices and all possible edges between every pair of nodes. Each vertex has degree $n-1$. The total number of edges is $\frac{n(n-1)}{2}$.
\item \textbf{Star graph} $S_{n}$: the acyclic graph with $n+1$ vertices consisting of a single central vertex (the \textit{center} of $S_{n}$) and $n$ edges connecting to the remaining $n$ vertices to the center (the \textit{spokes} of the star). The center has degree $n$ and each spoke has degree 1.
\item \textbf{Path graph} $P_n$: the acyclic graph with $n$ vertices $V(P_n)=\{v_1,\cdots,v_n\}$ and $n-1$ edges $E(P_n)=\{(v_1,v_2),\cdots,(v_{n-1},v_n)\}$ corresponding to a path of length $n-1$ from $v_1$ to $v_n$. All vertices have degree 2 except for the degree 1 endpoints.
\item \textbf{Cycle graph} $C_n$: the graph with $n$ vertices $V(C_n)=\{v_1,\cdots,v_n\}$ and $n$ edges $E(C_n)=\{(v_1,v_2),\cdots,(v_{n-1},v_n),(v_n,v_1)\}$ corresponding to a cycle beginning and ending at $v_1$. All vertices have degree 2. 
\item \textbf{Complete bipartite graph} $K_{n,m}$: the bipartite graph $G$ containing $n+m$ vertices, where $V(K_{n,m})=U\sqcup V$ with $|U|=n$ and $|V|=m$, and there exist all possible edges between the vertices in $U$ and the vertices in $V$: $E(K_{n,m})=\{(v_i,v_j):v_i\in U,v_j\in V\}$.
Note that there are exactly $|U||V|=nm$ edges in $G$. The vertices in $U$ each have degree $m$ and the vertices in $V$ each have degree $n$.
\item \textbf{Complete $k$-partite graph} $K_{n_1,\cdots,n_k}$: the $k$-partite graph $K_{n_1,\cdots,n_k}$ containing $n_1+\cdots+n_k$ vertices, where $V(K_{n_1,\cdots,n_k})=U_1\sqcup\cdots\sqcup U_k$ with $|U_i|=n_i$, and there exist all possible edges between the vertices in $U_i$ and the vertices in $U_{j\neq i}$: $E(K_{n_1,\cdots,n_k})=\{(v,v'):v\in U_i, v'\in U_{j\neq i}\}$.
Note that there are exactly $\sum_{i=1}^{k-1}n_i\left(\sum_{j=i+1}^k n_j\right)$ edges in $K_{n_1,\cdots,n_k}$.
Each vertex $v\in U_i$ has degree $\text{deg}(v)=-n_i+\sum_{j=1}^kn_j$.
\item \textbf{Clique-Star} $CS_{n_1,\cdots,n_k}^r$: the graph consisting of $n_1+\cdots+n_k$ vertices total that can be partitioned into subsets $V(CS^r_{n_1,\cdots,n_k})=U_1\sqcup\cdots\sqcup U_k$ where $|U_i|=n_i$, similar to a $k$-partite graph. Each $U_i$ defines a clique in $CS_{n_1,\cdots,n_k}^r$. For some index $1\leq r\leq k$, the clique $U_r$ is the center of the clique-star, meaning that all vertices in $U_{i\neq r}$ are adjacent to all vertices in $U_r$. There exist no edges between vertices from different cliques $U_{i\neq r}$ and $U_{j\neq r}$.
There are $\sum_{i=1}^k\frac{n_i(n_i-1)}{2}+\sum_{i=1\neq r}^kn_in_r$ edges in total.
Each vertex $v\in U_{i\neq r}$ has degree $\text{deg}(v)=n_r+n_i-1$. Each vertex $u\in U_r$ has degree $\text{deg}(u)=-1+\sum_{i=1}^{k}n_i$.
\item \textbf{Repeater graph~\cite{azuma2015all}} $R_n$: the graph with $2n$ vertices consisting of the complete graph $K_n$ as a subgraph with the addition of a single additional vertex connected to each of the $n$ vertices in $K_n$. These additional vertices, each of degree 1, are referred to as the leaves. The non-leaf vertices each have degree $n$. The total number of edges is $\frac{n(n-1)}{2}+n$.
\item \textbf{Multi-leaf repeater graph} $MR_{n_1,\cdots,n_k}$: the graph with $n_1+\cdots+n_k$ vertices consisting of the complete graph $K_k$ as a subgraph with the addition of $n_i-1$ additional degree 1 vertices (leaves) attached to the $i^\text{th}$ vertex in the subgraph $K_k$. We require that $k\geq3$ and $n_i\geq2$. The $i^{th}$ non-leaf vertex in $K_k$ has degree $n_i+k-1$. The total number of edges is $\frac{k(k-1)}{2}+n_1+\cdots+n_k-k$.
\end{itemize}

Excluding cycle graphs with 5 or more vertices, each of the special graph classes enumerated above is an example of a distance-hereditary graph.
Cycles are included to provide a simple example of graphs which do not have this property.
We will reexamine graphs with the distance-hereditary property using split decompositions in Section~\ref{sect:split_decompositions}.

\section{Local Complements}
\label{sect:local_complements}

Local complement is an operation on graphs originally introduced by Bouchet ~\cite{bouchet1993recognizing,bouchet1991efficient,bouchet1988graphic}.
Among his contributions, Bouchet developed an algorithm to recognize when two graphs are considered equivalent with respect to this operation. He also used local complements to provide a characterization of those graphs equivalent to trees~\cite{bouchet1988transforming} as well as the class of \textit{circle graphs} (graphs defined via the intersection of chords in a circle, also known as \textit{alternance graphs})~\cite{bouchet1994circle,bouchet1987reducing}.
Following Bouchet's work, we will define the operation of local complement and introduce our basic notation that will be used throughout this paper.

\subsection{Local Complement as a Graph Operation}
\label{sect:graph_operation}

\textbf{Local complement} (LC) is an operation which transforms a graph in the following way. For a graph $G$ and any vertex $v\in V(G)$, the \textbf{local complement of $G$ with respect to $v$} is defined to be the graph $G'$ obtained from $G$ by replacing the subgraph $N_G(v)$ (the neighborhood of $v$ in $G$) with its \textbf{complementary} neighborhood $N_G^C(v)$ (Figure~\ref{fig:LC_example}).
Here, we define the complementary neighborhood of $v$ to have the same vertex set as $N_G(v)$, but complement the set of edges: for all $v_i,v_j\in V(N_G^C(v))=V(N_G(v))$, $(v_i,v_j)\in E(N_G^C(v))$ if and only if $(v_i,v_j)\notin E(N_G(v))$.

\begin{figure}
\centering
\includegraphics[page=1,width=0.7\linewidth]{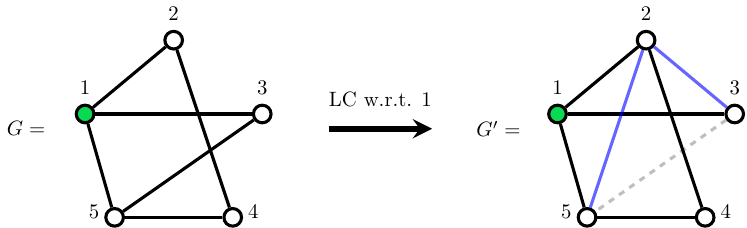}
\caption{Graphs related via local complement on vertex $1$. The edges between the vertices in $V(N_G(1))=\{2,3,5\}=V(N_{G'}(1))$ are deleted, while missing edges are added.}
\label{fig:LC_example}
\end{figure}

If $G'$ is the graph obtained from $G$ via local complement with respect to a vertex $v\in V(G)$, then $G$ and $G'$ have the same vertex set. 
Note that $v$ has the same set of neighboring vertices in both $G$ and $G'$.
Furthermore, $G'$ preserves all edges in $G$ which do not connect two neighbors of $v$.
This operation is self-inverse since $N_G^{CC}(v)=N_G(v)$: the local complement of $G'$ with respect to $v\in V(G')=V(G)$ is again the original graph $G$. The graphs $G$ and $G'$ are said to be \textbf{locally equivalent}, usually shortened to \textbf{LC equivalent}.

\subsection{Local Complement as a Function}
\label{sect:LC_function}

Local complement can be formalized as a function on graphs. Let ${\mathcal G}_n$ denote the set of simple graphs on $n$ vertices (not necessarily connected). For any graph $G\in{\mathcal G}_n$, we will denote the vertices using positive integers $V(G)=\{1,2,\cdots, n\}$, and all graphs in ${\mathcal G}_n$ will be said to have the same vertex set. For any vertex $i\in\{1,2,\cdots,n\}$, define the map $c_i:{\mathcal G}_n\rightarrow{\mathcal G}_n$ such that $c_i$ maps a graph $G\in{\mathcal G}_n$ to its local complement with respect to vertex $i$. Hence, the observation that local complement is self-inverse can be restated as $c_i^{-1}=c_i$.

With this description, it can be seen that the set of sequences of zero or more local complements on graphs in ${\mathcal G}_n$ forms a group under function composition. Let ${\mathcal L}_n$ denote this group, observing that it is generated by the \textbf{primitive local complements} ${\mathcal L}_n\coloneq\langle c_1,c_2,\cdots,c_n\rangle$. The identity map $\text{id}\in{\mathcal L}_n$ is regarded as an empty sequence. Given any sequence of local complements $f=c_{i_k}\circ c_{i_{k-1}}\circ\cdots\circ c_{i_2}\circ c_{i_1}\in{\mathcal L}_n$, where the indices $i_1,\cdots i_k\in\{1,\cdots,n\}$ need not be distinct, the inverse function is obtained by reversing the order of the sequence: $f^{-1}=c_{i_1}\circ c_{i_2}\circ\cdots\circ c_{i_{k-1}}\circ c_{i_k}\in{\mathcal L}_n$.

With this characterization in mind, we extend the definition of local equivalence to say that any two graphs $G,G'\in{\mathcal G}_n$ which can be transformed into each other via some sequence of local complements $f\in{\mathcal L}_n$ are \textbf{locally equivalent}. This is an \textit{equivalence relation} (i.e.~it is reflexive, symmetric, and transitive), and we will denote it by $G\cong_{\text{LC}}G'$.

Although graphs in ${\mathcal G}_n$ need not be connected in general, we will primarily be interested in the study of connected graphs. In particular, local complements preserve connectivity. This can be verified by comparing the complementary neighborhoods of any vertex $i\in V(G)$ for a graph $G\in{\mathcal G}_n$.
Every neighbor of $i$ in $G$ remains adjacent to $i$ in $c_i(G)$, and so all vertices in the same connected component as $i$ remain connected after taking the local complement. In other words, $c_i(G)$ is connected if and only if $G$ is.

\subsection{Local Complement as a Group Action}
\label{sect:group_action}

The relationship between the group ${\mathcal L}_n$ and the set ${\mathcal G}_n$ is that of a \textit{group action}, wherein ${\mathcal L}_n$ is said to \textit{act} on ${\mathcal G}_n$.
More generally, a \textbf{group action} $G\times S\rightarrow S$ consists of a group $G$ and a set $S$. Each group element $g\in G$, thought of as a map $g:S\rightarrow S$, is said to \textit{act on} each set element $s\in S$, with the image denoted $g(s)\in S$. For all $g_1,g_2\in G$ and $s\in S$, we have that $g_1g_2(s)=g_1(g_2(s))$. In other words, the binary operation in $G$ can be regarded as function composition, and the elements of $G$ thought of as functions from $S$ to $S$.

With these basic definitions, the group ${\mathcal L}_n$ (sequences of compositions of zero or more primitive local complements) and the set ${\mathcal G}_n$ (simple graphs on $n$ vertices) satisfy the definition of a group action. This also implies the existence of additional group theoretic structure, such as that of \textit{orbits} and \textit{stabilizers} (not to be confused with stabilizers in the context of quantum information theory).
For a group action $G\times S\rightarrow S$ and any set element $s\in S$, the \textbf{orbit} of $s$ is the set of images of $s$ under all group elements: $G(s)=\{g(s):g\in G\}\subseteq S$.
Contrastingly, the \textbf{stabilizer} of $s$ is the subset of group elements which leave $s$ fixed: $\text{Stab}_G(s)=\{g\in G:g(s)=s\}\leq G$.
Although we will not explore stabilizers in this work, the orbit of a graph under LC operations is our central topic of research.

\subsection{LC Equivalence Classes and LC Orbits}
\label{sect:LC_equivalence_classes_and_orbits}

We commented in Subsection~\ref{sect:LC_function} that \textit{local equivalence} is an \textit{equivalence relation}.
This is because the relation defined by $\cong_{\text{LC}}$ is \textit{reflexive} (since $\text{id}\in{\mathcal L}_n$, any $G\in{\mathcal G}_n$ is locally equivalent to itself), \textit{symmetric} ($G'=f(G)$ if and only if $G=f^{-1}(G')$ for any $f\in{\mathcal L}_n$), and \textit{transitive} ($G'=f(G)$ and $G''=g(G')$ imply that $g\circ f\in{\mathcal L}_n$ with $G''=g\circ f(G)$).
In particular, this means that we may study the \textit{equivalence classes} of locally equivalent graphs.
For a fixed number of vertices $n$, the set of simple graphs ${\mathcal G}_n$ can be partitioned into disjoint \textbf{LC equivalence classes}, wherein all graphs in the same equivalence class are locally equivalent.

In the context of group actions, these equivalence classes correspond exactly to the \textit{orbits} of the group action.
Given $G\in{\mathcal G}_n$, the \textbf{LC orbit} of $G$ under ${\mathcal L}_n$ is defined to be the set of images of $G$ under all possible functions in ${\mathcal L}_n$. We will usually denote this ${\mathcal O}(G)$, or in some cases ${\mathcal O}_{{\mathcal L}_n}^{{\mathcal G}_n}(G)$ when we wish to make the group ${\mathcal L}_n$ and the set ${\mathcal G}_n$ containing $G$ explicit:
$${\mathcal O}_{{\mathcal L}_n}^{{\mathcal G}_n}(G)=\{f(G):f\in{\mathcal L}_n\}\subseteq{\mathcal G}_n.$$
In other words, given any $G,G'\in{\mathcal O}(G)$, there exists a sequence of local complements $f\in{\mathcal L}_n$ such that $f(G)=G'$. By symmetry, we know that $G\in{\mathcal O}(G')$ as well.
More generally, we have that ${\mathcal O}(G')={\mathcal O}(G)$ for any $G'\in{\mathcal O}(G)$.

The LC orbit of any graph $G\in{\mathcal G}_n$ is precisely its LC equivalence class, and the collection of disjoint LC orbits under the action of ${\mathcal L}_n$ forms a partition of ${\mathcal G}_n$. 
In general, we write ${\mathcal O}(G)$ when we wish to refer to either.
As a simple example, Figure~\ref{fig:example_K4_orbit} shows the LC orbit for the complete graph on four (labeled) vertices.

\begin{figure}[t]
\centering
\includegraphics[width=1\linewidth,page=3]{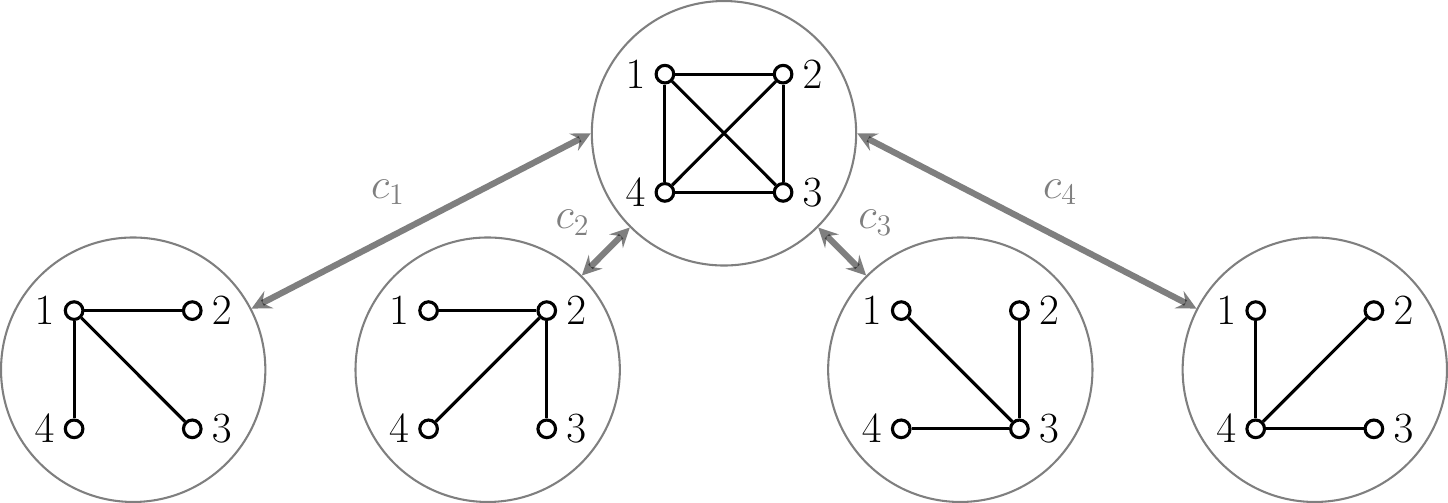}
\caption{The LC equivalence class of the complete graph on four vertices, the LC orbit ${\mathcal O}(K_4)$.}
\label{fig:example_K4_orbit}
\end{figure}

\subsection{LC Invariants}
\label{sect:LC_invariants}

Graph properties which are preserved under the operation of local complement are called \textbf{LC invariants}. 
The \textit{order} of a graph (i.e.~the number of vertices) is a simple example of an LC invariant, whereas the number of edges is not.
For a more interesting example, the property of being distance-hereditary is another feature preserved by LC operations.
In other words, if one graph in the LC orbit is determined to be distance-hereditary, then all graphs in this equivalence class are distance-hereditary.
In this paper, we will be principally interested in working with distance-heredity along with another LC invariant known as the \textit{split decomposition} of a graph (detailed in Section~\ref{sect:split_decompositions}).

Invariants are useful tools for distinguishing between LC orbits of non-equivalent graphs.
In his work with locally equivalent graphs, Bouchet proved a number of interesting LC invariants, including distance-heredity (equivalently known as total decomposability)~\cite{bouchet1988transforming} and the split decomposition~\cite{bouchet1987reducing,bouchet1987digraph,bouchet1989connectivity}.
One of Bouchet's key results was the recognition that the relationship between locally equivalent graphs could be restated using another mathematical structure known as \textit{isotropic systems}~\cite{bouchet1988graphic}. Based on this idea, Bouchet showed how to construct a set of equations given any two graphs defined on the same vertex set; the graphs are locally equivalent if and only if the equations have a solution.

In contrast to this, we approach the question of enumerating a few special families of locally equivalent graphs by first examining those graphs which have the same split decomposition. By exhaustively examining these families of graphs with the same splits, we show how these can be partitioned into disjoint LC orbits, and thus obtain explicit formulas for the size of the local equivalence class in these cases.

\section{Split Decompositions}
\label{sect:split_decompositions}

The \textbf{split decomposition} (also referred to as the \textbf{join decomposition}) of a graph $G$ is a recursive method to examine the structure of a graph originally introduced by Cunningham~\cite{cunningham1982decomposition,cunningham1980combinatorial}.
The technique is based on identifying complete bipartite subgraphs and then using these to separate the graph into components.
Although originally defined for directed graphs, this process generalizes naturally to undirected graphs as well~\cite{ma1990split,ma1994n2}.
We only consider the undirected case in this paper.

Among its applications, the split decomposition has been used to recognize certain graphs classes. These include, for example, circle graphs~\cite{bouchet1987reducing,gabor1989recognizing,spinrad1994recognition} and distance-hereditary graphs~\cite{gavoille2003distance,gioan2007dynamic,gioan2012split}.
Furthermore, there exist efficient algorithms to compute the split decomposition of any simple, undirected graph~\cite{ma1994n2,dahlhaus1994efficient,dahlhaus2000parallel,charbit2012linear}.
In Section~\ref{sect:distance_hereditary_QASST}, we will be particularly interested in examining the splits within certain subfamilies of distance-hereditary graphs.

In this section, we begin by reviewing the basic concepts of splits, building up to the formal definition of the split decomposition for a graph.
We then introduce a new tool based on this: the QASST of a graph, which will be our primary technique for studying graphs throughout this paper.

\subsection{Splits and Strong Splits}
\label{sect:splits_and_strong_splits}

Given a graph $G$, a {\bf split} in $G$ is defined to be a bipartition of the vertex set $V(G)=U_1\cup U_2$ into disjoint, nonempty subsets such that the subgraph induced by the edges crossing between $U_1$ and $U_2$ is a \textit{complete bipartite subgraph}. A split is called \textbf{trivial} if either $U_1$ or $U_2$ consists of a single vertex, and \textbf{nontrivial} otherwise. A graph is called \textbf{prime} (with respect to splits) if it contains no nontrivial splits. The 5-cycle $C_5$ is the smallest example of a prime graph.

Two splits $V(G)=U_1\cup U_2$ and $V(G)=U_1'\cup U_2'$ of the same graph $G$ are said to \textbf{cross} if each side of one split has nonempty intersection with each side of the other split: $U_1\cap U_1'\neq\emptyset$; $U_1\cap U_2'\neq\emptyset$; $U_2\cap U_1'\neq\emptyset$; and $U_2\cap U_2'\neq\emptyset$. A split is called \textbf{strong} if it is not crossed by any other split.
By definition, trivial splits are always strong.
Intuitively, the edge-induced subgraph corresponding to a nontrivial strong split is \textit{maximal} in the sense that it is not properly contained in a larger complete bipartite subgraph.
Figure~\ref{fig:Splits_Examples} shows an example of these basic concepts concerning splits.

\begin{figure*}[t]
\centering
\begin{subfigure}{0.2\textwidth}
\centering
\includegraphics[page=1,width=0.8\linewidth]{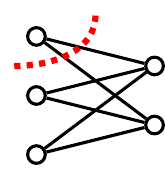}
\caption{Trivial split}
\label{fig:Splits_Examples_trivial_split}
\end{subfigure}\begin{subfigure}{0.4\textwidth}
\centering
\includegraphics[page=2,width=0.4\linewidth]{Figures/Section_Split_Decompositions_Figures.pdf}\includegraphics[page=3,width=0.4\linewidth]{Figures/Section_Split_Decompositions_Figures.pdf}
\caption{Two crossing nontrivial splits}
\label{fig:Splits_Examples_crossing_splits}
\end{subfigure}\begin{subfigure}{0.2\textwidth}
\centering
\includegraphics[page=4,width=0.8\linewidth]{Figures/Section_Split_Decompositions_Figures.pdf}
\caption{Strong split}
\label{fig:Splits_Examples_strong_split}
\end{subfigure}\begin{subfigure}{0.2\textwidth}
\centering
\includegraphics[page=5,width=0.8\linewidth]{Figures/Section_Split_Decompositions_Figures.pdf}
\caption{Prime graph}
\label{fig:Splits_Examples_prime_graph}
\end{subfigure}

\caption{Examples showing the basic concepts of splits in a graph, indicated by the red dashed line.}
\label{fig:Splits_Examples}
\end{figure*}

The split decomposition of a graph $G$, as introduced by Cunningham~\cite{cunningham1982decomposition,cunningham1980combinatorial}, refers to the recursive process by which a graph can be divided into components by finding splits.
This is not unique in general, but it is when restricting to strong splits.
From this perspective, the \textit{strong split decomposition} uses a \textit{minimal} number of \textit{maximal} splits.
In the special case of undirected graphs, Cunningham showed that any connected graph has a unique minimal decomposition whose components are \textit{prime}, \textit{complete}, or \textit{star} (Theorem 3 of~\cite{cunningham1982decomposition}). 

The decomposition process by which graphs are split recursively into components is naturally described by a tree, an idea which Gioan and Paul formalized using \textit{graph-labeled trees} in the particular case of distance-hereditary graphs~\cite{gioan2007dynamic,gioan2012split}.
A graph-labeled tree refers to an acyclic graph whose nodes are themselves labeled by a set of graphs, in this case the components of the split decomposition.
With some additional modifications, we follow the same basic approach as Gioan and Paul to build up a tree-like representation of the structural information encoded in the split decomposition.

\subsection{Strong Split Trees and Split Components}
\label{sect:split_trees_and_components}

In this paper, the \textbf{split decomposition} of a graph $G$ refers to a representation of the strong splits in $G$ with a tree $SST(G)$, which we will refer to as the \textbf{strong split tree} of $G$. The strong splits in $G$ are in bijection with the edges in $SST(G)$. The vertices in $G$ are in bijection with the leaves (degree 1 nodes) in $SST(G)$; an edge in $SST(G)$ incident with a leaf denotes the trivial split defined by the corresponding vertex in $G$. The strong split tree is most easily understood via an example, as shown in Figure~\ref{fig:Strong_Split_Tree_example}.

\begin{figure*}[t!]
\centering
\begin{subfigure}{0.65\textwidth}
\centering
\includegraphics[page=6,width=0.7\linewidth]{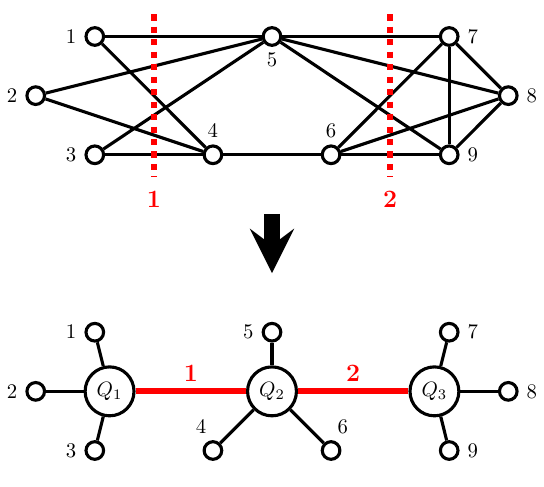}
\caption{A graph $G$, its strong splits, and its strong split tree $SST(G)$.}
\label{fig:Strong_Split_Tree_example_G_and_SST}
\end{subfigure}\begin{subfigure}{0.35\textwidth}
\centering
\includegraphics[page=7,width=0.7\linewidth]{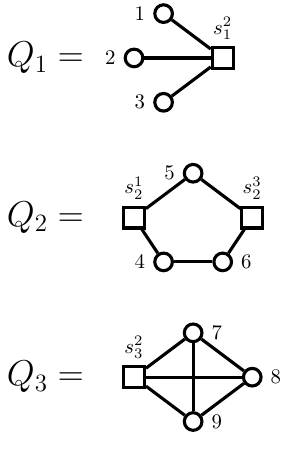}
\caption{Quotient graphs $Q_1$, $Q_2$, and $Q_3$.}
\label{fig:Strong_Split_Tree_example_Quotient_Graphs}
\end{subfigure}

\bigskip

\begin{subfigure}{1\textwidth}
\centering
\includegraphics[width=0.85\linewidth,page=8]{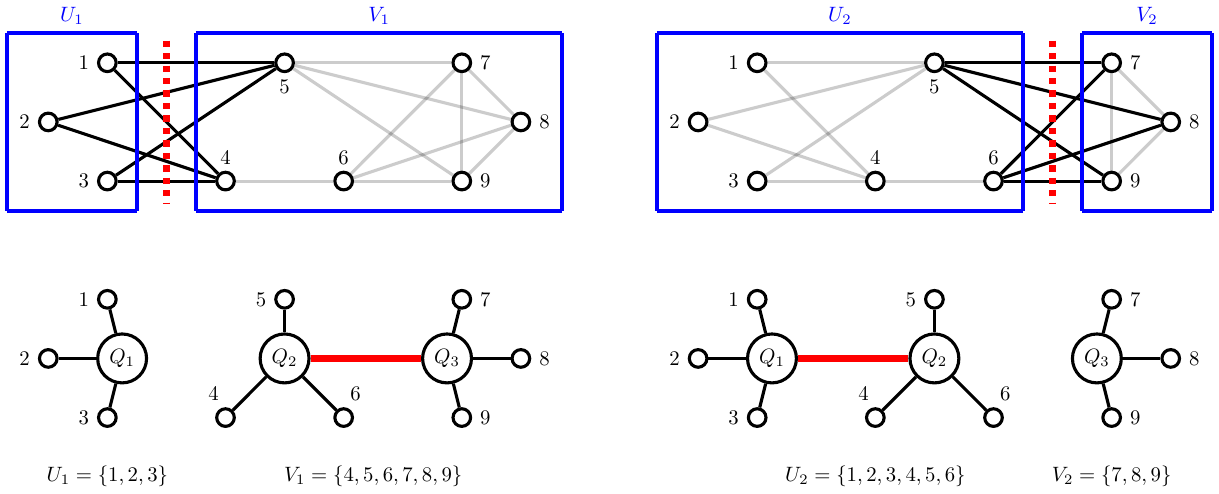}
\caption{The two induced strong splits $U_1\cup V_1$ and $U_2\cup V_2$ corresponding to the internal edges in the strong split tree.}
\label{fig:Strong_Split_Tree_example_induced_splits}
\end{subfigure}

\caption{Example of a graph with two strong splits\footnote{Graph based on the figure by David Eppstein for the Wikipedia article on split decompositions:
\url{https://en.wikipedia.org/wiki/Split_(graph_theory)\#/media/File:Split_decomposition.svg}.}.
\ref{fig:Strong_Split_Tree_example_G_and_SST} A graph $G$ with its strong splits marked by red dashed lines and the corresponding strong split tree $SST(G)$.
\ref{fig:Strong_Split_Tree_example_Quotient_Graphs} The quotient graphs corresponding to this split decomposition; squares denote split-nodes.
\ref{fig:Strong_Split_Tree_example_induced_splits} Visualization of the relationship between the two internal edges in $SST(G)$ and the the corresponding strong splits in $G$. Deleting an edge in the tree divides it into two parts; the corresponding bipartition of leaves in the tree defines the split in the graph. Notice that the subgraphs induced by the edges crossing between both halves of the splits are complete bipartite.
}
\label{fig:Strong_Split_Tree_example}
\end{figure*}

Deleting any edge in $SST(G)$ divides the tree into two connected components, each a subtree of $SST(G)$. The bipartition of the leaves in $SST(G)$ defined by these two subtrees induces a bipartition of the vertices in the original graph $G$. This bipartition is exactly the strong split in $G$ corresponding to the deleted edge in $SST(G)$.

In general, the strong split tree of a graph $G$ can be computed recursively by identifying the strong splits and building up a series of intermediate graphs by ``collapsing" the edges defining a given strong split into a pair of adjacent virtual nodes. 
To distinguish these new virtual nodes from the original vertices in $V(G)$, we will refer to these as \textit{split-nodes}, and they will always be represented in figures by squares.
Since strong splits do not cross other splits, deleting the edge between the split-nodes effectively divides the graph into smaller components while preserving the remaining strong splits, allowing the process to continue.
The procedure described herein is similar to the split decomposition construction used in~\cite{gavoille2003distance}.

When all nontrivial strong splits have been collapsed in this way, the resulting disconnected components represent exactly the internal nodes of $SST(G)$, with internal edges in the tree defined by the pairs of split-nodes. The trivial splits of $G$ are then accounted for by adding a leaf to $SST(G)$ for each vertex in these components matching a vertex from the original graph $G$. For this reason, we will refer to the non-split-nodes as \textit{leaf-nodes} in the context of the strong split tree, and these will always be represented in figures by circles. 

Although the process of collapsing strong splits as described above can be carried out in any order, Cunningham showed that the original graph $G$ uniquely determines the resulting tree, what we refer to here as $SST(G)$, and that this decomposition is minimal in the sense that it uses as few splits as possible~\cite{cunningham1982decomposition}.
Furthermore, the final intermediate graphs obtained after collapsing all strong splits are also unique and have only three possible types: \textit{complete}, \textit{star}, or \textit{prime} graphs.

We will refer to these as the \textbf{quotient graphs} (also called \textbf{split-components}) of $G$ and associate to each the corresponding internal node of $SST(G)$.
Figure~\ref{fig:Strong_Split_Tree_example_Quotient_Graphs} shows an example of quotient graphs, where the labels on the vertices are formally defined in the next subsection.
Note that this construction ensures that the quotient graphs of $SST(G)$ contain no nontrivial splits.
A quotient graph is not properly a subgraph of $G$ since it may contain vertices which are not in the original graph (specifically, the split-nodes).
However, each quotient graph is isomorphic to an induced subgraph of $G$.

\subsection{The Quotient-Augmented Strong Split Tree}
\label{sect:qasst}

For our purposes, the split decomposition of $G$ consists of the strong split tree $SST(G)$ together with the quotient graphs obtained from it.
It is convenient to combine this information into a single graph-labeled tree representing the entire split decomposition. 
We define the \textbf{quotient-augmented strong split tree} of a graph $G$, denoted $QASST(G)$\footnote{QASST is read like \textit{kwast}, the Dutch word for \textit{brush}, reminiscent of the branching structure of the tree.}, to be the graph-labeled tree obtained from $SST(G)$ by deleting the leaves and then labeling the remaining internal nodes with the corresponding quotient graphs.
The QASST is a slightly modified version of the tree introduced by Gioan and Paul~\cite{gioan2007dynamic,gioan2012split} which makes explicit the distinction between split-nodes and leaf-nodes and their connection to strong splits in the original graph.
Figure~\ref{fig:qasst_example} shows an example of the QASST obtained from the graph of Figure~\ref{fig:Strong_Split_Tree_example}.

\begin{figure*}[t]
\centering
\includegraphics[page=9,width=0.6\linewidth]{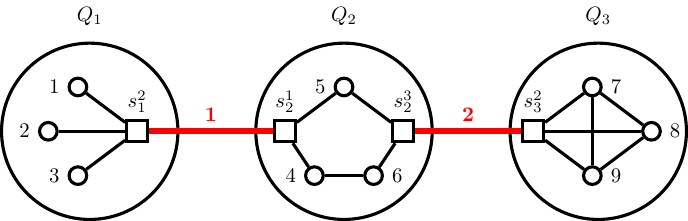}
\caption{The \textit{quotient-augmented strong split tree} (QASST) for the graph shown in Figure~\ref{fig:Strong_Split_Tree_example}. The QASST is obtained from the \textit{strong split tree} by deleting the leaves and labeling the remaining internal nodes by their corresponding \textit{quotient graphs}.
Edges between quotient graphs are in bijection with the strong splits in the original graph; these are marked in red to distinguish them from the edges within the quotient graphs.
Each of these red edges corresponds to a pair of \textit{split-nodes} (marked by squares) from two different quotient graphs, to which we connect them directly.
}
\label{fig:qasst_example}
\end{figure*}

We introduce the following notation to describe the components of the QASST.
The $n$ vertices of the graph $G\in{\mathcal G}_n$ are labeled by integers $\{1,\cdots,n\}$. The quotient graphs of $G$ are enumerated $Q_1,\cdots,Q_k$.
Recall that the leaves of $SST(G)$ are in bijection with the vertices of $G$, and the internal nodes are in bijection with the quotient graphs. Hence, matching this, the leaves of $SST(G)$ will be labeled $1,\cdots,n$ and the internal nodes will be labeled $Q_1,\cdots,Q_k$.
When convenient, we will also refer to the quotient graphs in $QASST(G)$ by $Q_1,\cdots,Q_k$.

Finally, consider the quotient graphs themselves, recalling that each $Q_i$ contains two types of nodes: leaf-nodes, corresponding to a subset of vertices from $G$, and split-nodes, corresponding to a nontrivial strong split of $G$ (and hence equivalently to an edge in $SST(G)$). Thus, we label the leaf-nodes matching the integer labels of the corresponding vertices of $G$, and we label the split-nodes based on the edges of $SST(G)$ according to the following rule.

Since split nodes come in pairs for adjacent quotient graphs in $SST(G)$, each edge $(Q_i,Q_j)\in E(SST(G))$ corresponds to two split-nodes (one in $Q_i$ and one in $Q_j$). As shown in Figure~\ref{fig:Strong_Split_Tree_example_Quotient_Graphs}, let $s_i^j$ denote the split-node in $Q_i$, and let $s_j^i$ denote the split-node in $Q_j$.
Hence, the subscript on the split-node matches the index of the quotient graph which contains it, and the superscript matches the index of its partner's quotient graph.
In this way, all vertices in all quotient graphs have distinct labels.
These are labeling conventions used for all of the split-nodes in Figures~\ref{fig:Strong_Split_Tree_example} and \ref{fig:qasst_example}.

\subsection{Computing the QASST from a Graph and Reconstructing a Graph from the QASST}
\label{sect:qasst_construction}

The recursive technique used to compute the strong split tree of a graph can be extended to compute the QASST. In the same way as described in Section~\ref{sect:split_trees_and_components}, after identifying the strong splits in a graph, the edges of the corresponding complete-bipartite subgraph can be collapsed one at a time into a pair of split-nodes, which we always indicate by squares in the figures.
These two split-nodes are joined by an edge, which we highlight in red to distinguish from edges incident to vertices from the original graph.
If a component obtained from collapsing in this way has no remaining strong splits, then this is exactly a quotient graph and we mark it as such.
This process is illustrated in Figure~\ref{fig:qasst_recursive_construction}.

\begin{figure*}[t]
\centering
\includegraphics[page=10,width=0.8\linewidth]{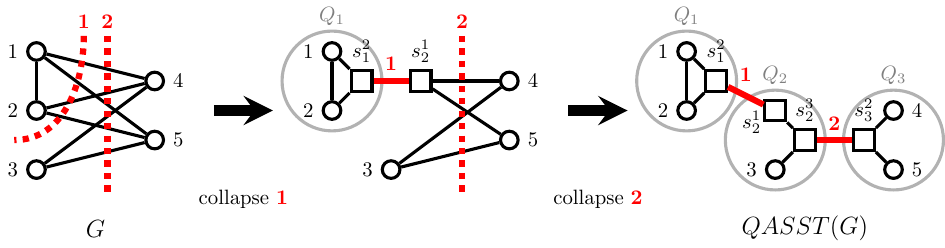}
\caption{Example of the recursive process to compute $QASST(G)$ by collapsing the edges corresponding to strong splits into pairs of split-nodes. Components with no remaining strong splits become quotient graphs.}
\label{fig:qasst_recursive_construction}
\end{figure*}

Given the quotient-augmented strong split tree of a graph, there is sufficient information to reconstruct the original graph. In other words, $G$ and $QASST(G)$ contain equivalent structural information. This is in contrast to $SST(G)$, which only contains information about the strong splits in $G$ when the quotient graphs are excluded.
We consider both the process of computing the QASST and the reconstruction of the original graph because, in later sections, we will use the QASST both to enumerate graphs and to find an optimal representive across the LC equivalence class.

The reconstruction of $G$ can be achieved by reversing the recursive construction of the QASST. This is done by iteratively merging adjacent quotient graphs in $QASST(G)$ until only a single component remains.
This is achieved by replacing pairs of split-nodes with full connections between their respective neighbors.
Split-nodes always come in pairs matching the strong splits from the original graph, and we mark these pairs explicitly in $QASST(G)$ by connecting them with red edges, as shown in the figures.
Keeping track of these pairs is necessary since a single quotient graph may have multiple split-nodes.
In this way, there is no ambiguity about how to reconnect nodes from different quotient graphs.

Figure~\ref{fig:reconstruct_graph_from_qasst} shows an example of a graph reconstructed from its QASST.
This mirrors the recursive process shown in Figure~\ref{fig:qasst_recursive_construction} in reverse.
Although we define it here for the QASST, note that this process can be used to reconstruct a graph for any graph-labeled split tree, even if the splits matching the edges are not strong. Pairs of split-nodes provide a convenient way to denote all-to-all connectivity between the split-nodes' respective neighbors. Although we use split-node pairs in this paper exclusively to represent strong splits, these could be used in a more general graph-labeled tree where a pair of split-nodes does not define a strong split.

\begin{figure*}[t]
\centering
\includegraphics[width=0.9\linewidth,page=11]{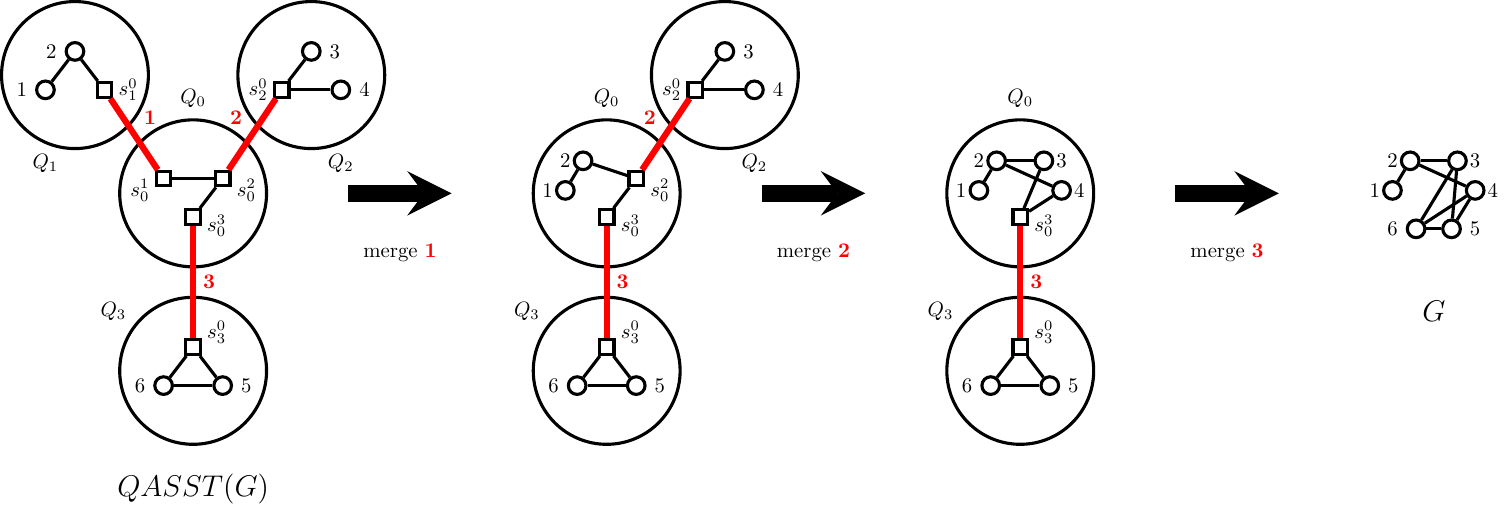}
\caption{Example showing how to reconstruct a graph from its QASST through a sequence of merging quotient graphs. Pairs of split-nodes from adjacent quotient graphs are replaced at each iteration with all-to-all connections between neighbors. This works for any split tree, even when the splits are not strong.}
\label{fig:reconstruct_graph_from_qasst}
\end{figure*}

\section{QASST Structure for Distance-Hereditary Graphs}
\label{sect:distance_hereditary_QASST}

Recall that distance-hereditary (DH) graphs are those graphs for which connected induced subgraphs preserve distance between nodes~\cite{howorka1977characterization}.
Here, we only consider connected graphs.
Some alternative and equivalent characterizations of DH graphs include: graphs admitting a recursive construction via twins~\cite{bandelt1986distance}; completely separable graphs~\cite{hammer1990completely}; totally decomposable graphs~\cite{cunningham1982decomposition}; and graphs of rank width 1\cite{oum2005rank}.
It will be useful to draw on these alternative characterizations for certain proofs, at which point we will elaborate these other definitions as needed.

Our goal in this section is to describe the relationship between a distance-hereditary graph and its quotient-augmented strong split tree.
In Subsection~\ref{sect:restrictions_on_quotient_graphs}, we identify restrictions on adjacencies between quotient graphs of different types.
We then describe how the QASST evolves under the recursive construction using one-vertex extensions in Subsection~\ref{sect:recursive_qasst_construction_for_DH_graphs}.
Finally, in Subsection~\ref{sect:qasst_evolution_for_induced_subgraphs} we discuss how the QASST of a DH graph changes under induced subgraphs.

\subsection{Restrictions on Quotient Graphs}
\label{sect:restrictions_on_quotient_graphs}

Cunningham established that the quotient graphs obtained via the split decomposition are either star, complete, or prime~\cite{cunningham1982decomposition}. However, when restricting to distance-hereditary graphs in particular, the quotient graphs are only star or complete (\cite{hammer1990completely}, Lemma 4.4 of~\cite{gavoille2003distance}). Indeed, this feature can be taken as an equivalent definition: a graph is distance-hereditary if and only if its split decomposition contains only star and complete quotient graphs.
A graph with this property is called \textbf{completely separable}.

Rather than starting with a graph and computing its split decomposition, we can imagine the reverse problem of beginning with a split decomposition and reconstructing the graph via the technique outlined in Subsection~\ref{sect:qasst_construction} and visualized in Figure~\ref{fig:reconstruct_graph_from_qasst}. Given any graph-labeled tree, where the nodes are labeled by some valid combination of quotient graphs, we can reconstruct a graph in this way. The quotient graphs are valid provided each contains the correct number of split-nodes (matching the degree of the node in the tree), and pairs of split-nodes from adjacent quotient graphs are properly identified, as in Figure~\ref{fig:reconstruct_graph_from_qasst}.

\begin{table*}[t]
\centering
\scalebox{0.9}{
\begin{tabular}{|c|ccc|}
\hline
$(Q_1,Q_2)$&$Q_2$ is sc&$Q_2$ is ss&$Q_2$ is c\\
\hline
&&&\\
$Q_1$ is sc&\includegraphics[width=0.2\linewidth,valign=c,page=1]{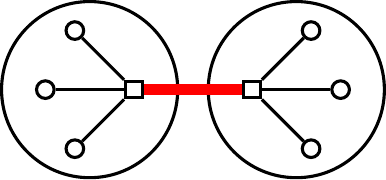}&\includegraphics[width=0.2\linewidth,valign=c,page=2]{Figures/Section_QASST_Structure_for_DH_Graphs.pdf}&\includegraphics[width=0.2\linewidth,valign=c,page=3]{Figures/Section_QASST_Structure_for_DH_Graphs.pdf}\\
&sc-sc&sc-ss&sc-c\\
&&&\\
$Q_1$ is ss&\includegraphics[width=0.2\linewidth,valign=c,page=4]{Figures/Section_QASST_Structure_for_DH_Graphs.pdf}&\includegraphics[width=0.2\linewidth,valign=c,page=5]{Figures/Section_QASST_Structure_for_DH_Graphs.pdf}&\includegraphics[width=0.2\linewidth,valign=c,page=6]{Figures/Section_QASST_Structure_for_DH_Graphs.pdf}\\
&ss-sc&ss-ss&ss-c\\
&&&\\
$Q_1$ is c&\includegraphics[width=0.2\linewidth,valign=c,page=7]{Figures/Section_QASST_Structure_for_DH_Graphs.pdf}&\includegraphics[width=0.2\linewidth,valign=c,page=8]{Figures/Section_QASST_Structure_for_DH_Graphs.pdf}&\includegraphics[width=0.2\linewidth,valign=c,page=9]{Figures/Section_QASST_Structure_for_DH_Graphs.pdf}\\
&c-sc&c-ss&c-c\\
\hline
\end{tabular}
}
\caption{The nine possibilities for two adjacent star or complete quotient graphs $Q_1$ and $Q_2$ in a split tree. There are three possibilities for each quotient graph depending on the location of the split-node involved in this split relative to the other nodes in the quotient graph. If the quotient graph is star, then the split node is either the center of the star or else it is a spoke; there is only one case if the quotient graph is complete. These cases are referred to as \textit{star-center} (sc), \textit{star-spoke} (ss), and \textit{complete} (c), respectively. By symmetry, of the nine cases shown above, six of these are distinct and the remaining three are redundant.}
\label{tab:adjacent_star_complete_quotient_graphs}
\end{table*}

We would like to characterize those graphs which can be obtained via this reconstruction when restricting to star and complete quotient graphs. Although the edges in the graph-labeled tree always represent splits, depending on the quotient graphs, these may not always be strong splits. To see this, consider the examples shown in Table~\ref{tab:adjacent_star_complete_quotient_graphs}, which illustrates the nine possible ways that two adjacent star or complete quotient graphs could be connected in a split tree (by symmetry, only six of these are distinct).
Although we show only two quotient graphs here, these possibilities generalize to a split tree of any size; what matters is the relationship between the pair of split-nodes matching the edge in the split tree.
Here, we make a distinction between two types of star quotient graphs: \textit{star-center}, in which the split-node is the center of the star, and \textit{star-spoke}, in which the split-node is not the center.

\begin{figure*}[t]
\centering
\begin{subfigure}{0.16\textwidth}
\centering
\includegraphics[width=0.9\textwidth,page=10]{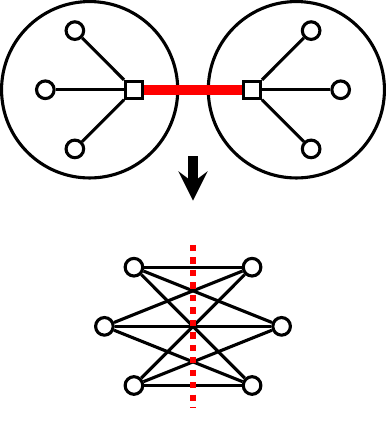}
\caption{sc-sc}
\label{fig:quotient_combo_valid_sc_sc}
\end{subfigure}\begin{subfigure}{0.16\textwidth}
\centering
\includegraphics[width=0.9\textwidth,page=11]{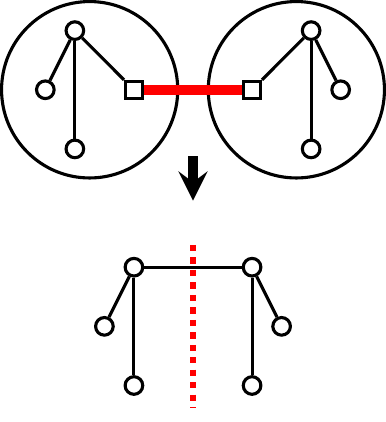}
\caption{ss-ss}
\label{fig:quotient_combo_valid_ss_ss}
\end{subfigure}\begin{subfigure}{0.16\textwidth}
\centering
\includegraphics[width=0.9\textwidth,page=12]{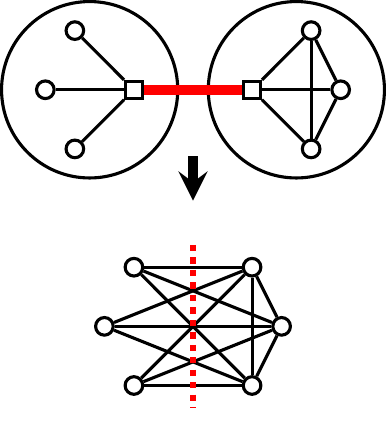}
\caption{sc-c or c-sc}
\label{fig:quotient_combo_valid_sc_c}
\end{subfigure}\begin{subfigure}{0.16\textwidth}
\centering
\includegraphics[width=0.9\textwidth,page=13]{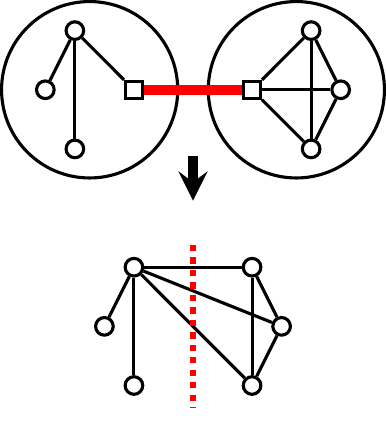}
\caption{ss-c or c-ss}
\label{fig:quotient_combo_valid_ss_c}
\end{subfigure}\begin{subfigure}{0.16\textwidth}
\centering
\includegraphics[width=0.9\textwidth,page=14]{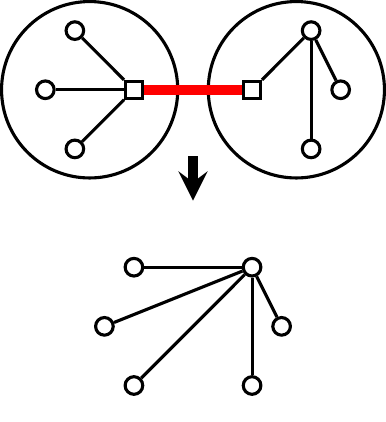}
\caption{sc-ss or ss-sc}
\label{fig:quotient_combo_invalid_ss_sc}
\end{subfigure}\begin{subfigure}{0.16\textwidth}
\centering
\includegraphics[width=0.9\textwidth,page=15]{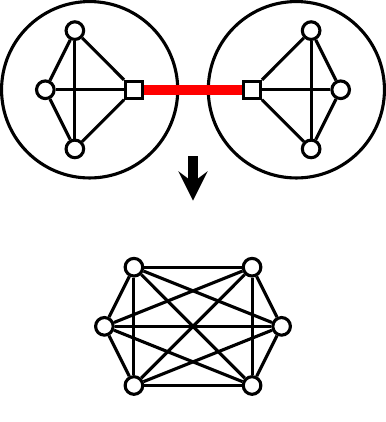}
\caption{c-c}
\label{fig:quotient_combo_invalid_c_c}
\end{subfigure}
\caption{Examples of the reconstructed graphs from two adjacent star or complete quotient graphs for each of the six distinct cases outlined in Table~\ref{tab:adjacent_star_complete_quotient_graphs}.
Cases~\ref{fig:quotient_combo_valid_sc_sc}--\ref{fig:quotient_combo_valid_ss_c} correspond to strong splits in the graph.
These are valid and pairs of adjacent quotient graphs matching these four cases may occur in the split decomposition of a graph.
However, cases~\ref{fig:quotient_combo_invalid_ss_sc} and \ref{fig:quotient_combo_invalid_c_c} correspond to non-strong splits of a star or complete graph, respectively.
Hence, these two cases are invalid and cannot occur in the split decomposition.
}
\label{fig:quotient_graph_combinations_valid}
\end{figure*}

Of the six distinct possibilities illustrated in Table~\ref{tab:adjacent_star_complete_quotient_graphs}, all but two of these represent a strong split, as can be seen in Figure~\ref{fig:quotient_graph_combinations_valid}. The remaining two represent valid splits, but they are not strong. 
Two adjacent complete quotient graphs yield a larger complete graph (Figure~\ref{fig:quotient_combo_invalid_c_c}), while the combination of a star-spoke and a star-center is a larger star graph (Figure~\ref{fig:quotient_combo_invalid_ss_sc}).
Since complete and star graphs are indecomposable under strong splits, these two possibilities cannot occur in the split decomposition of any graph.
However, the remaining four cases of Figure~\ref{fig:quotient_graph_combinations_valid} preserve strong splits and hence match the split decomposition of the reconstructed graph.

What this shows is that, for any split tree with star and complete quotient graphs, the splits are strong provided each edge in the tree corresponds to one of the four valid cases shown in Figure~\ref{fig:quotient_graph_combinations_valid}. In other words, such a split tree and quotient graphs represent the split decomposition of a graph, and this graph is totally decomposable by definition and hence necessarily distance-hereditary.
We will exploit this relationship in Section~\ref{sect:explicit_formulas_for_local_equivalence_classes} to enumerate those distance-hereditary graphs which are equivalent under local complements.

\subsection{Recursive QASST Construction for distance-hereditary Graphs}
\label{sect:recursive_qasst_construction_for_DH_graphs}

An alternative characterization for distance-hereditary graphs is the class of connected graphs which can be constructed recursively from a single vertex using a sequence of \textbf{one-vertex extensions} via twins, of which there are three possible operations (Theorem 1 of~\cite{bandelt1986distance}):
\begin{enumerate}
\item addition of a \textbf{pendant vertex} $p$ (connect to just a single existing vertex $q$);
\item addition of a \textbf{false twin} $p$ (connect to only the neighbors of an existing vertex $q$);
\item addition of a \textbf{true twin} $p$ (connect to both an existing vertex $q$ and its neighbors).
\end{enumerate}
In all cases, the existing vertex $q$ and its one-vertex extension $p$ are called \textit{twins}.
An example of these three one-vertex extensions is shown in Figure~\ref{fig:one-vertex_extensions}. Note that the construction of a distance-hereditary graph described by these operations is non-unique in general; there may be many equivalent sequences of one-vertex extensions which yield the same graph.

\begin{figure*}[t]
\centering
\begin{subfigure}{0.5\textwidth}
\centering
\includegraphics[width=1\textwidth,page=16]{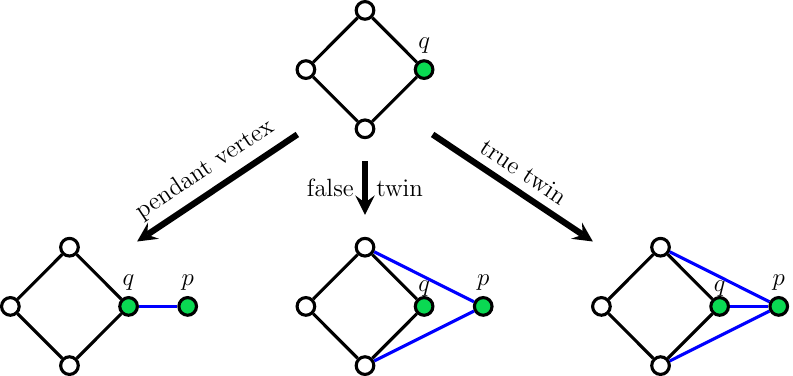}
\caption{The three types of one-vertex extensions.}
\label{fig:one-vertex_extensions}
\end{subfigure}\begin{subfigure}{0.5\textwidth}
\centering
{\scriptsize
\begin{tabular}{ccc}
\includegraphics[width=0.23\textwidth,page=18]{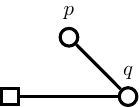}&\includegraphics[width=0.23\textwidth,page=19]{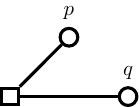}&\includegraphics[width=0.23\textwidth,page=20]{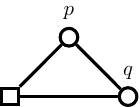}\\
star-spoke&star-center&complete\\
\\ \\
\end{tabular}
}
\caption{The three cases for $Q_2'$ in Lemma~\ref{thm:one-vertex_quotient_evolution_lemma}.}
\label{fig:one_vertex_evolution_Q2_possibilities}
\end{subfigure}
\caption{
\ref{fig:one-vertex_extensions} The three types of one-vertex extension $p$ of an existing node $q$; $p$ is called a \textit{twin} of $q$. Any distance-hereditary graph can be constructed recursively from a single vertex using only these operations.
\ref{fig:one_vertex_evolution_Q2_possibilities} The three possible cases for the evolution of a quotient graph after a one-vertex extension. When a new split is introduced, the extended quotient graph divides into two parts, matching one of these possibilities.
}
\label{fig:one-vertex_extension_possibilities}
\end{figure*}

The recursive construction used to define a given distance-hereditary graph can also be used to compute its QASST. If $G$ is a distance-hereditary graph and $G'$ is obtained from $G$ via one of the three one-vertex extensions described above, then $QASST(G')$ is obtained from a slight modification of $QASST(G)$. By showing how the QASST changes under these operations, we provide a general characterization of the QASST of any distance-hereditary graph starting from a single vertex.
First, however, we formalize the evolution of the QASST via the following lemma.

\begin{lemma}
\label{thm:one-vertex_quotient_evolution_lemma}
Let $G$ be a graph with a vertex $q\in V(G)$. Let $Q$ be the quotient graph containing $q$ in the split decomposition of $G$.
If $G'$ is a graph obtained from $G$ by a one-vertex extension $p$ of $q$, then $QASST(G')$ is computed from $QASST(G)$ by converting $Q$ into $Q'$ according to a number of cases listed below.

Either $Q'$ is obtained from $Q$ by using the same vertex extension $p$ of $q$ within the quotient graph $Q$, or else $Q$ splits into a disjoint union of two quotient graphs $Q'=Q_1'\cup Q_2'$ which are connected by an edge between split-nodes in the QASST. $Q_1'$ is obtained from $Q$ by replacing $q$ with a split-node. $Q_2'$ is a graph with three vertices: $q$, $p$, and a split-node (Figure~\ref{fig:one_vertex_evolution_Q2_possibilities}). This pair of split-nodes in $Q_1'$ and $Q_2'$ defines an edge in $QASST(G')$.
Depending on the structure of $Q$ and type of vertex extension, there are twelve sub-cases in total, all summarized in Table~\ref{tab:one-vertex_extension_new_quotient_graph}.
\begin{enumerate}
\item Suppose that $Q$ is a star with center $q$ (Figure~\ref{fig:one-vertex_extension_quotient_graph_evolution_star_q_center}).
\begin{enumerate}
\item If $p$ is a pendant vertex of $q$, then $Q'$ is obtained from $Q$ by adding a spoke $p$.
\item If $p$ is a false-twin of $q$, then $Q'=Q_1'\cup Q_2'$, where $Q_2'=S_2$ is star-center with spokes $p$ and $q$. 
\item If $p$ is a true-twin of $q$, then $Q'=Q_1'\cup Q_2'$, where $Q_2'=K_3$ is complete. 
\end{enumerate}
\item Suppose that $Q$ is star with spoke $q$ (Figure~\ref{fig:one-vertex_extension_quotient_graph_evolution_star_q_spoke}).
\begin{enumerate}
\item If $p$ is a pendant vertex of $q$, then $Q'=Q_1'\cup Q_2'$, where $Q_2'=S_{2}$ is star-spoke with center $q$ and spokes $p$ and a split-node. 
\item If $p$ is a false-twin of $q$, then $Q'$ is obtained from $Q$ by adding a spoke $p$.
\item If $p$ is a true-twin of $q$, then $Q'=Q_1'\cup Q_2'$, where $Q_2'=K_3$ is complete. 
\end{enumerate}
\item Suppose that $Q$ is complete with a vertex $q$ (Figure~\ref{fig:one-vertex_extension_quotient_graph_evolution_complete}).
\begin{enumerate}
\item If $p$ is a pendant vertex of $q$, then $Q'=Q_1'\cup Q_2'$, where $Q_2'=S_{2}$ is star-spoke with center $q$ and spokes $p$ and a split-node. 
\item If $p$ is a false-twin of $q$, then $Q'=Q_1'\cup Q_2'$, where $Q_2'=S_{2}$ is star-center with spokes $q$ and $p$. 
\item If $p$ is a true-twin of $q$, then $Q'$ is obtained from $Q$ by adding one additional vertex $p$ and full connections to all other vertices.
\end{enumerate}
\item Suppose that $Q$ is prime with a vertex $q$ (Figure~\ref{fig:one-vertex_extension_quotient_graph_evolution_prime}).
\begin{enumerate}
\item If $p$ is a pendant vertex of $q$, then $Q'=Q_1'\cup Q_2'$, where $Q_2'=S_{2}$ is star-spoke with center $q$ and spokes $p$ and a split-node. 
\item If $p$ is a false-twin of $q$, then $Q'=Q_1'\cup Q_2'$, where $Q_2'=S_{2}$ is star-center with spokes $q$ and $p$. 
\item If $p$ is true-twin of $q$, then $Q'=Q_1'\cup Q_2'$, where $Q_2'=K_3$ is complete. 
\end{enumerate}
\end{enumerate}
\end{lemma}

\begin{figure*}[t]
\centering
\begin{subfigure}{0.5\textwidth}
\centering
\includegraphics[width=0.95\textwidth,page=21]{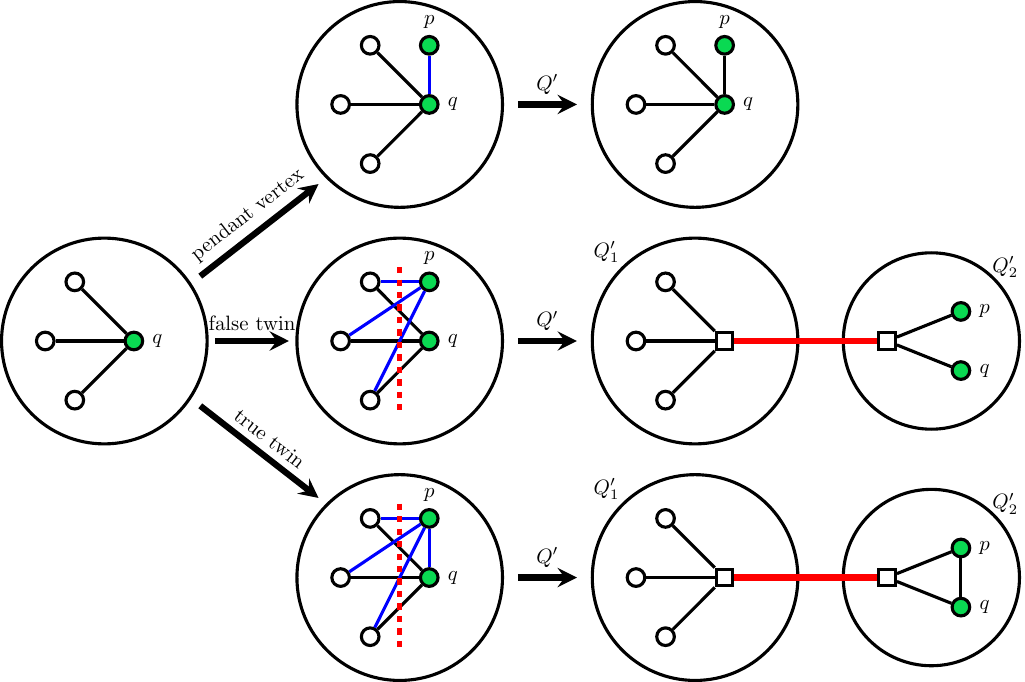}
\caption{If $Q$ is star with center $q$.}
\label{fig:one-vertex_extension_quotient_graph_evolution_star_q_center}
\end{subfigure}\begin{subfigure}{0.5\textwidth}
\centering
\includegraphics[width=0.95\textwidth,page=22]{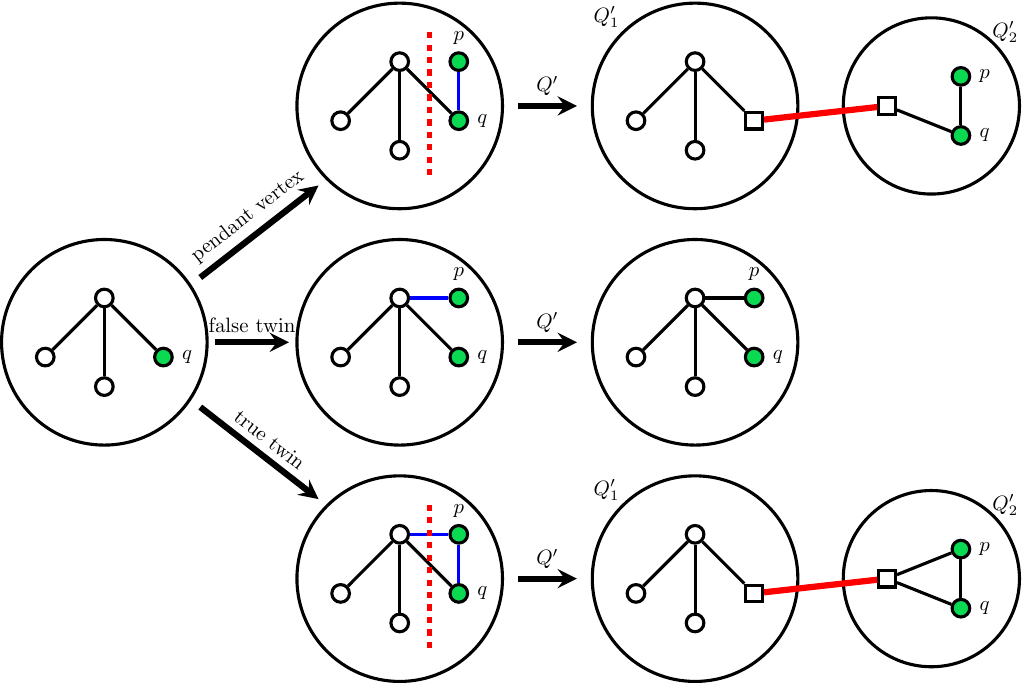}
\caption{If $Q$ is star with spoke $q$.}
\label{fig:one-vertex_extension_quotient_graph_evolution_star_q_spoke}
\end{subfigure}

\bigskip

\begin{subfigure}{0.5\textwidth}
\centering
\includegraphics[width=0.95\textwidth,page=23]{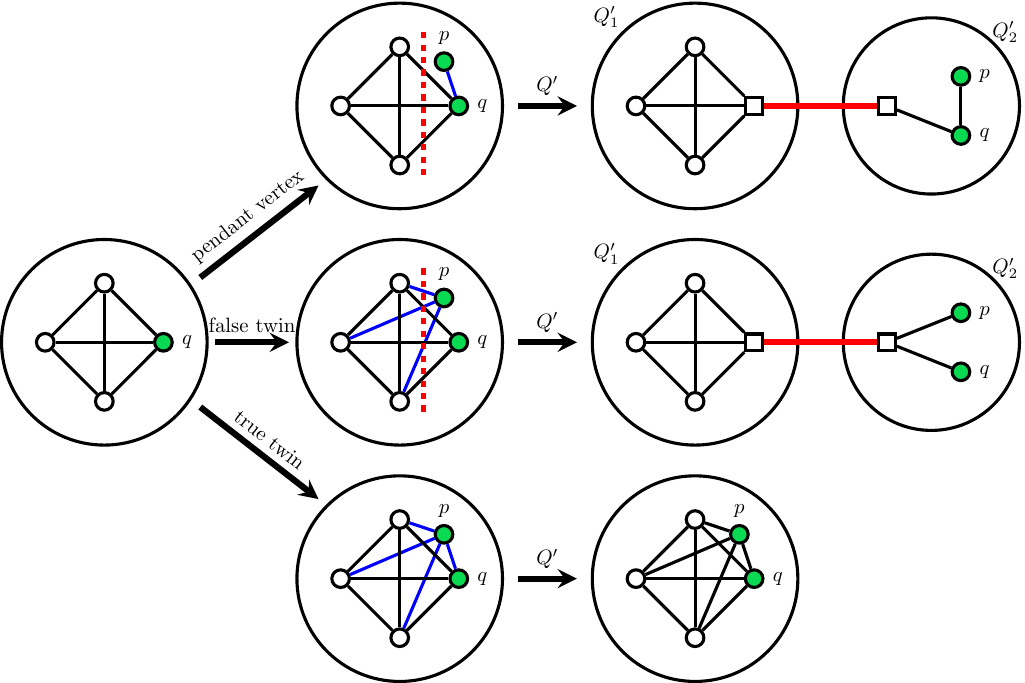}
\caption{If $Q$ is complete with vertex $q$.}
\label{fig:one-vertex_extension_quotient_graph_evolution_complete}
\end{subfigure}\begin{subfigure}{0.5\textwidth}
\centering
\includegraphics[width=0.95\textwidth,page=24]{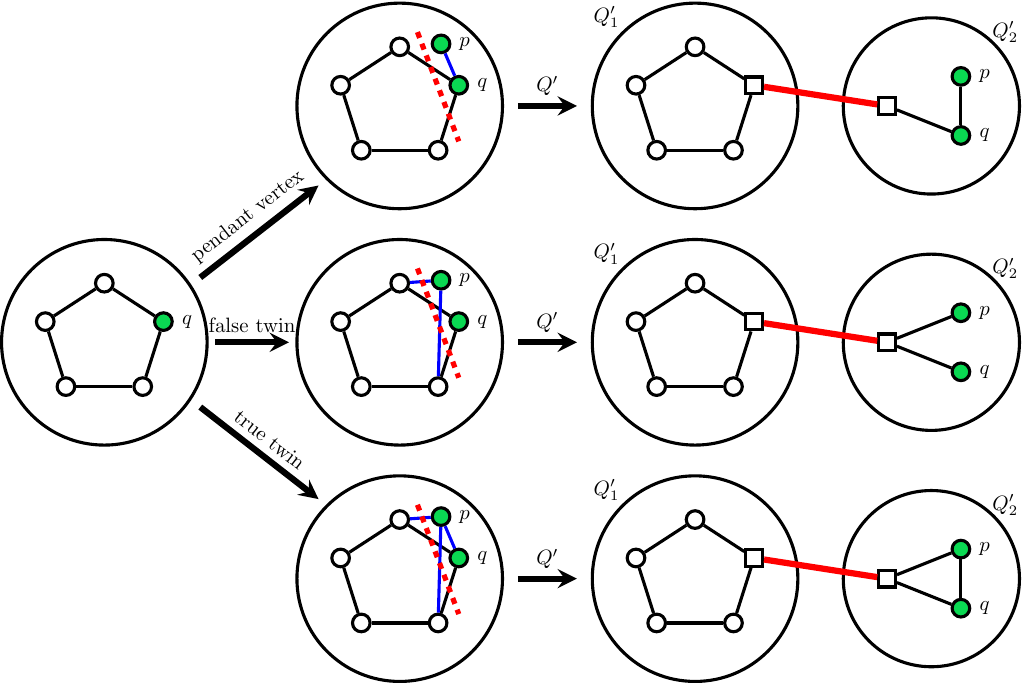}
\caption{If $Q$ is prime with vertex $q$.}
\label{fig:one-vertex_extension_quotient_graph_evolution_prime}
\end{subfigure}
\caption{Visual proof of Lemma~\ref{thm:one-vertex_quotient_evolution_lemma}, showing the evolution under one-vertex extensions of a quotient graph $Q$ in the split decomposition of a graph $G$, where $Q$ contains a vertex $q\in V(G)$. If $G'$ is the graph obtained from $G$ by a one-vertex extension $p$ of $q$ (either a pendant vertex, a false-twin, or a true-twin), then $Q$ evolves into $Q'$ in the split decomposition of $G'$ as shown.}
\label{fig:one-vertex_extension_quotient_graph_evolution}
\end{figure*}

\begin{table*}[t!]
\centering
\scalebox{0.9}{
\begin{tabular}{|c|ccc|}
\hline
Original quotient $Q$&pendant vertex $p$&false twin $p$&true twin $p$\\
\hline
&&&\\
\includegraphics[scale=0.5,page=25]{Figures/Section_QASST_Structure_for_DH_Graphs.pdf}&\includegraphics[scale=0.5,page=29]{Figures/Section_QASST_Structure_for_DH_Graphs.pdf}&\includegraphics[scale=0.5,page=30]{Figures/Section_QASST_Structure_for_DH_Graphs.pdf}&\includegraphics[scale=0.5,page=31]{Figures/Section_QASST_Structure_for_DH_Graphs.pdf}\\
$Q$ is star (center $q$)&$Q'$&$Q'=Q_1'\cup Q_2'$&$Q'=Q_1'\cup Q_2'$\\
&&&\\
\includegraphics[scale=0.5,page=26]{Figures/Section_QASST_Structure_for_DH_Graphs.pdf}&\includegraphics[scale=0.5,page=32]{Figures/Section_QASST_Structure_for_DH_Graphs.pdf}&\includegraphics[scale=0.5,page=33]{Figures/Section_QASST_Structure_for_DH_Graphs.pdf}&\includegraphics[scale=0.5,page=34]{Figures/Section_QASST_Structure_for_DH_Graphs.pdf}\\
$Q$ is star (spoke $q$)&$Q'=Q_1'\cup Q_2'$&$Q'$&$Q'=Q_1'\cup Q_2'$\\
&&&\\
\includegraphics[scale=0.5,page=27]{Figures/Section_QASST_Structure_for_DH_Graphs.pdf}&\includegraphics[scale=0.5,page=35]{Figures/Section_QASST_Structure_for_DH_Graphs.pdf}&\includegraphics[scale=0.5,page=36]{Figures/Section_QASST_Structure_for_DH_Graphs.pdf}&\includegraphics[scale=0.5,page=37]{Figures/Section_QASST_Structure_for_DH_Graphs.pdf}\\
$Q$ is complete&$Q'=Q_1'\cup Q_2'$&$Q'=Q_1'\cup Q_2'$&$Q'$\\
&&&\\
\includegraphics[scale=0.5,page=28]{Figures/Section_QASST_Structure_for_DH_Graphs.pdf}&\includegraphics[scale=0.5,page=38]{Figures/Section_QASST_Structure_for_DH_Graphs.pdf}&\includegraphics[scale=0.5,page=39]{Figures/Section_QASST_Structure_for_DH_Graphs.pdf}&\includegraphics[scale=0.5,page=40]{Figures/Section_QASST_Structure_for_DH_Graphs.pdf}\\
$Q$ is prime&$Q'=Q_1'\cup Q_2'$&$Q'=Q_1'\cup Q_2'$&$Q'=Q_1'\cup Q_2'$\\
&&&\\
\hline
\end{tabular}
}
\caption{Summary of the twelve possible cases for computing a quotient graph $Q'$, the evolution of an existing quotient graph $Q$ under a one-vertex extension $p$ of $q$, as described in Lemma~\ref{thm:one-vertex_quotient_evolution_lemma} and Figure~\ref{fig:one-vertex_extension_quotient_graph_evolution}.}
\label{tab:one-vertex_extension_new_quotient_graph}
\end{table*}

\begin{proof}
Although this lemma has a lengthy statement, this is merely to list all of the combinations of possibilities. The proof of these cases follows readily from a visual argument, as summarized in Figure~\ref{fig:one-vertex_extension_quotient_graph_evolution}. The twelve subcases for the resulting quotient graph $Q'$ are also enumerated in Table~\ref{tab:one-vertex_extension_new_quotient_graph}. In these examples, we show a graph with no nontrivial strong splits and hence only a single quotient graph. However, this same technique can be used in the more general case of a QASST with any number of quotient graphs by applying a vertex extension to a single quotient graph. 
\end{proof}

The preceding lemma shows how the sequence of one-vertex extensions used to construct a distance-hereditary graph can also be used to define a corresponding sequence of evolutions in the QASST.
We conclude this section with an alternative proof of the fact that quotient graphs of distance-hereditary graphs are necessarily star or complete, a result which was originally established in \cite{hammer1990completely}. Our proof leverages Lemma~\ref{thm:one-vertex_quotient_evolution_lemma} using an inductive argument based on the sequence of one-vertex extensions used to construct the graph.

Before this, we pause here to make a comment about the evolution of the quotient graphs under one-vertex extensions as described in Table~\ref{tab:one-vertex_extension_new_quotient_graph}.
As shown, the quotient graph $Q$ either evolves into a single new quotient graph $Q'$, or else it evolves into two quotient graphs $Q'=Q_1'\cup Q_2'$. In the former case, if $Q$ is star or complete, then $Q'$ is also star or complete with one additional vertex.
In the latter case, $Q_1'\cong Q$ except with $q$ replaced by a split-node, and $Q_2'$ belongs to one of the three cases described in Figure~\ref{fig:one_vertex_evolution_Q2_possibilities}, all of which are star or complete.
Hence, if the original $Q$ is star or complete, then both $Q_1'$ and $Q_2'$ are also star or complete.
We rely on this fact for the inductive step in proving the following theorem.

\begin{theorem}[Originally established in~\cite{hammer1990completely}]
\label{thm:distance_hereditary_iff_star_complete_quotient}
A graph $G$ is distance-hereditary if and only if the quotient graphs in the split decomposition of $G$ are star or complete (i.e.~$G$ is completely separable).
\end{theorem}

\begin{proof}
If $G$ is a distance-hereditary graph, then it can be constructed using a sequence of one-vertex extensions (Theorem 1 of~\cite{bandelt1986distance}).
We proceed by induction on the number of vertices in $G$, noting for the base case that a graph consisting of a single vertex is both star and complete and is the only graph in the split decomposition. Lemma~\ref{thm:one-vertex_quotient_evolution_lemma} describes explicitly how the quotient graphs in the split decomposition of $G$ evolve under this sequence of one-vertex extensions (we assume that $G$ is connected). Specifically, if the preceding graph in the sequence of one-vertex extensions contains only star and complete quotient graphs, then so does the next graph, up until $G$. This establishes the first part of the theorem.

To prove the converse, we use one of the alternative characterizations of distance-hereditary graphs: a connected graph is distance-hereditary if and only if it has rank-width 1 (Proposition 7.3 of \cite{oum2005rank}).
Star graphs and complete graphs are special cases of distance-hereditary graph, and hence the graphs in these families always have rank-width 1.
Another result relates the rank-width of a graph to the quotient graphs in its split decomposition: the rank-width of a graph is the maximum rank-width of all of its quotient graphs (Theorem 4.3 of \cite{hlinveny2008width}).
Therefore, $G$ necessarily has rank-width 1 and is thus distance-hereditary if its split decomposition consists only of star and complete quotient graphs.
\end{proof}

\subsection{QASST Evolution for Induced Subgraphs}
\label{sect:qasst_evolution_for_induced_subgraphs}
Continuing our study of the structural properties of the QASST for distance-hereditary graphs, we now consider their behavior under induced subgraphs.
Induced subgraphs of distance-hereditary graphs are also distance-hereditary by definition. Furthermore, there is a natural relationship between the split decompositions of two graphs related in this way. Here, we seek to characterize how the QASST of a distance-hereditary graph evolves when restricting to induced subgraphs.
This examination is partly motivated by the connection with quantum information theory, wherein certain measurements on graph states apply local complements and vertex deletions~\cite{dahlberg2018transforming}, operations which characterize vertex minors~\cite{oum2005rank}.
In general, a new QASST for an induced subgraph can be obtained through a process of modifying and merging quotient graphs in the QASST of the original graph.

Consider a distance-hereditary graph $G$ with split decomposition described by the graph-labeled tree $QASST(G)=(Q_1,\cdots,Q_k)$. In the trivial case where $G$ is star or complete, then $k=1$ and the split decomposition consists of a single quotient graph $Q_1=G$. In this case, every connected induced subgraph of $G$ is a smaller star or complete graph, likewise with a single quotient graph in the QASST.
If $G$ is a nontrivial distance-hereditary graph, then $QASST(G)=(Q_1,\cdots,Q_k)$ consists of a graph-labeled tree of two or more quotient graphs. In this case, every quotient graph is incident to at least one other quotient graph via an edge in $QASST(G)$; this edge corresponds to a pair of split-nodes in the adjacent quotient graphs. Hence, every quotient graph of $G$ contains at least one split-node.

Suppose that $H$ is a connected induced subgraph of a nontrivial distance-hereditary graph $G$.
From the quotient graphs $Q_1,\cdots,Q_k$ of $G$, we may consider corresponding induced subgraphs $Q_1',\cdots,Q_k'$ obtained by deleting any leaf-nodes not contained in $H$.
This results in an intermediate graph-labeled tree $(Q_1',\cdots,Q_k')$ from which $H$ may be reconstructed, although this is not necessarily $QASST(H)$.
Although each $Q_i'$ obtained in this way is still star or complete since they are induced subgraphs of the original quotient graph $Q_i$, the splits described by the edges in this tree may not be strong depending on the cases described in Figure~\ref{fig:quotient_graph_combinations_valid}.
However, this graph-labeled tree may be refined into $QASST(H)$ by merging together some subset of the intermediate quotient graphs $Q_1',\cdots,Q_k'$ until all remaining splits are strong.
Merging entails deleting a pair of connected split-nodes from adjacent quotient graphs and replacing these with all-to-all connections between the split-nodes' neighbors, as was shown in Figure~\ref{fig:reconstruct_graph_from_qasst}.

Since the intermediate quotient graphs $Q_1',\cdots,Q_k'$ are obtained from $Q_1,\cdots,Q_k$ by deleting some combination of leaf-nodes, and every quotient graph has at least one split-node, all resulting graphs have at least one split-node as well.
This remains true when merging quotient graphs unless only a single graph remains.
As a consequence of this, there are three possibilities for each intermediate quotient graph $Q_1',\cdots,Q_k'$. These are:
\begin{enumerate}
\item $|V(Q_i')|=1$ ($Q_i'$ consists of a single split-node);
\item $|V(Q_i')|=2$ ($Q_i'$ consists of two split-nodes, or one split-node and one leaf-node);
\item $|V(Q_i')|\geq 3$.
\end{enumerate}
Any intermediate quotient graph containing two or fewer vertices (cases i and ii) can be merged into a neighboring quotient graph.
This is because a connected graph with one or two vertices is simultaneously complete, star-center, and star-spoke with respect to any vertex.
Hence, such a graph necessarily falls into one of the invalid cases described in Figure~\ref{fig:quotient_graph_combinations_valid} since an adjacent edge in the graph-labeled tree does not represent a strong split. Examples of these cases are illustrated in Figure~\ref{fig:induced_subgraph_QASST}.

\begin{figure*}[t]
\centering
\begin{subfigure}{0.33\textwidth}
\centering
\includegraphics[width=0.8\textwidth,page=1]{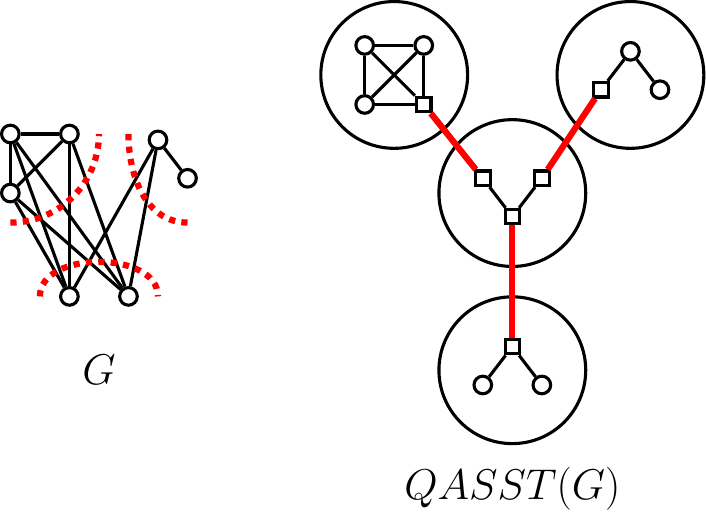}
\caption{Original graph $G$}
\label{fig:induced_subgraph_QASST_G}
\end{subfigure}\begin{subfigure}{0.66\textwidth}
\centering
\includegraphics[width=0.8\textwidth,page=2]{Figures/Section_QASST_Structure_for_DH_Graphs_part2.pdf}
\caption{Induced quotient with a single remaining vertex}
\label{fig:induced_subgraph_QASST_H1}
\end{subfigure}\\

\bigskip

\begin{subfigure}{0.5\textwidth}
\centering
\includegraphics[width=0.8\textwidth,page=3]{Figures/Section_QASST_Structure_for_DH_Graphs_part2.pdf}
\caption{Induced quotient with two remaining vertices}
\label{fig:induced_subgraph_QASST_H2}
\end{subfigure}\begin{subfigure}{0.5\textwidth}
\centering
\includegraphics[width=0.8\textwidth,page=4]{Figures/Section_QASST_Structure_for_DH_Graphs_part2.pdf}
\caption{Induced quotient with three or more remaining vertices}
\label{fig:induced_subgraph_QASST_H3}
\end{subfigure}
\caption{Examples of how the QASST evolves for induced subgraphs. A graph-labeled tree for the subgraph is obtained from $QASST(G)$ by deleting leaf-nodes corresponding to excluded vertices. For each quotient graph, there are several cases depending on the number of remaining vertices. Quotient graphs with two or fewer vertices are merged into their neighbors by deleting a pair of adjacent split-nodes. This process continues until all red edges represent strong splits or else only a single quotient graph remains.}
\label{fig:induced_subgraph_QASST}
\end{figure*}

As shown in Figure~\ref{fig:induced_subgraph_QASST_H1}, the process of merging two quotient graphs results in a new intermediate quotient graph, but since the structure of the graph can change after merging, incident edges in the graph-labeled tree might no longer represent valid strong splits.
Specifically, this occurs whenever such an edge represents a join of type sc-ss, ss-sc, or c-c (in the sense of Figure~\ref{fig:quotient_graph_combinations_valid}).
In these cases, the intermediate quotient graphs may also be merged with a neighbor.
This process of merging continues until either all remaining edges in the graph-labeled tree represent valid strong splits, or else only a single quotient graph remains.
Note that the former case is not possible if there exists a quotient graph with fewer than three vertices.
Algorithm~\ref{alg:induced_QASST} outlines this procedure using pseudo code.

Although we only consider connected induced subgraphs here, note that this procedure readily generalizes to disconnected induced subgraphs as well. The only difference is that the intermediate graph-labeled tree must first be divided into connected components.

\begin{algorithm}[ht]
    \caption{Compute Induced Subgraph QASST}
    \label{alg:induced_QASST}
    \KwData{An induced subgraph $H\leq G$, the split decomposition $QASST(G)=(Q_1,\cdots,Q_k)$.}
    \KwResult{Subgraph split decomposition $QASST(H)=(Q_1',\cdots,Q_\ell')$.}

    \For{$i=1,\cdots,k$}{
        Initialize induced subgraph $Q_i'\leq Q_i$\Comment*[r]{delete leaf-nodes not in $H$}
    }
    Initialize $QASST(H)=(Q_1',\cdots,Q_k')$;\\
    \If{$k=1$}{
        Return $QASST(H)$\Comment*[r]{trivial case}
    }
    \While{there exist adjacent $(Q_i',Q_j')\in E(QASST(H))$ whose corresponding split is not strong}{
        \Comment{The split between $(Q_i',Q_j')$ is of type sc-ss, ss-sc, or c-c}
        Remove $Q_j'$ from $QASST(H)$;\\
        Update $Q_i'=\textbf{merge}(Q_i',Q_j')$ in $QASST(H)$;
    }
    Return $QASST(H)=(Q_1',\cdots,Q_\ell')$;\\
\end{algorithm}

\section{Examining LC Equivalence Classes using the QASST}
\label{sect:LC_equivalence_classes_with_the_QASST}

Through the lens of split decompositions, we now consider how this tool can be used to help us characterize the orbits of locally equivalent graphs. Although we will primarily restrict ourselves to distance-hereditary graphs, many of the facts discussed here apply broadly to any simple, connected graph.

\subsection{Invariant Split Decomposition}

We showed in Section~\ref{sect:split_decompositions} that the quotient augmented strong split tree of a graph $G$ can be used to fully recover the original graph, and hence $QASST(G)$ contains equivalent information to $G$. In contrast to this, the strong split tree without the quotient graphs is not enough to perform this reconstruction.
Although it contains less information, $SST(G)$ allows us to examine the split structure of a graph $G$ even without knowing its exact connectivity.
In fact, this is significant for our exploration into local equivalence classes because Bouchet showed that LC equivalent graphs have the same splits~\cite{bouchet1987reducing,bouchet1989connectivity}.

\begin{lemma}[Bouchet~\cite{bouchet1987reducing}, Lemma 2.1]
\label{thm:splits_in_locally_equivalent_graphs}
Locally equivalent graphs have the same splits.
\end{lemma}

This statement applies both to strong splits and splits which are not strong.
Given that the strong splits in a graph define the edges in the strong split tree, this implies that all graphs in the same local equivalence class have the same strong split tree.
In other words, for any graph $G$, we have that $SST(G)$ is an invariant of the LC orbit ${\mathcal O}(G)$.
We state this here as a corollary.

\begin{corollary}
\label{thm:LC_graphs_STT}
If $G_1$ and $G_2$ are locally equivalent, then $SST(G_1)=SST(G_2)$.
\end{corollary}

As for the quotient graphs, these can change under local complements, but there is a natural bijective relationship between the quotient graphs obtained via the split decomposition of two locally equivalent graphs.
In particular, two quotient graphs corresponding in this way must have the same number of vertices, both leaf-nodes and split-nodes in matching proportions.

In connection with distance-hereditary graphs, we can make an even stronger statement using another result from Bouchet.

\begin{lemma}[Bouchet~\cite{bouchet1988transforming}, Corollary 3.3]
\label{thm:prime_graph_local_equivalence}
A graph is distance-hereditary if and only if it is totally decomposable.
\end{lemma}

A graph is \textbf{totally decomposable} in the sense of Cunningham and Bouchet if it admits a decomposition using splits (not necessarily \textit{strong} splits) into quotient graphs with exactly three nodes~\cite{cunningham1982decomposition,bouchet1988transforming}. In fact, this occurs precisely when the quotient graphs of the \textit{strong} split decomposition (what Cunningham refers to as the \textit{minimal} or \textit{canonical} decomposition) are star or complete (what Cunningham refers to as \textit{brittle} split components)~\cite{cunningham1982decomposition}.
This is exactly what it means for a graph to be completely separable, and hence also distance-hereditary.
Bouchet showed further that the property of being distance-hereditary is preserved under local complements~\cite{bouchet1988transforming}, as is the property of being prime (i.e.~containing no nontrivial splits)~\cite{bouchet1987reducing}.

\begin{lemma}[Bouchet~\cite{bouchet1988transforming}, Corollary 4.2]
\label{thm:distance_hereditary_LC_equivalence}
Any graph locally equivalent to a totally decomposable [distance-hereditary] graph is totally decomposable [distance-hereditary].
\end{lemma}

\begin{lemma}[Bouchet~\cite{bouchet1987reducing}, Lemma 2.2]
\label{thm:prime_graph_local_equivalence}
Any graph locally equivalent to a prime graph is prime.
\end{lemma}

Bouchet leveraged split decompositions to characterize those graphs which are locally equivalent to certain special graph cases, including trees~\cite{bouchet1988transforming} and circle graphs~\cite{bouchet1987reducing}. 
Trees are a special case of distance-hereditary graph; these are precisely the graphs which can be constructed recursively using only pendant vertices.
Every distance-hereditary graph is also a special case of the broader family of circle graphs, although there exist non-distance-hereditary circle graphs.

We conclude this subsection by stating the result concerning trees proven by Bouchet~\cite{bouchet1988transforming}, which was originally conjectured by Mulder~\cite{mulder1986}.

\begin{theorem}[Mulder's Conjecture~\cite{mulder1986}, proven by Bouchet~\cite{bouchet1988transforming}, Corollary 5.4]
Any two locally equivalent trees are isomorphic.
\end{theorem}

\subsection{Propagation of Local Complements through the QASST}

In the spirit of Bouchet, we seek to characterize the LC equivalence classes of certain families of distance-hereditary graphs using the split decomposition.
We do this by exploiting the QASST.
To understand this, we first make explicit the effect of local complements applied to a graph on that graph's QASST.

Recall that a local complement with respect to a chosen vertex $v\in V(G)$ affects the edges in $N_G(v)$. However, two adjacent vertices in $G$ may belong to different quotient graphs in $QASST(G)$.
This means that local complement can have an effect on the edges in multiple quotient graphs.
Intuitively, we can interpret local complements as \textit{transmitting} through adjacent quotient graphs in the QASST via pairs of adjacent split-nodes. Figure~\ref{fig:LC_propagation_thru_QASST_example} shows an explicit example, and we formalize this idea in terms of the recursive function defined in Algorithm~\ref{alg:LC_propagation}.

\begin{figure*}[t]
\centering
\begin{subfigure}{0.33\textwidth}
\centering
\includegraphics[width=0.7\textwidth,page=1]{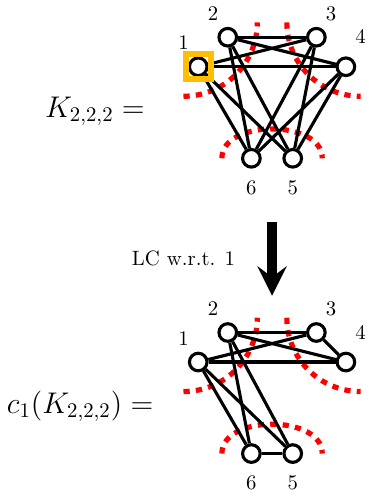}
\caption{Graphs related by local complement}
\label{fig:LC_propagation_original_graphs}
\end{subfigure}\begin{subfigure}{0.66\textwidth}
\centering
\includegraphics[width=0.8\textwidth,page=2]{Figures/Section_Exploring_LC_Equivalence_Classes_with_the_QASST.pdf}
\caption{Propagation of LC operations through the corresponding QASST}
\label{fig:LC_propagation_QASSTs}
\end{subfigure}
\caption{Example of local complement propagating through the QASST. Part~\ref{fig:LC_propagation_original_graphs} shows an example of the graph obtained by applying local complement with respect to vertex 1 in $K_{2,2,2}$. Part~\ref{fig:LC_propagation_QASSTs} shows the effect of this on the QASST. Local complement in $K_{2,2,2}$ corresponds to applying a series of local complements on the quotient graphs of $QASST(K_{2,2,2})$ via the function defined in Algorithm~\ref{alg:LC_propagation}. In this example, local complements are transmitted across each pair of split-nodes. The final quotient graphs obtained in $QASST(c_1(K_{2,2,2}))$ are: $Q_1'=c_1(Q_1)$, $Q_0'=c_{s_0^1}(Q_0)$, $Q_2'=c_{s_2^0}(Q_2)$, and $Q_3'=c_{s_3^0}(Q_3)$.}
\label{fig:LC_propagation_thru_QASST_example}
\end{figure*}

\begin{algorithm}[htb]
    \caption{LC Propagation}
    \label{alg:LC_propagation}
    \KwData{A vertex $v$, quotient graphs $Q_1,\cdots,Q_k$.}
    \KwResult{Quotient graphs $Q_1',\cdots,Q_k'$.}

    \For{$\ell=1,\cdots,k$}{
        Initialize $Q_\ell'=Q_\ell$;\\
    }
    Identify $Q_i$ for which $v\in V(Q_i)$;\\
    Update $Q_i'=c_v(Q_i)$\Comment*[r]{apply LC w.r.t. $v$ to $Q_i$}
    \For{split-node $s_i^j\in N_{Q_i}(v)$}{
        \Comment{Apply LC on neighboring quotient graphs recursively}
        Update $(Q_1',\cdots,Q_k')=\textbf{LC Propagation}(s_j^i,Q_1',\cdots,Q_k')$;\\
    }
    Return $(Q_1',\cdots,Q_k')$;\\
\end{algorithm}

Local complement on a graph $G$ is only with respect to a vertex $v\in V(G)$; $v$ corresponds to some leaf-node in one of the quotient graphs of $QASST(G)$.
The first call of the function defined in Algorithm~\ref{alg:LC_propagation} must use a leaf-node, but observe that during the recursive step, every subsequent call of the function is with respect to a split-node.
Furthermore, each call of the function only computes one instance of local complement as determined by the input vertex.
Split-nodes always come in pairs, representing two adjacent quotient graphs in the QASST. Hence, for a given split-node $s_i^j\in N_{Q_i}(v)$, the recursive step of the algorithm calls the function using that split-node's partner $s_j^i\in Q_j$.
In this way, a series of local complements propagates through $QASST(G)$.

Split-nodes can be thought of intuitively as bundles of fully connected edges between pairs of adjacent quotient graphs. Hence, when applying local complement on a quotient graph with respect to the neighbor of a split-node, this propagates to local complement on the next quotient graph via that split-node's partner.
Since $QASST(G)$ is a tree, this process will eventually terminate after transmitting from the original node through some number of bridges in the QASST defined by these pairs of split-nodes.
Not every quotient graph needs to be visited; only those quotient graphs which contain vertices adjacent to the original vertex $v\in V(G)$.

Figure~\ref{fig:LC_propagation_thru_QASST_example} shows the effect of propagating local complements on the QASST. In this example, the function of Algorithm~\ref{alg:LC_propagation} has three recursive iterations after the initial function call. This sequence of local complements is visualized by a branching path through the QASST. After the initial leaf-node, all subsequent local complements are with respect to the split-node where this path enters the corresponding quotient graph.

Algorithm~\ref{alg:LC_propagation} shows that local complement applied to a graph induces some number of local complements on the quotient graphs of the QASST.
In particular, this means that any two graphs related via local complement will have either the same quotient graphs or quotient graphs related via some induced local complement.
This can be extended to the full LC orbit to prove the following fact about locally equivalent graphs.

\begin{theorem}
\label{thm:LC_equiv_graphs_have_LC_equiv_quotients}
If $G_1$ and $G_2$ are locally equivalent graphs, then the quotient graphs in $QASST(G_1)$ and $QASST(G_2)$ are also locally equivalent.
\end{theorem}

\begin{proof}
Suppose $G_1,G_2\in{\mathcal G}_n$ are locally equivalent graphs.
Thus, there exists some sequence of primitive local complements $f=c_{i_\ell}\circ\cdots\circ c_{i_1}\in{\mathcal L}_n$ for which $G_2=f(G_1)$.
As locally equivalent graphs, $G_1$ and $G_2$ have the same splits (Lemma~\ref{thm:splits_in_locally_equivalent_graphs}) and hence there is a natural bijection of quotient graphs $Q_1^1,\cdots,Q_k^1$ of $QASST(G_1)$ and $Q_1^2,\cdots,Q_k^2$ of $QASST(G_2)$.
By repeated application of Algorithm~\ref{alg:LC_propagation}, for each corresponding pair of quotient graphs $Q_j^1$ and $Q_j^2$, there exists a sequence of primitive local complements $f_j=c_{j_\ell}\circ\cdots\circ c_{j_1}\in{\mathcal L}_{|V(Q_j)|}$ such that $f_j(Q_j^1)=Q_j^2$, which shows that $Q_j^1$ and $Q_j^2$ are locally equivalent.
This completes the proof.
\end{proof}

\begin{figure*}[t]
\centering
\begin{subfigure}{0.33\textwidth}
\centering
\includegraphics[width=0.7\textwidth,page=3]{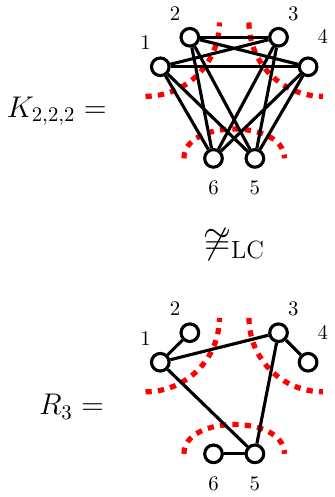}
\caption{Non-LC equivalent graphs}
\label{fig:non_LC_equivalent_QASST_equivalent_graphs}
\end{subfigure}\begin{subfigure}{0.66\textwidth}
\centering
\includegraphics[width=0.8\textwidth,page=4]{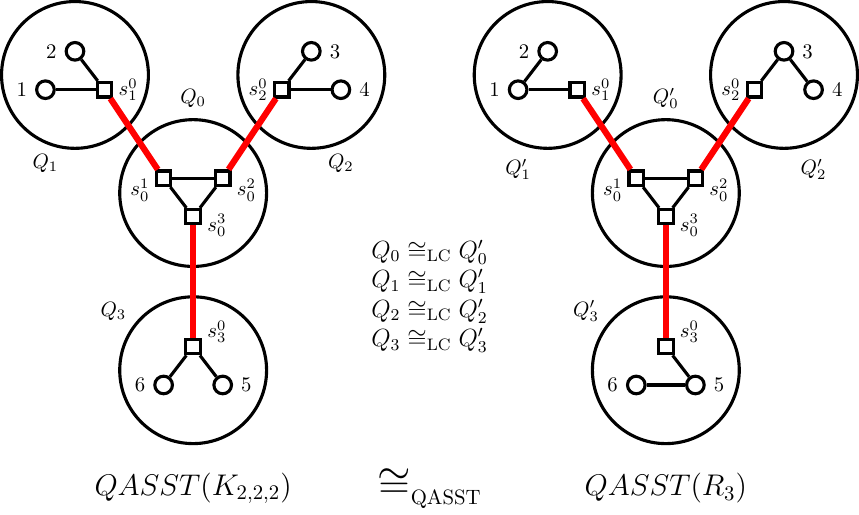}
\caption{QASST equvialent split decompositions}
\label{fig:non_LC_equivalent_QASST_equivalent_QASSTs}
\end{subfigure}
\caption{An example of QASST equivalent graphs $K_{2,2,2}$ and $R_3$ which are not LC equivalent. These graphs have the same strong splits: $STT(K_{2,2,2})=STT(R_3)$; and each quotient graph of $K_{2,2,2}$ is locally equivalent to each quotient graph of $R_3$. However, these two graphs are not locally equivalent (these are special cases of graphs shown to be non-LC equivalent in \ref{app:LC_non-equivalence}).}
\label{fig:Non_LC_equivalent_QASST_equivalent_counterexample}
\end{figure*}

Note that the converse of Theorem~\ref{thm:LC_equiv_graphs_have_LC_equiv_quotients} does not hold. There exist graphs $G_1$ and $G_2$ for which $SST(G_1)=SST(G_2)$ (i.e.~the graphs have the same strong splits), and each quotient graph of $QASST(G_1)$ is locally equivalent to each corresponding quotient graph of $QASST(G_2)$, but $G_1$ and $G_2$ are not themselves locally equivalent.
Figure~\ref{fig:Non_LC_equivalent_QASST_equivalent_counterexample} shows one such counterexample.
Define two graphs with these properties (the same strong split tree and locally equivalent quotient graphs) to be \textbf{QASST equivalent}.
QASST equivalence can be thought of as a weaker version of LC equivalence.

Theorem~\ref{thm:LC_equiv_graphs_have_LC_equiv_quotients} can be used as a tool to check when two graphs with the same strong split tree are not locally equivalent.
Although the counterexample of Figure~\ref{fig:Non_LC_equivalent_QASST_equivalent_counterexample} shows that it cannot distinguish between all LC equivalence classes, it does imply that two graphs with non-locally equivalent quotient graphs cannot be locally equivalent.
Since star and complete graphs are always locally equivalent, this is of limited utility in the distance-hereditary case since these are the only types of quotient graphs possible.
However, for the non-distance-hereditary case (i.e.~allowing for prime quotient graphs), this fact suggests that a general classification of LC orbits should begin with classifying the orbits of prime graphs.

\subsection{Counting QASST Equivalences}
\label{sect:QASST_equivalence}

Although QASST equivalent graphs are not necessarily LC equivalent, Theorem~\ref{thm:LC_equiv_graphs_have_LC_equiv_quotients} implies that LC equivalent graphs must be QASST equivalent.
The size of the QASST equivalence class of any graph $G$ depends on the sizes of the LC equivalence classes of the quotient graphs of $G$.
Since $QASST$ equivalent graphs have the same strong splits, the product of the sizes of the LC orbits of the quotient graphs immediately gives an upper bound on the total number of QASST equivalent graphs.
In other words, if $G$ is a graph whose split decomposition has quotient graphs $Q_1,\cdots,Q_k$, then the number of graphs QASST equivalent to $G$ is at most $|{\mathcal O}(Q_1)||{\mathcal O}(Q_2)|\cdots|{\mathcal O}(Q_k)|$, where ${\mathcal O}(Q_i)$ is the LC orbit of the $i^{\text{th}}$ quotient graph in $QASST(G)$.
This bound simply counts the number of ways that combinations of locally equivalent quotient graphs could be substituted into $QASST(G)$.

However, recall from the examples of Figure~\ref{fig:quotient_graph_combinations_valid} that certain combinations of adjacent quotient graphs in a split tree do not correspond to strong splits. In this case, the split decomposition would not be preserved.
Hence, to convert the stated upper bound into an equality, these invalid cases must be excluded.
Recall from Figure~\ref{fig:quotient_graph_combinations_valid} that the two types of invalid possibilities for adjacent quotient graphs in the QASST occur only between neighboring star or complete graphs.
Prime quotient graphs in the QASST can be adjacent to any other type of quotient graph; this will always correspond to a strong split.
Thus, we need only consider the former two cases when counting the number of invalid combinations. 
All other combinations of star, complete, or prime quotient graphs in a split tree preserve strong splits.
We state these observations as a lemma.

\begin{lemma}
\label{thm:counting_QASST_equivalence}
If $G$ is a graph with quotient graphs $Q_1,\cdots,Q_k$ in $QASST(G)$, then the number of graphs $QASST$ equivalent to $G$ is
\begin{eqnarray}\label{eq:count_QASST_equivalence}
\Phi(G)&=&|{\mathcal O}(Q_1)|\cdots|{\mathcal O}(Q_k)|-\mu(G),
\end{eqnarray}
where ${\mathcal O}(Q_i)$ is the LC orbit of the $i^{\text{th}}$ quotient graph of $QASST(G)$ and $\mu(G)$ denotes the number of distinct combinations of locally equivalent quotient graphs which do not preserve strong splits.
\end{lemma}

If restricting to the distance-hereditary case, computing the first part of the formula above is relatively easy. Recall that star and complete graphs are locally equivalent with LC orbits satisfying $|{\mathcal O}(K_n)|=|{\mathcal O}(S_{n})|=n+1$.
Hence, knowing the number of vertices in each quotient graph is sufficient to compute the first part of $\Phi(G)$ for any distance-hereditary graph $G$.
As for the second part, $\mu(G)$ depends explicitly on the structure of $QASST(G)$; there is no generic formula which can be stated elegantly.
Indeed, $QASST(G)$ defines a constraint satisfaction problem, wherein each edge in the split tree defines additional constraints and $\Phi(G)$ counts the number of solutions satisfying all constraints.
This problem is known to $\#P$-complete~\cite{dahlberg2020counting} and hence challenging in general, but by restricting to certain special cases of distance-hereditary graphs with sufficiently simple QASSTs, we can still derive some specific formulas for $\mu(G)$.
Section~\ref{sect:explicit_formulas_for_local_equivalence_classes} illustrates a number of these special families of graphs where we may compute explicit expressions for $\Phi(G)$.

\begin{figure*}[t]
\centering
\begin{subfigure}{0.47\textwidth}
\centering
\includegraphics[width=0.9\textwidth,page=30]{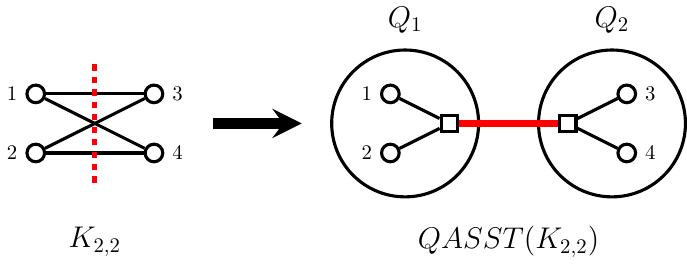}
\caption{Split decomposition of $K_{2,2}$}
\label{fig:K22_split_decomposition}
\end{subfigure}\begin{subfigure}{0.53\textwidth}
\centering
\includegraphics[width=0.9\textwidth,page=31]{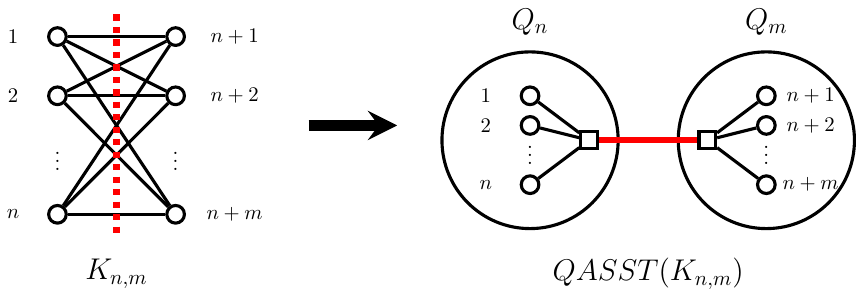}
\caption{Split decomposition of general $K_{n,m}$}
\label{fig:Knm_split_decomposition}
\end{subfigure}
\caption{
The split decomposition of a complete bipartite graph in the simple case $K_{2,2,}$ and general case $K_{n,m}$.
There is always a single strong split and two quotient graphs regardless of the number of vertices.
}
\label{fig:complete_bipartite_split_decomposition}
\end{figure*}

We conclude this section with a simple example illustrating the formula of Equation~\ref{eq:count_QASST_equivalence}.
Consider the complete bipartite graph $K_{2,2}$, which has a single strong split as shown in Figure~\ref{fig:K22_split_decomposition}.
This graph is equivalent to the 4-cycle $C_4$, whose complete LC orbit is illustrated in Adcock et al. (Figure 1d of~\cite{adcock2020mapping}), and is known to have size $|\mathcal{O}(K_{2,2})|=11$. 
$K_{2,2}$ has two star-center quotient graphs with two spokes each: $Q_1\cong Q_2\cong S_{2+1}$. Since $|{\mathcal O}(S_{2})|=4$, the number of QASST equivalent graphs is hence bounded above by $\Phi(K_{2,2})\leq|{\mathcal O}(Q_1)||{\mathcal O}(Q_2)|=|{\mathcal O}(S_2)|^2=4^2=16$.

\def\CBTableScale{0.65} 
\begin{table*}[tp]
\centering
\begin{tabular}{|cc|cc|c|c|}
\hline
$Q_1$&(count)&$Q_2$&(count)&Total&Transformation from $K_{2,2}$\\
\hline
sc&1&sc&1&1&sc-sc\\
\includegraphics[scale=\CBTableScale,valign=c,page=6]{Figures/Section_Exploring_LC_Equivalence_Classes_with_the_QASST.pdf}&&\includegraphics[scale=\CBTableScale,valign=c,page=10]{Figures/Section_Exploring_LC_Equivalence_Classes_with_the_QASST.pdf}&&&$\text{id}(K_{2,2})=\includegraphics[scale=\CBTableScale,valign=c,page=14]{Figures/Section_Exploring_LC_Equivalence_Classes_with_the_QASST.pdf}$\\
\hline
sc&1&c&1&1&sc-c\\
\includegraphics[scale=\CBTableScale,valign=c,page=6]{Figures/Section_Exploring_LC_Equivalence_Classes_with_the_QASST.pdf}&&\includegraphics[scale=\CBTableScale,valign=c,page=11]{Figures/Section_Exploring_LC_Equivalence_Classes_with_the_QASST.pdf}&&&$c_1(K_{2,2})=\includegraphics[scale=\CBTableScale,valign=c,page=15]{Figures/Section_Exploring_LC_Equivalence_Classes_with_the_QASST.pdf}$\\
\hline
c&1&sc&1&1&c-sc\\
\includegraphics[scale=\CBTableScale,valign=c,page=7]{Figures/Section_Exploring_LC_Equivalence_Classes_with_the_QASST.pdf}&&\includegraphics[scale=\CBTableScale,valign=c,page=10]{Figures/Section_Exploring_LC_Equivalence_Classes_with_the_QASST.pdf}&&&$c_3(K_{2,2})=\includegraphics[scale=\CBTableScale,valign=c,page=16]{Figures/Section_Exploring_LC_Equivalence_Classes_with_the_QASST.pdf}$\\
\hline
ss&2&c&1&2&ss-c\\
\includegraphics[scale=\CBTableScale,valign=c,page=8]{Figures/Section_Exploring_LC_Equivalence_Classes_with_the_QASST.pdf}&&\includegraphics[scale=\CBTableScale,valign=c,page=11]{Figures/Section_Exploring_LC_Equivalence_Classes_with_the_QASST.pdf}&&&$c_1\circ c_3(K_{2,2})=\includegraphics[scale=\CBTableScale,valign=c,page=17]{Figures/Section_Exploring_LC_Equivalence_Classes_with_the_QASST.pdf}$\\
\includegraphics[scale=\CBTableScale,valign=c,page=9]{Figures/Section_Exploring_LC_Equivalence_Classes_with_the_QASST.pdf}&&&&&$c_2\circ c_3(K_{2,2})=\includegraphics[scale=\CBTableScale,valign=c,page=18]{Figures/Section_Exploring_LC_Equivalence_Classes_with_the_QASST.pdf}$\\
\hline
c&1&ss&2&2&c-ss\\
\includegraphics[scale=\CBTableScale,valign=c,page=7]{Figures/Section_Exploring_LC_Equivalence_Classes_with_the_QASST.pdf}&&\includegraphics[scale=\CBTableScale,valign=c,page=12]{Figures/Section_Exploring_LC_Equivalence_Classes_with_the_QASST.pdf}&&&$c_3\circ c_1(K_{2,2})=\includegraphics[scale=\CBTableScale,valign=c,page=19]{Figures/Section_Exploring_LC_Equivalence_Classes_with_the_QASST.pdf}$\\
&&\includegraphics[scale=\CBTableScale,valign=c,page=13]{Figures/Section_Exploring_LC_Equivalence_Classes_with_the_QASST.pdf}&&&$c_4\circ c_1(K_{2,2})=\includegraphics[scale=\CBTableScale,valign=c,page=20]{Figures/Section_Exploring_LC_Equivalence_Classes_with_the_QASST.pdf}$\\
\hline
ss&2&ss&2&4&ss-ss\\
\includegraphics[scale=\CBTableScale,valign=c,page=8]{Figures/Section_Exploring_LC_Equivalence_Classes_with_the_QASST.pdf}&&\includegraphics[scale=\CBTableScale,valign=c,page=12]{Figures/Section_Exploring_LC_Equivalence_Classes_with_the_QASST.pdf}&&&$c_1\circ c_3\circ c_1(K_{2,2})=\includegraphics[scale=\CBTableScale,valign=c,page=21]{Figures/Section_Exploring_LC_Equivalence_Classes_with_the_QASST.pdf}$\\
\includegraphics[scale=\CBTableScale,valign=c,page=9]{Figures/Section_Exploring_LC_Equivalence_Classes_with_the_QASST.pdf}&&\includegraphics[scale=\CBTableScale,valign=c,page=13]{Figures/Section_Exploring_LC_Equivalence_Classes_with_the_QASST.pdf}&&&$c_2\circ c_3\circ c_2(K_{2,2})=\includegraphics[scale=\CBTableScale,valign=c,page=22]{Figures/Section_Exploring_LC_Equivalence_Classes_with_the_QASST.pdf}$\\
&&&&&$c_1\circ c_4\circ c_1(K_{2,2})=\includegraphics[scale=\CBTableScale,valign=c,page=23]{Figures/Section_Exploring_LC_Equivalence_Classes_with_the_QASST.pdf}$\\
&&&&&$c_2\circ c_4\circ c_2(K_{2,2})=\includegraphics[scale=\CBTableScale,valign=c,page=24]{Figures/Section_Exploring_LC_Equivalence_Classes_with_the_QASST.pdf}$\\
\hline
\hline
c&1&c&1&1&c-c (invalid)\\
\includegraphics[scale=\CBTableScale,valign=c,page=7]{Figures/Section_Exploring_LC_Equivalence_Classes_with_the_QASST.pdf}&&\includegraphics[scale=\CBTableScale,valign=c,page=11]{Figures/Section_Exploring_LC_Equivalence_Classes_with_the_QASST.pdf}&&&\includegraphics[scale=\CBTableScale,valign=c,page=25]{Figures/Section_Exploring_LC_Equivalence_Classes_with_the_QASST.pdf}\\
\hline
sc&1&ss&2&2&sc-ss (invalid)\\
\includegraphics[scale=\CBTableScale,valign=c,page=6]{Figures/Section_Exploring_LC_Equivalence_Classes_with_the_QASST.pdf}&&\includegraphics[scale=\CBTableScale,valign=c,page=12]{Figures/Section_Exploring_LC_Equivalence_Classes_with_the_QASST.pdf}&\includegraphics[scale=\CBTableScale,valign=c,page=13]{Figures/Section_Exploring_LC_Equivalence_Classes_with_the_QASST.pdf}&&\includegraphics[scale=\CBTableScale,valign=c,page=28]{Figures/Section_Exploring_LC_Equivalence_Classes_with_the_QASST.pdf} \includegraphics[scale=\CBTableScale,valign=c,page=29]{Figures/Section_Exploring_LC_Equivalence_Classes_with_the_QASST.pdf}\\
\hline
ss&2&sc&1&2&ss-sc (invalid)\\
\includegraphics[scale=\CBTableScale,valign=c,page=8]{Figures/Section_Exploring_LC_Equivalence_Classes_with_the_QASST.pdf}&\includegraphics[scale=\CBTableScale,valign=c,page=9]{Figures/Section_Exploring_LC_Equivalence_Classes_with_the_QASST.pdf}&\includegraphics[scale=\CBTableScale,valign=c,page=10]{Figures/Section_Exploring_LC_Equivalence_Classes_with_the_QASST.pdf}&&&\includegraphics[scale=\CBTableScale,valign=c,page=26]{Figures/Section_Exploring_LC_Equivalence_Classes_with_the_QASST.pdf} \includegraphics[scale=\CBTableScale,valign=c,page=27]{Figures/Section_Exploring_LC_Equivalence_Classes_with_the_QASST.pdf}\\
\hline
\end{tabular}
\caption{Enumeration of the valid and invalid quotient graph symmetries based on $QASST(K_{2,2})$.}
\label{tab:K2,2_quotient_graph_symmetries}
\end{table*}

Of these 16 possible combinations for two adjacent quotient graphs, some will correspond to valid strong splits while others will be invalid in the sense of Figure~\ref{fig:quotient_graph_combinations_valid} (i.e.~combinations of the form sc-ss, ss-sc, or c-c). 
Table~\ref{tab:K2,2_quotient_graph_symmetries} enumerates all of these possibilities, valid and invalid.
Observe that we recover the 11 locally equivalent graphs enumerated in Adcock et al. We also include a transformation showing that the corresponding graph is indeed locally equivalent to $K_{2,2}$.
The remaining 5 invalid cases show that $\mu(K_{2,2})=5$, so that Equation~\ref{eq:count_QASST_equivalence} gives $\Phi(K_{2,2})=16-5=11$.
Furthermore, the graphs reconstructed from the split trees for these invalid cases coincide exactly with the LC orbit of $K_4$ pictured in Figure~\ref{fig:example_K4_orbit}.
In this example of $K_{2,2}$, the QASST equivalence class matches the LC equivalence, although this need not be true in general.

\subsection{Bounds on LC Equivalence Class Sizes Using QASST Equivalences}
\label{sect:upper_bound_on_LC_orbit_sizes}

Since LC equivalent graphs are necessarily QASST equivalent, Lemma~\ref{thm:counting_QASST_equivalence} immediately gives an upper bound on the size of the LC orbit.
In general, the QASST equivalence class can be partitioned into disjoint LC equivalence classes.
Hence, using non-LC equivalent graph representatives from each component of the partition, the combined orbits of these graphs exhaust the QASST equivalence class.
We state these facts here as corollaries of the preceding lemma.

\begin{corollary}
\label{thm:LC_orbit_upper_bound}
The size of the LC equivalence class of any graph $G$ is bounded above by the size of its QASST equivalence class:
\begin{eqnarray}
|{\mathcal O}(G)|&\leq&\Phi(G).
\end{eqnarray}
\end{corollary}
\begin{corollary}
\label{thm:LC_orbits_exhaust_QASST_equivalence}
If $G_1,\cdots,G_k$ are QASST equivalent graphs which are pairwise non-LC equivalent, and if every graph in this QASST equivalence class is LC equivalent to some graph from this list, then
\begin{eqnarray}
|{\mathcal O}(G_1)|+\cdots+|{\mathcal O}(G_k)|&=&\Phi(G_1).
\end{eqnarray}
\end{corollary}

The split decomposition of a graph can be used with combinatorial arguments to give an exhaustive enumeration of graphs in the same QASST equivalence class.
However, while Corollary~\ref{thm:LC_orbit_upper_bound} gives a useful upper bound, the existence of non-LC equivalent graphs which are QASST equivalent shows that this technique is insufficient for computing the size of the LC orbit.
Even so, computing the size of the QASST equivalence class is a useful starting point in combination with the result of Corollary~\ref{thm:LC_orbits_exhaust_QASST_equivalence}.
By verifying when two QASST equivalent graphs are non-LC equivalent, either with an invariant or through Bouchet's algorithm~\cite{bouchet1991efficient,bouchet1993recognizing}, we can leverage this upper bound by enumerating all LC orbits which partition the QASST equivalence class.

The brute force technique used to examine the QASST equivalence class of $K_{2,2}\cong C_4$ (Table~\ref{tab:K2,2_quotient_graph_symmetries}) is adequate for small cases like this example, but it does not give the general formula for a family of graphs.
However, we can extend this to more general cases by counting the sizes and combinations of quotient graph symmetry classes (c, sc, or ss) rather than enumerating specific quotient graphs.
Quotient graphs in the same symmetry class are isomorphic, as are the graphs reconstructed from these.
By showing that a representative of a symmetry class is LC equivalent to the chosen graph via an explicit transformation, we can extend this to the entire symmetry class merely by relabeling the vertices in the quotient graphs.
In this way, we will derive a lower bound for the size of the LC orbit that can be compared with the upper bound based on QASST equivalence.

The split decompositions of certain locally equivalent distance-hereditary graphs in particular have a high amount of symmetry. Since the vertices in our graphs are labeled, in many cases two locally equivalent graphs are isomorphic with a different choice of labels. This occurs because there are many distinct ways to have a star-spoke quotient graph merely by choosing a different node as the center of the star.
We exploit this for our purposes by counting the combinations of symmetry classes, which are determined by the QASST rather than the number of vertices.

We return to $K_{2,2}$ to illustrate an explicit example of this process.
Table~\ref{tab:K2,2_quotient_graph_symmetries} enumerates the symmetry classes of locally equivalent quotient graphs, while counting the number of distinct graphs which give this same symmetry.
Those combinations with the same symmetries yield isomorphic graphs, which can be verified by inspection in this case.
Furthermore, this table gives an explicit transformation (non-unique) in terms of compositions of primitive local complements, starting from $K_{2,2}$.

Clearly, the existence of an explicit transformation proves that each valid graph in Table~\ref{tab:K2,2_quotient_graph_symmetries} is locally equivalent to $K_{2,2}$. Hence, counting the number of possibilities gives a simple method for computing a lower bound on the size of the LC orbit. Our lower bound of 11 shown here coincides with the upperbound of $\Phi(K_{2,2})=11$, which is enough to confirm the size of the orbit.
Although we have already viewed the entire orbit of $K_{2,2}$, this technique of counting symmetries will also work for larger graphs where listing the whole LC orbit is not feasible, in particular for general cases such as $K_{n,m}$ (Figure~\ref{fig:Knm_split_decomposition}).

Observe that the explicit transformations for graphs in the same equivalence class of Table~\ref{tab:K2,2_quotient_graph_symmetries} follow a regular pattern. More generally, there is a relationship between the choice of center vertex for a star-spoke quotient graph and the corresponding transformation. Hence, rather than listing all possibilities, it is sufficient to give a single representative of the symmetry class; all remaining possibilities can be obtained by changing the index of vertex at the center of the star.
Thus, we obtain a generic expression for how to transform to any graph in this symmetry class.
By fixing a symmetry class for each quotient graph, the number of LC equivalent graphs with the same symmetry is the product of the sizes of the symmetry classes.
Computing the lower bound on the LC orbit then reduces to adding the totals.

\section{Explicit Formulas for LC Orbit Sizes of Selected Graph Families}
\label{sect:explicit_formulas_for_local_equivalence_classes}

One of Bouchet's original applications involving local complements was to count the number of graphs locally equivalent to a given graph~\cite{bouchet1993recognizing}. For example, he derived the following explicit formulas for the size of the local equivalence class of an arbitrary path graph $P_n$ and an arbitrary cycle graph $C_n$:
\begin{eqnarray}
|{\mathcal O}(P_n)|&=&\frac{\sqrt{3}}{6}\Big((1+\sqrt{3})^{n+1}-(1-\sqrt{3})^{n+1}\Big);\\
|{\mathcal O}(C_n)|&=&(1+\sqrt{3})^n+(1-\sqrt{3})^n-4(2^{n-1}+(-1)^n)/3.\nonumber\\
\end{eqnarray}
In the same spirit but with a different approach, we will derive explicit formulas for some families of distance-hereditary graph. Like the example of Table~\ref{tab:K2,2_quotient_graph_symmetries}, this approach is based on counting symmetries in the split decomposition.

\subsection{Complete Bipartite Graphs}
\label{sect:compelte_bipartite_graphs_LC_orbit_size}

Complete bipartite graphs have been shown to be locally equivalent to certain types of repeater graphs~\cite{tzitrin2018local} (specifically \textit{binary stars}, which consist of two star graphs with an additional edge connecting their centers). This is just one possible symmetry class that can be obtained using the QASST.

A generic complete bipartite graph $K_{n,m}$ has a split decomposition consisting of two star-center quotient graphs, with QASST shown in Figure~\ref{fig:Knm_split_decomposition}. Let $Q_n$ and $Q_m$ denote the left and right quotient graphs. Observe that each is a star graph with a split-node center and $n$ or $m$ leaf-node spokes, respectively.
In particular, $K_{n,m}$ is distance-hereditary by Theorem~\ref{thm:distance_hereditary_iff_star_complete_quotient}.
If either $n=1$ or $m=1$, then $K_{n,m}$ reduces to a star graph, and so we also assume that $n,m\geq2$.

Let $[n]=\{1,\cdots,n\}$ denote the set of leaf-node indices in $Q_n$ and $[m]=\{n+1,\cdots,n+m\}$ denote the set of leaf-node indices in $Q_m$.
Let $i^n\in[n]$ and $i^m\in[m]$ refer to a choice of leaf-node index from each quotient graph.
Thanks to the symmetry of the locally equivalent quotient graphs, the explicit LC transformations we introduce here can be written up to a choice of index from each quotient graph. A different choice yields a distinct but isomorphic graph.

Table~\ref{tab:complete_bipartite_quasst_equivalences} enumerates all combinations of LC equivalent quotient graphs which can be substituted into $QASST(K_{n,m})$ up to symmetry, including the invalid cases listed here in the first block of the table.
Counting the c-c, sc-ss, and ss-sc cases yields a total of $\mu(K_{n,m})=n+m+1$ invalid combinations. Hence, the size of the QASST equivalence class of $K_{n,m}$ is computed as
\begin{eqnarray}
\Phi(K_{n,m})&=&|{\mathcal O}(Q_n)||{\mathcal O}(Q_m)|-\mu(K_{n,m})\nonumber\\
&=&(n+2)(m+2)-(n+m+1)\nonumber\\
&=&nm+n+m+3.
\end{eqnarray}
This is an upper bound on the total number of graphs in the LC orbit of $K_{n,m}$.
However, explicit transformations for these remaining $nm+n+m+3$ cases are listed in the second block of Table~\ref{tab:complete_bipartite_quasst_equivalences}, showing that each of these is obtainable using local complements and thus exhausting the LC orbit.
Since Table~\ref{tab:complete_bipartite_quasst_equivalences} provides an explicit local complement transformation for each of the remaining $nm+n+m+3$ valid cases, the upper bound is tight, and thus $|O(K_{n,m})|=\Phi(K_{n,m})$.
Any star-spoke quotient graph can use any choice of leaf-node for the center of the star without changing the symmetry. In terms of the transformation from $K_{n,m}$, the choice of index on the center leaf-node matches the index $i^n\in[n]$ or $i^m\in[m]$ of the local complements used in the transformation.

\begin{table*}[t]
\centering
\begin{tabular}{|cc|cc|c|c|c|}
\hline
$Q_n$&(count)&$Q_m$&(count)&Total&Transformation from $K_{n,m}$&Number of Edges\\
\hline
c&1&c&1&1&invalid&N/A\\
sc&1&ss&$m$&$m$&invalid&N/A\\
ss&$n$&sc&1&$n$&invalid&N/A\\
\hline
sc&1&sc&1&1&$\text{id}(K_{n,m})$&$nm$\\
sc&1&c&1&1&$c_{i^n}(K_{n,m})$&$nm+\frac{m(m-1)}{2}$\\
c&1&sc&1&1&$c_{i^m}(K_{n,m})$&$nm+\frac{n(n-1)}{2}$\\
ss&$n$&c&1&$n$&$c_{i^n}\circ c_{i^m}(K_{n,m})$&$n+m-1+\frac{m(m-1)}{2}$\\
c&1&ss&$m$&$m$&$c_{i^m}\circ c_{i^n}(K_{n,m})$&$n+m-1+\frac{n(n-1)}{2}$\\
ss&$n$&ss&$m$&$nm$&$c_{i^n}\circ c_{i^m}\circ c_{i^n}(K_{n,m})$&$n+m-1$\\
\hline
\end{tabular}
\caption{
Enumeration of the combinations of symmetric quotient graphs in $QASST(K_{n,m})$, where each component is either complete (c), star-center (sc), or star-spoke (ss). The first block represents the $n+m+1$ cases which do not preserve strong splits; these are not LC equivalent to $K_{n,m}$. The remaining $nm+n+m+3$ cases are all obtainable from $K_{n,m}$ with local complements using the given transformations. In this notation, $i^n\in[n]=\{1,\cdots,n\}$ denotes any index on the leaf-nodes from $Q_n$, and $i^m\in[m]=\{n+1,\cdots,n+m\}$ denotes any index on the leaf-nodes from $Q_m$. The resulting graphs with the same symmetry will be isomorphic, with the center of the star-spoke quotient graph matching the vertex $i^n$ in $Q_n$ and $i^m$ in $Q_m$.
The number of edges in the corresponding graph is also given for each valid case.
}
\label{tab:complete_bipartite_quasst_equivalences}
\end{table*}

The QASST symmetry classes obtained here are consistent with those for $K_{2,2}$  enumerated in Table~\ref{tab:K2,2_quotient_graph_symmetries}.
These facts give the following general result regarding the size of the LC orbit for any complete bipartite graph.

\begin{theorem}
If $n,m\geq2$, then the local equivalence class of the complete bipartite graph $K_{n,m}$ has size
\begin{eqnarray}
|{\mathcal O}(K_{n,m})|&=&nm+n+m+3.
\end{eqnarray}
\end{theorem}

For each graph $G\in{\mathcal O}(K_{n,m})$, Table~\ref{tab:complete_bipartite_quasst_equivalences} also includes the number of edges in $G$ based on the QASST decomposition. Since each graph in this LC orbit contains $|V(G)|=n+m$ vertices, we know that $|E(G)|\geq m+n-1$ at minimum, which happens if and only if $G$ is a tree.
This occurs when $Q_n$ and $Q_m$ are both star-spoke, and hence any member of this symmetry class is a minimal edge representative of the local equivalence class. We state this here as a theorem.

\begin{theorem}
Let $G\in{\mathcal O}(K_{n,m})$ be locally equivalent to a complete bipartite graph, where $n,m\geq2$. $G$ is a minimal edge representative of the local equivalence class if and only if $G$ is a \textit{binary star}, which occurs when $QASST(G)$ matches the last row of Table~\ref{tab:complete_bipartite_quasst_equivalences}.
In this case, the number of edges is
\begin{eqnarray}
|E(G)|&=&n+m-1.
\end{eqnarray}
\end{theorem}

In addition to counting edges, we may also identify the maximum vertex degree of any graph in ${\mathcal O}(K_{n,m})$ using the QASST decomposition.
We do this by computing the maximum vertex degree of leaf-nodes in $Q_n$ and $Q_m$ and then comparing these values.
These results are summarized in Table~\ref{tab:complete_bipartite_max_vetex_degrees}.

\begin{table*}[t]
\centering
\begin{tabular}{|c|c|c|c|c|}
\hline
$Q_n$&$Q_m$&Max Degree in $Q_n$&Max Degree in $Q_m$&Max Degree in $G$\\
\hline
sc&sc&$m$&$n$&$\text{max}\{n,m\}$\\
sc&c&$m$&$n+(m-1)$&$n+m-1$\\
c&sc&$(n-1)+m$&$n$&$n+m-1$\\
ss&c&$n$&$n+(m-1)$&$n+m-1$\\
c&ss&$(n-1)+m$&$m$&$n+m-1$\\
ss&ss&$n$&$m$&$\text{max}\{n,m\}$\\
\hline
\end{tabular}
\caption{
The maximum vertex degree for graphs locally equivalent to a complete bipartite graph in each of the six valid cases of Table~\ref{tab:complete_bipartite_quasst_equivalences}.
These are obtained by comparing the maximum vertex degrees of leaf-nodes from $Q_n$ and $Q_m$.
}
\label{tab:complete_bipartite_max_vetex_degrees}
\end{table*}

Since the graphs with which we work have labeled vertices, each graph in the LC orbit is considered distinct. However, many of these graphs will be isomorphic. If we instead consider the local equivalence class of unlabeled graphs, then isomorphic graphs would be taken to be equal. The size of the LC orbit up to isomorphism has previously been examined in the literature~\cite{adcock2020mapping}.
We end this subsection by deriving a few facts about the size of the LC orbit for unlabeled complete bipartite graphs.

If $n\neq m$, then graphs from different QASST symmetry classes in Table~\ref{tab:complete_bipartite_quasst_equivalences} cannot be isomorphic.
One way to see this immediately is by observing that the graphs in each symmetry class have a different number of edges.
Observe that the first three classes contain only a single graph (the sc-sc, sc-c, and c-sc cases), while the remaining three classes contain graphs for every combination of center nodes for star-spoke quotient graphs (the ss-c, c-ss, and ss-ss cases); in the latter case, these graphs are isomorphic via a different choice of center.
Hence, when $n\neq m$, the size of the LC orbit up to isomorphism is 6.

When $n=m$, observe that cases 2 and 3 from Table~\ref{tab:complete_bipartite_quasst_equivalences} are symmetric (sc-c and c-sc), as are cases 4 and 5 (c-ss and ss-c).
The graphs in these symmetry classes will be isomorphic by mapping the leaf-nodes from $Q_n$ to $Q_m$ and vice-versa.
Hence, when $n=m$, the size of the LC orbit up to symmetry is only 4.
We conclude with a statement of these facts as a theorem.

\begin{theorem}
Consider the LC orbit of the \text{unlabeled} complete bipartite graph $K_{n,m}$, wherein isomorphic graphs are taken to be equal. Assume that $n,m\geq2$.
If $n\neq m$, then $|{\mathcal O}(K_{n,m})/\cong|=6$, matching the six symmetry classes of Table~\ref{tab:complete_bipartite_quasst_equivalences}.
If $n=m$, then $|{\mathcal O}(K_{n,m})/\cong|=4$; in this case, symmetry classes 2 and 3 are identical, as are classes 4 and 5.
\end{theorem}

\subsection{Complete $k$-Partite Graphs}
\label{sect:complete_k-partite_graphs}

\begin{figure*}[tp]
\centering
\begin{subfigure}{0.55\textwidth}
\centering
\includegraphics[width=0.8\textwidth,page=1]{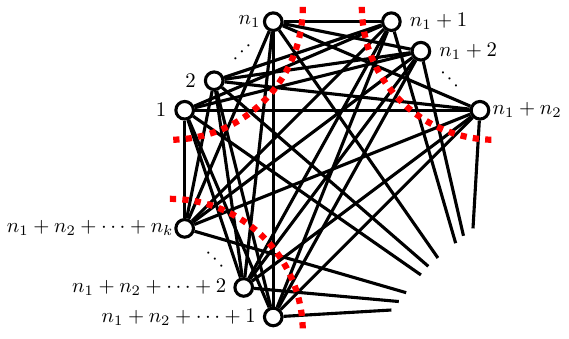}
\caption{Complete $k$-partite graph $K_{n_1,n_2,\cdots,n_k}$}
\label{fig:complete_k-partite_generic_graph}
\end{subfigure}\begin{subfigure}{0.45\textwidth}
\centering
\includegraphics[width=0.8\textwidth,page=2]{Figures/Section_Explicit_Formulas_for_LC_Orbits.pdf}
\caption{Split decomposition $QASST(K_{n_1,n_2,\cdots,n_k})$}
\label{fig:complete_k-partite_generic_QASST}
\end{subfigure}

\bigskip

\begin{subfigure}{0.55\textwidth}
\centering
\includegraphics[width=0.8\textwidth,page=3]{Figures/Section_Explicit_Formulas_for_LC_Orbits.pdf}
\caption{Clique-star $CS^1_{n_1,n_2,\cdots,n_k}$}
\label{fig:clique-star_generic_graph}
\end{subfigure}\begin{subfigure}{0.45\textwidth}
\centering
\includegraphics[width=0.8\textwidth,page=4]{Figures/Section_Explicit_Formulas_for_LC_Orbits.pdf}
\caption{Split decomposition $QASST(CS^1_{n_1,n_2,\cdots,n_k})$}
\label{fig:clique-star_generic_QASST}
\end{subfigure}

\bigskip

\begin{subfigure}{0.55\textwidth}
\centering
\includegraphics[width=0.8\textwidth,page=5]{Figures/Section_Explicit_Formulas_for_LC_Orbits.pdf}
\caption{Multi-leaf repeater $MR_{n_1,n_2,\cdots,n_k}$}
\label{fig:multi-leaf_repeater_generic_graph}
\end{subfigure}\begin{subfigure}{0.45\textwidth}
\centering
\includegraphics[width=0.8\textwidth,page=6]{Figures/Section_Explicit_Formulas_for_LC_Orbits.pdf}
\caption{Split decomposition $QASST(MR_{n_1,n_2,\cdots,n_k})$}
\label{fig:multi-leaf_repeater_generic_QASST}
\end{subfigure}

\caption{
Split decompositions of a generic complete $k$-partite graph (\ref{fig:complete_k-partite_generic_graph},\ref{fig:complete_k-partite_generic_QASST}), clique-star (\ref{fig:clique-star_generic_graph},\ref{fig:clique-star_generic_QASST}), and multi-leaf repeater graph (\ref{fig:multi-leaf_repeater_generic_graph},\ref{fig:multi-leaf_repeater_generic_QASST}).
These graphs are QASST equivalent with quotient graphs $Q_0,Q_1,\cdots,Q_k$.
In each case, the strong split tree is a star, consisting of a central $Q_0$ connected to $Q_1,\cdots,Q_k$ arranged in a ring around the center.
A complete $k$-partite graph occurs when $Q_0$ is complete and each $Q_1,\cdots,Q_k$ is star-center.
A clique-star occurs when $Q_0$ is star-center (pointing towards some $Q_{r>0}$) and each $Q_1,\cdots,Q_k$ is complete.
A multi-leaf repeater graph occurs when $Q_0$ is complete and each $Q_1,\cdots,Q_k$ is star-spoke.
Other QASST equivalent combinations are also possible, but these represent special cases.
}
\label{fig:complete_k_partite_CS_MLR_QASST_decomposition}
\end{figure*}

Generalizing from the complete bipartite case, we consider the complete $k$-partite graph $K_{n_1,\cdots,n_k}$, where $k\geq3$ and $n_i\geq2$ (Figure~\ref{fig:complete_k-partite_generic_graph}).
The split decomposition and QASST for a generic graph of this form is shown in Figure~\ref{fig:complete_k-partite_generic_QASST}, consisting of a ring of star-center quotient graphs $Q_1,\cdots,Q_k$ arranged around a central complete graph $Q_0$.
The quotient graphs $Q_1,\cdots,Q_k$ each contain $n_1,\cdots,n_k$ leaf-nodes, respectively, connected to a single split-node as shown.
The central quotient graph $Q_0$ is complete on $k$ split-nodes.

In much the same way as the complete bipartite graph, we may compute upper and lower bounds on the size of the LC orbit by examining the symmetries of the QASST decomposition. However, there are many more cases to consider and the full analysis is quite technical.
We leave the details to the appendix but highlight the results here.

For a generic complete $k$-partite graph $K_{n_1,\cdots,n_k}$, the size of the QASST equivalence class is derived in \ref{sect:complete_k-partite_general_QASST_decomposition} and reproduced here.
\begin{equation}
\Phi(K_{n_1,\cdots,n_k})=\prod_{i=1}^k(n_i+1)+2\sum_{j=1}^k\left[\prod_{i=1\neq j}^k(n_i+1)\right] \ .\tag{\ref{eq:complete_k-partite_QASST_upper_bound}}
\end{equation}
This is an upper bound on the number of graphs locally equivalent to $K_{n_1,\cdots,n_k}$. A lower bound on the size of the LC orbit is derived in \ref{app:complete_k-partite_LC_orbit_lower_bound} and we prove that this lower bound coincides with the size of the full orbit (Theorem~\ref{thm:complete_k-partite_LC_orbit_size}).
\begin{equation}
|{\mathcal O}(K_{n_1,\cdots,n_k})|=\underbrace{\prod_{i=1}^k(n_i+1)}_{\text{even products}}+\sum_{j=1}^k\left(\prod_{i\in[k]\setminus\{j\}}(n_i+1)\right) \ .\tag{\ref{eq:complete_k-partite_LC_orbit_size}} 
\end{equation}
Note that we specify \textit{even products} of $\prod_{i=1}^k(n_i+1)$ in this formula, meaning only those terms from the fully expanded form which use an even number of $n_i$ such as $n_1n_2$ or $n_1n_3n_7n_8$. This includes a product of zero $n_i$ (a product over an empty set), which is 1.

Through a further analysis of the symmetry classes in the QASST decomposition, we identify the structure of a minimal edge representative of the LC orbit and an explicit formula for the number of edges (Theorem~\ref{thm:complete_k-partite_minimal_edge_representative}). However, which graph in the orbit corresponds to a minimal edge representative depends on both the size of $k$ relative to $n_1,\cdots,n_k$ as well as the parity of $k$; it will not be the same graph in all cases. A full analysis is outlined in \ref{app:edge_counting}.
Similarly, we also examine the maximum vertex degree of those graphs locally equivalent to $K_{n_1,\cdots,n_k}$ in Section~\ref{app:max_vertex_degrees}.
We derive a general formula for a graph minimizing the maximum vertex degree across the LC orbit based on the values of $k,n_1,\cdots,n_k$ (Theorem~\ref{thm:ckp_min_mvd}).

Finally, we also consider the question of isomorphic graphs in the LC orbit of $K_{n_1,\cdots,n_k}$ in \ref{app:LC_and_isomorphisms}.
We only examine the special case when $n_1=\cdots=n_k$, for which we derive a formula for the number of distinct graphs in the LC orbit up to isomorphism (Theorem~\ref{thm:complete_k-partite_LC_orbit_up_to_isomorphism}), restated here:
\begin{equation}
\left|{\mathcal O}(K_{n_1,\cdots,n_k})/\cong\right|=\lfloor\tfrac{k}{2}\rfloor+k+1. \tag{\ref{eq:complete_k-partite_LC_orbit_size_up_to_isomorphism}}
\end{equation}

\subsection{Clique-Stars}
\label{clique-stars}

Unlike the case of a complete bipartite graph, the LC equivalence class of a complete $k$-partite graph does not exhaust the QASST equivalence class: $|\mathcal O(K_{n_1,\cdots,n_k})|<\Phi(K_{n_1,\cdots,n_k})$.
In other words, there exist graphs QASST equivalent to $K_{n_1,\cdots,n_k}$ which are not LC equivalent.
One of these graphs, and the graph we take as the representative for this other local equivalence class, is the \textit{clique-star} $CS^r_{n_1,\cdots,n_k}$ (Figure~\ref{fig:clique-star_generic_graph}).
We show in \ref{app:LC_non-equivalence} that these graphs are not LC equivalent (Theorem~\ref{thm:LC_non-equivalence}).

A clique-star can be viewed as a generalization of a star graph (recall that a \textit{clique} refers to a fully connected subset of nodes in a graph).
Rather than a center node with several spokes, instead we have a clique-center surrounded by a number of clique-spokes.
The cliques themselves are enumerated $1,\cdots,k$ and contain $n_1,\cdots,n_k$ vertices, respectively.
In the notation $CS^r_{n_1,\cdots,n_k}$, $1\leq r\leq k$ is taken to be the index of the central-clique.
Although the graphs themselves are distinct, using any clique as the center yields a locally equivalent graph. Hence, we usually default to $r=1$ in figures.

The clique-star has the same strong splits as a complete $k$-partite graph, but the quotient graphs are different (Figure~\ref{fig:clique-star_generic_QASST}). Each quotient graph $Q_1,\cdots,Q_n$ is a complete graph, with $n_i$ leaf-nodes and a single split-node.
The central quotient graph $Q_0$ consists of $k$ split-nodes, and is star-center pointing towards $Q_r$. We denote this by $Q_0=\text{sc}_r$.

Since they are QASST equivalent, we have that $\Phi(CS^r_{n_1,\cdots,n_k})=\Phi(K_{n_1,\cdots,n_k})$, with explicit formula given by Equation~\ref{eq:complete_k-partite_QASST_upper_bound}.
As with the complete $k$-partite case, we derive a lower bound on the size of the LC orbit in \ref{app:clique-star_LC_orbit_lower_bound}. We then show that this coincides with the size of the full orbit (Theorem~\ref{thm:clique-star_LC_orbit_size}).
\begin{equation}
|{\mathcal O}(CS^r_{n_1,\cdots,n_k})|=\underbrace{\prod_{i=1}^k(n_i+1)}_{\text{odd products}}+\sum_{j=1}^k\left(\prod_{i\in[k]\setminus\{j\}}(n_i+1)\right) \ .\tag{\ref{eq:clique-star_LC_orbit_size}}
\end{equation}
In the expression above, \textit{odd products} in $\prod_{i=1}^k(n_i+1)$ refers to those terms in the fully expanded form using only an odd number of $n_i$, such as $n_1$ or $n_1n_2n_3$.

As in the preceding case, we also identify the structure of a minimal edge representative of the LC orbit through an analysis of the QASST symmetry classes (Theorem~\ref{thm:clique-star_minimal_edge_representative}). There are a few different possibilities depending on the size and parity of $k$ relative to the values of $n_1,\cdots,n_k$, but each of these cases is also outlined in \ref{app:edge_counting}.
Likewise, we derive a formula for graph with minimal maximum vertex degree across the LC orbit of $CS^r_{n_1,\cdots,n_k}$ (Theorem~\ref{thm:cs_min_mvd}), leaving the technical details to \ref{app:max_vertex_degrees}.

As for the consideration of graphs in the LC orbit of $CS^r_{n_1,\cdots,n_k}$ up to isomorphism, this question is also explored in \ref{app:LC_and_isomorphisms}. In the special case when $n_1=\cdots=n_k$, we prove in Theorem~\ref{thm:clique-star_LC_orbit_up_to_isomorphism} that the number of distinct isomorphism classes within the local equivalence class is
\begin{equation}
\left|{\mathcal O}(K_{n_1,\cdots,n_k})/\cong\right|=\lceil\tfrac{k}{2}\rceil+k. \tag{\ref{eq:clique-star_LC_orbit_up_to_isomorphism}}
\end{equation}

\subsection{Multi-Leaf Repeater Graphs}
\label{sect:multi-leaf_repeater_graphs}

Together, clique-stars and complete $k$-partite graphs exhaust the QASST equivalence class. This can be seen by the fact that
\begin{eqnarray*}
|{\mathcal O}(K_{n_1,\cdots,n_k})|+|{\mathcal O}(CS^r_{n_1,\cdots,n_k})|&=&\Phi(K_{n_1,\cdots,n_k})\\
&=&\Phi(CS^r_{n_1,\cdots,n_k}).
\end{eqnarray*}
However, we mention here another interesting family of graphs contained in the same QASST equivalence class.
These are \textit{multi-leaf repeater graphs} $MR_{n_1,\cdots,n_k}$, which are generalizations of the standard repeater graphs introduced in~\cite{azuma2015all}.

A multi-leaf repeater consists of a complete graph on $k$ vertices, and each of these vertices has an additional $n_i-1$ leaves attached to it (Figure~\ref{fig:multi-leaf_repeater_generic_graph}).
This is in contrast to standard repeater graph $R_k$, wherein only a single additional leaf is attached to each vertex of the complete graph.
In this sense, a repeater graph can be regarded as a special case of multi-leaf repeater graph where each $n_i=2$: $R_k=MR_{2,\cdots,2}$.

A multi-leaf repeater has the same strong split tree as the clique-star and complete $k$-partite graph but with a different combination of quotient graphs $Q_0,Q_1,\cdots,Q_k$ (Figure~\ref{fig:multi-leaf_repeater_generic_QASST}).
The central quotient graph $Q_0$ is complete on $k$ split-nodes.
Each quotient graph $Q_1,\cdots,Q_k$ is star-spoke with $n_i$ leaf-nodes and a single split-node.

The key observation regarding the local equivalence of multi-leaf repeater graphs is that, depending on the parity of $k$, they can be locally equivalent to a complete $k$-partite graph (when $k$ is even) or to a clique-star (when $k$ is odd) (Theorem~\ref{thm:multi-leaf_repeater_graph_LC_equivalence}).
We provide more details about the relationship between multi-leaf repeaters and these other two graph families in \ref{app:local_equivalence_of_MLR}, including explicit LC transformations between these graphs.
Often, graphs of this form will be minimal edge representatives of the local equivalence class, but not always. Some examples are included in Figure~\ref{fig:complete_k-partite_LC_orbit_MER_examples}.

\section{Conclusion}
\label{sect:conclusion}

This paper introduced a new method to count the size of graph local equivalence classes.
The key technique for doing this is the examination of quotient graph symmetry classes in the quotient-augmented strong split tree, a graph-labeled tree used to represent the split decomposition of a graph.
This method is especially promising for distance-hereditary graphs given the simplicity of the quotient graphs in these cases.
We illustrated the utility of this technique by fully classifying the local equivalence class for several interesting families of distance-hereditary graph, including complete bipartite graphs, complete $k$-partite graphs, clique-stars, and multi-leaf repeater graphs.
Furthermore, we derived explicit formulas counting the number of locally equivalent graphs in each case, along with explicit LC transformations between graphs in the same LC orbit, as well as finding minimal edge and minimal maximum vertex degree representatives from each orbit.
These results represent a significant step forward for an active topic of research, especially given the connection to and recent developments in quantum information theory.

Other interesting distance-hereditary graphs that could be classified using this method in future work include bipartite multi-leaf repeater graphs, fully-connected layer graphs, and caterpillar graphs.
It would also be interesting to extend these techniques to non-distance-hereditary cases through an examination of the local equivalence classes for prime graphs.
Moreover, the combination of Algorithms~\ref{alg:induced_QASST} and \ref{alg:LC_propagation} give a method for computing the QASST of a \textit{vertex minor} (a graph is a \textit{vertex minor} of another graph if it can be obtained through a sequence of local complements and vertex deletions~\cite{oum2005rank}).
It is challenging in general to identify when a graph is a vertex minor, and so these techniques could be employed to study this related problem.

\section*{Acknowledgments}
\label{sect:acknowledgements}

We thank Kenneth Goodenough, Kyle Stanley Grant, and Dan Browne for valuable and illuminating discussions. This work was supported in part by the JST Moonshot R\&D Program (Grant Nos. JPMJMS2061 \& JPMJMS226C). SN  also acknowledges support from the New Energy and Industrial Technology Development Organization (NEDO, JPNP23003) and a JSPS Overseas Research Fellowship.


\bibliographystyle{abbrv} 
\bibliography{references.bib}

\appendix
\section{Review of Graph Theory Fundamentals}
\label{app:graph_theory_fundamentals}

A \textbf{graph} $G=(V,E)$ consists of a set of \textbf{vertices} $V=\{v_1,\cdots,v_n\}$ and a set of \textbf{edges} $E=\{e_1,\cdots,e_m\}$, where each edge $e_{\ell}=(v_i,v_j)$ denotes a pair of connected vertices. Unless otherwise stated, we will assume that all graphs we work with are \textbf{simple}, meaning there do not exist self-loops ($(v_i,v_i)\notin E$) and there do not exist multi-edges between vertices (i.e.~there exists at most one edge $(v_i,v_j)\in E$ between any pair of vertices $v_i,v_j\in V$.)
We assume that edges are not directed (i.e.~$(v_i,v_j)=(v_j,v_i)$).
If there exists an edge $e=(v_i,v_j)\in E$, then the vertices $v_i,v_j\in V$ are said to be {\bf adjacent}, and the edge $e=(v_i,v_j)$ is said to be {\bf incident} to both $v_i$ and $v_j$.
We consider only finite graphs.
Graphs need not be connected in general, but in this paper we restrict ourselves to connected graphs.

For a general graph $G$, we will denote the set of vertices as $V(G)$ and the set of edges as $E(G)$.
A \textbf{subgraph} $H$ of $G$ is a graph for which $V(H)\subseteq V(G)$ and $E(H)\subseteq E(G)$, noting that $e=(v_i,v_j)\in E(H)$ implies that $v_i,v_j\in V(H)$.
Given a subset of vertices $S\subseteq V(G)$, the \textbf{induced subgraph} $G[S]$ of $G$ is the subgraph for which $V(G[S])=S$ and $E(G(S))=\{(v_i,v_j)\in E(G):v_i,v_j\in S\}$. That is to say, $G(S)$ inherits any edges in $G$ between the vertices in $S$. A subgraph induced by a subset of edges can be defined in a similar way.

Given a node $v_i\in V(G)$, the \textbf{degree} of this vertex $\text{deg}(v_i)\in{\mathbb Z}_{\geq0}$ is the number of edges in $E(G)$ incident to $v_i$. The \textbf{neighborhood} of $v_i$ in $G$ is the subgraph $N_G(v_i)$ of $G$ induced by the vertices adjacent to $v_i$: $V(N_G(v_i))=\{v_j\in V(G):(v_i,v_j)\in E(G)\}$.
With this definition, $|V(N_G(v_i))|=\text{deg}(v_i)$.
We choose to state the parent graph $G$ in $N_G(v)$ explicitly to help distinguish between vertices with the same label from different graphs.

A \textbf{walk} between two vertices $v,v'\in V(G)$ is an ordered sequence of $\ell$ edges denoted $(v_0,v_1),(v_1,v_2),\cdots,(v_{\ell-1},v_{\ell})$ between $\ell+1$ pairs of adjacent vertices, where $v=v_0$ and $v'=v_\ell$. In a walk, edges and vertices need not be distinct. The length of the walk is $\ell$, the number of edges.
A \textbf{trail} is a walk in which all edges are distinct, but not necessarily vertices.
A \textbf{path} is a walk in which both vertices and edges are distinct.
A \textbf{cycle} is a trail in which vertices are distinct except for the first and the last.
A graph is \textbf{acyclic}, or a \textbf{tree}, if it contains no subgraphs which are cycles. Given a tree $T$, the degree 1 vertices in $T(V)$ are known as the \textbf{leaves} of the tree, while the other vertices are said to be \textbf{internal}.

A graph is called \textbf{connected} if there exists a path between every pair of vertices.
Given any two vertices $v,v'\in V(G)$ in a connected graph $G$, their \textbf{distance} $\text{dist}(v,v')$ is defined to be the length of the shortest path between them.
A graph $G$ is called \textbf{distance-hereditary} (also called \textbf{completely separable} or \textbf{totally decomposable}) provided that, for any connected induced subgraph $H$ of $G$, the distance between any two vertices $v,v'\in V(H)$ is also the distance between these same two vertices in $G$~\cite{howorka1977characterization,bandelt1986distance,hammer1990completely}.
An alternative characterization of graphs with this property is provided by the \textbf{rank-width} of a graph~\cite{oum2005rank,oum2006approximating,oum2017rank}; Oum proved that a connected graph is distance-hereditary if and only if it has rank-width 1~\cite{oum2005rank}. This fact plays a role in the proof of Theorem~\ref{thm:distance_hereditary_iff_star_complete_quotient}.

It is often useful to examine specific paths and cycles within a graph, and hence we provide the following alternative characterization of these in terms of subgraphs.
A \textbf{path} $P$ of length $\ell$ between two vertices $v,v'\in V(G)$ is a subgraph of $G$ with $\ell+1$ distinct vertices that can be sequentially ordered $V(P)=\{v=v_0,v_1,\cdots,v_{\ell-1},v_{\ell}=v'\}$
and $\ell$ edges between subsequent vertices $E(P)=\{(v_{0},v_{1}),(v_{1},v_{2}),\cdots,(v_{\ell-1},v_{\ell})\}$.
A \textbf{cycle} $C$ of length $\ell$ based at $v\in V(G)$ is a subgraph of $G$ defined as above in the same way as a path of length $\ell$, except beginning and ending at the same vertex $v=v_0=v_{\ell}$ (so there are $\ell$ edges between $\ell$ distinct vertices in total). Note that any vertex $v\in V(C)$ can be chosen equivalently as the base point of the cycle, but it is often useful to compare different cycles through a fixed point.
Furthermore, deleting any edge $e=(v_i,v_j)$ in a cycle $C$ of length $\ell$ breaks the cycle, resulting in a path of length $\ell-1$ from $v_i$ to $v_j$.

A graph $G$ is called \textbf{bipartite} if the vertex set can be partitioned into two disjoint, nonempty subsets $V(G)=U\sqcup V$ such that there do not exist any edges between the vertices in $U$, nor do there exist any edges between the vertices in $V$. In other words, for all edges $e=(v_i,v_j)\in E(G)$, if $v_i\in U$, then $v_j\in V$ and vice versa.
More generally, a graph $G$ is called \textbf{$k$-partite} if its vertex set can be partitioned into $k$ nonempty, disjoint subsets $V(G)=U_1\sqcup\cdots\sqcup U_k$ such that there do not exist any edges between the vertices in each subset $U_i$.
Note that any $k$-partite graph is automatically $k+1$-partite provided there exists a subset $U_i$ with at least two vertices that can be further subdivided into $U_i=U_i'\sqcup U_i''$.

A graph $G$ is called \textbf{complete} if all vertices are adjacent.
A subset of vertices $S\subseteq V(G)$ is called a \textbf{clique} in $G$ if all of the vertices in $S$ are adjacent. Equivalently, the induced subgraph $G(S)$ is complete.

Two graphs $G$ and $G'$ are called \textbf{isomorphic}, denoted $G\cong G'$, if there exists an edge-preserving bijection of vertex sets $\phi:V(G)\rightarrow V(G')$. By edge preserving, we mean that $(v_1,v_2)\in E(G)$ if and only if $(\phi(v_1),\phi(v_2))\in E(G')$. If such a map exists, we call $\phi$ a \textbf{graph isomorphism} of $G$ and $G'$ and write $G'=\phi(G)$.
In this paper, we will primarily work with \textbf{labeled graphs}, meaning that each vertex has a distinct label, usually consecutive integers.
In this context, isomorphic graphs with different labels will be considered distinct from each other.

\section{Complete $k$-partite and Clique-Star QASST Structure}
\label{app:complete_k-partite_clique-star_qasst_structure}

In the next few sections, our goal is to build up to a general classification of the LC orbits of complete $k$-partite graphs and clique-stars (Figures~\ref{fig:complete_k-partite_generic_graph} and \ref{fig:clique-star_generic_graph}), two classes of graph which are QASST equivalent but not LC equivalent.
This will also account for the class of multi-leaf repeater graphs, which we will show are always LC equivalent to one of these two cases.
We start with an analysis of these graphs' QASST structure.

\subsection{Basic Structure and Notation}

Let $G$ be any graph QASST equivalent to a complete $k$-partite graph $K_{n_1,\cdots,n_k}$, including but not limited to a clique-star $CS^r_{n_1,\cdots,n_k}$ or multi-leaf repeater graph $MR_{n_1,\cdots,n_k}$.
The total number of vertices in $G$ is $n=n_1+\cdots+n_k$. It has a QASST consisting of quotient graphs $Q_0,Q_1,\cdots,Q_k$ as illustrated in the examples of Figure~\ref{fig:complete_k_partite_CS_MLR_QASST_decomposition}, with the components $Q_1,\cdots,Q_k$ arranged around $Q_0$.
We will refer to $Q_i$ as the $i^{\text{th}}$ component of $QASST(G)$, where $n_1,\cdots,n_k$ denote the number of leaf-nodes in $Q_1,\cdots,Q_k$, respectively.
$Q_0$ is referred to as the central quotient graph and it consists of $k$ split-nodes.
Assume $k\geq3$ and $n_i\geq2$.
Let $[k]=\{1,\cdots,k\}$ denote the ordered set of indices between 1 and $k$ (the indices of the quotient graphs of $G$, excluding $Q_0$).
Suppose that the indices of the $n=n_1+\cdots+n_k$ vertices are ordered sequentially, continuing from each of the $k$ components as in the examples of Figure~\ref{fig:complete_k_partite_CS_MLR_QASST_decomposition}. Denote the subset of vertex indices in each component as $[n_1],\cdots,[n_k]$.
Hence, $[n_1]=\{1,\cdots,n_1\}$, $[n_2]=\{n_1+1,\cdots,n_1+n_2\}$ and so forth.
For any $n_i$, the general expression for the subset of indices is $[n_i]=\left\{\left(\sum_{\ell=1}^{i-1}n_\ell\right)+1,\cdots,\left(\sum_{\ell=1}^{i-1}n_\ell\right)+n_i\right\}$.

In the discussion which follows, many of the LC equivalent graphs we consider will be isomorphic up to a choice of index from $[n_i]$.
In order to give an explicit transformation between LC equivalent graphs, it is necessary to specify a sequence of local complements of the form $c_{\ell}$ for some $\ell\in[n_i]$, but the exact choice is arbitrary up to relabeling (i.e.~a different choice of index in $[n_i]$ yields an isomorphic graph).
In such cases, we will use a superscript with the notation $c_{\ell}^{n_i}$ to denote a local complement with respect to any choice of index $\ell\in[n_i]$. We will also write $\bigcomp_{i\in I}f_i$ to denote a composition of functions over a set of indices $I$.

Finally, some transformations involve an operation known as an \textbf{edge pivot}. Any edge $(i,j)\in E(G)$ defines an edge pivot on $G$ via a composition of three primitive local complements: $\text{ep}(i,j)=c_i\circ c_j\circ c_i$. Note that the order of the endpoints of the edge is irrelevant; if $(i,j)=(j,i)\in E(G)$, then $\text{ep}(i,j)=\text{ep}(j,i)$.
In some explicit expressions, the component $Q_i$ of the graph matters, but not the exact index $\ell\in[n_i]$. When such cases require an edge pivot between nodes from two different components with indices $i,j\in[k]$, we introduce the following general expression $\text{ep}([n_i],[n_j])$ below.
This definition allows for generic formulas involving edge-pivots, which may not correspond to existing edges in some cases.
\begin{eqnarray}\label{eq:component_edge_pivot}
\text{ep}([n_i],[n_j])&=&\begin{cases}c_{\ell}^{n_i}\circ c_{\ell'}^{n_j}\circ c_{\ell}^{n_i}&\text{if $i\neq j$ and $(\ell,\ell')\in E(G)$}\\\text{id}&\text{if $i=j$ or $(\ell,\ell')\notin E(G)$}\end{cases}\nonumber\\
\end{eqnarray}
Equation~\ref{eq:component_edge_pivot} is used to define a formula on the QASST which has a non-unique expression; a different choice of indices $\ell\in[n_i]$ and $\ell'\in[n_j]$ would give a different sequence of local complements, but these would define the same transformation on the QASST.

\subsection{General QASST Decomposition}
\label{sect:complete_k-partite_general_QASST_decomposition}

As when we analyzed complete bipartite graphs, our first objective is to derive a general formula computing the total size of the QASST equivalence class by calculating the number of distinct combinations of quotient graphs which preserve strong splits.
Our enumeration is based on an examination of the symmetries in these graphs.
For reference in the following discussion, Figures~\ref{fig:complete_k-partite_split_decomposition} and \ref{fig:clique-star_split_decomposition} show examples of a complete 4-partite graph and QASST equivalent clique-star, along with their split decompositions.

\begin{figure*}[t]
\centering

\begin{subfigure}{0.33\textwidth}
\centering
\includegraphics[width=0.8\textwidth,page=1]{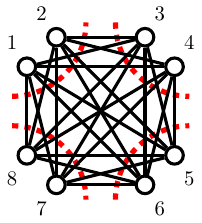}
\caption{$K_{2,2,2,2}$}
\label{fig:complete_k-partite_split_decomposition_graph}
\end{subfigure}\begin{subfigure}{0.33\textwidth}
\centering
\includegraphics[width=0.8\textwidth,page=2]{Figures/Appendix_Complete_k-Partite_and_Clique-Star_QASST_Structure.pdf}
\caption{$SST(K_{2,2,2,2})$}
\label{fig:complete_k-partite_split_decomposition_SST}
\end{subfigure}\begin{subfigure}{0.33\textwidth}
\centering
\includegraphics[width=0.9\textwidth,page=3]{Figures/Appendix_Complete_k-Partite_and_Clique-Star_QASST_Structure.pdf}
\caption{$QASST(K_{2,2,2,2})$}
\label{fig:complete_k-partite_split_decomposition_QASST}
\end{subfigure}
\caption{The split decomposition of the complete $4$-partite graph $K_{2,2,2,2}$.}
\label{fig:complete_k-partite_split_decomposition}
\end{figure*}

\begin{figure*}[t]
\centering

\begin{subfigure}{0.33\textwidth}
\centering
\includegraphics[width=0.8\textwidth,page=4]{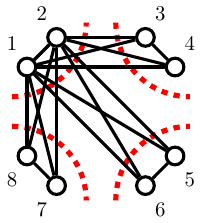}
\caption{$CS^1_{2,2,2,2}$}
\label{fig:clique-star_split_decomposition_graph}
\end{subfigure}\begin{subfigure}{0.33\textwidth}
\centering
\includegraphics[width=0.8\textwidth,page=5]{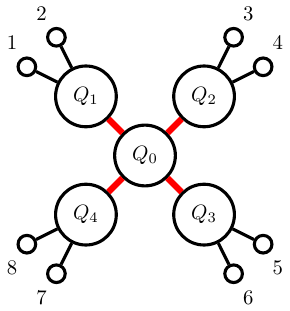}
\caption{$SST(CS^1_{2,2,2,2})$}
\label{fig:clique-star_split_decomposition_SST}
\end{subfigure}\begin{subfigure}{0.33\textwidth}
\centering
\includegraphics[width=0.9\textwidth,page=6]{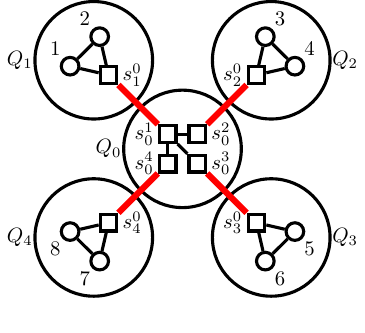}
\caption{$QASST(CS^1_{2,2,2,2})$}
\label{fig:clique-star_split_decomposition_QASST}
\end{subfigure}

\caption{The split decomposition of the clique-star $CS^1_{2,2,2,2}$.}
\label{fig:clique-star_split_decomposition}
\end{figure*}

The central quotient graph $Q_0$ can either be complete (c) as in Figure~\ref{fig:complete_k-partite_split_decomposition_QASST} or else star ($\text{sc}_j$); in the latter case, the center of the star always points towards one of the other components, indicated by the index $j\in[k]$.
In the example of Figure~\ref{fig:clique-star_split_decomposition_QASST}, the central quotient graph is $\text{sc}_1$ because $Q_0$ is a star whose center ``points" towards the quotient graph $Q_1$ in the QASST. This is denoted by the superscript in the notation $CS^1_{2,2,2,2}$ to distinguish it from clique-stars with a different choice of center

The other quotient graphs $Q_1,\cdots,Q_k$ have three possibilities: complete (c), star-center (sc), or star-spoke (ss) (recall that the distinction between star-center and star-spoke refers to the location of the split-node).
When a component $Q_i$ is star-spoke (i.e.~the single split-node is a spoke of the star), there are $n_i$ possible leaf-nodes that can be the center of the star, but each of these is symmetric. We make no distinction between these cases when listing the symmetry classes of $Q_i$, but we use these to count the size of each symmetry class.
By contrast, there is only one way for $Q_i$ to be complete or star-center.
Counting the size of a symmetry class amounts to comparing the number of star-spoke quotient graphs, avoiding redundancies.

Recall that certain combinations of adjacent quotient graphs do not preserve strong splits, and hence would define an invalid split decomposition if adjacent in the QASST. Our goal is to count the number of valid combinations of quotient graphs. There are two broad cases that can be distinguished based on the structure of the central quotient graph $Q_0$.
\begin{enumerate}
\item \textit{$Q_0$ is c.}

As long as $Q_{i>0}$ is not also c, any combination of quotient graphs preserves strong splits and is hence valid. That is to say, $Q_i$ can be sc (1 possibility) or ss ($n_i$ possibilities). The total number of combinations is thus $\prod_{i=1}^k(n_i+1)$.

\item \textit{For fixed $j\in[k]$, $Q_0$ is $\text{sc}_j$.}

In this case, $Q_j$ cannot be ss; it must be either c or sc (2 possibilities). For all remaining quotient graphs $Q_{i\neq0,j}$, these cannot be sc; they must be c or ss ($n_i+1$ possibilities). For fixed $j\in[k]$, this gives $2\prod_{i=1\neq j}^k(n_i+1)$ possible combinations.
\end{enumerate}

The size of the full QASST equivalence class, denoted $\Phi(K_{n_1,\cdots,n_k})$, is computed by combining the two cases above, wherein we sum over all choices of $j\in[k]$ for the second case. This yields the following formula:
\begin{eqnarray}\label{eq:complete_k-partite_QASST_upper_bound}
\Phi(K_{n_1,\cdots,n_k})&=&\prod_{i=1}^k(n_i+1)+2\sum_{j=1}^k\left[\prod_{i=1\neq j}^k(n_i+1)\right].\nonumber\\
\end{eqnarray}
Recall that the size of the QASST equivalence class is an upper bound on the size of the LC equivalence class, and so $|{\mathcal O}(K_{n_1,\cdots,n_k})|\leq\Phi(K_{n_1,\cdots,n_k})$.
Since they are QASST equivalent, we also have that $|{\mathcal O}(CS^r_{n_1,\cdots,n_k})|\leq\Phi(K_{n_1,\cdots,n_k})$.

\section{Classification of Complete $k$-partite and Clique-Star LC Orbits}
\label{app:classification_of_complete_k-partite_clique-star_LC_orbits}

In \ref{app:LC_non-equivalence}, we will verify explicitly that a complete $k$-partite graph $K_{n_1,\cdots,n_k}$ and the clique-star $CS^r_{n_1,\cdots,n_k}$ are not LC equivalent (Theorem~\ref{thm:LC_non-equivalence}).
Since they are QASST equivalent, this means that the sum of the sizes of their combined LC orbits must still be bounded above by the quantity computed in Equation~\ref{eq:complete_k-partite_QASST_upper_bound}: $|{\mathcal O}(K_{n_1,\cdots,n_k})|+|{\mathcal O}(CS^r_{n_1,\cdots,n_k})|\leq\Phi(K_{n_1,\cdots,n_k})$.
By computing lower bounds on the size of the LC orbit for each of these graphs, we will show that together they achieve the upper bound.
To compute such a lower bound, in addition to counting quotient graph symmetries, we will also give an explicit sequence of LC operations which transforms the original graph into a specified LC equivalent graph.
Doing so guarantees that the graph is included in the same local equivalence class, and hence the full LC orbit must be at least large enough to contain all of the graphs found in this way.
In contrast to the upper bound $\Phi(G)$, we will use the notation $\phi(G)$ to indicate the lower bound obtained in this way. We count labeled graphs as we did in Section~\ref{sect:compelte_bipartite_graphs_LC_orbit_size}.

\subsection{A Lower Bound on the Complete $k$-Partite LC Orbit}
\label{app:complete_k-partite_LC_orbit_lower_bound}

Each of the graphs in the LC equivalence class of $K_{n_1,\cdots,n_k}$ has a QASST decomposition which falls into one of three broad symmetry classes.
By symmetry class, we mean a choice of quotient graphs falling into one of three cases (c, sc, or ss) with this particular QASST structure.
These can be distinguished from each other by the behavior of the central quotient graph $Q_0$.
We begin by enumerating these cases.

\begin{enumerate}
\item \textit{$Q_0$ is c.}

For any even size subset of indices $I_{\text{even}}\subseteq[k]$, $Q_i$ is ss for each $i\in I_{\text{even}}$ and sc for each $i\in[k]\setminus I_{\text{even}}$. The total number of ss components is hence even.

\item \textit{For fixed $j\in[k]$, $Q_0$ is $\text{sc}_j$ and $Q_j$ is sc.}

For any even size subset of indices $I_{\text{even}}^{\setminus j}\subseteq[k]\setminus\{j\}$, $Q_i$ is ss for each $i\in I_{\text{even}}^{\setminus j}$ and c for each $i\in[k]\setminus\left(\{j\}\cup I_{\text{even}}^{\setminus j}\right)$. The number of ss components is even.

\item \textit{For fixed $j\in[k]$, $Q_0$ is $\text{sc}_j$ and $Q_j$ is c.}

For any odd size subset of indices $I_{\text{odd}}^{\setminus j}\subseteq[k]\setminus\{j\}$, $Q_i$ is ss for each $i\in I_{\text{odd}}^{\setminus j}$ and c for each $i\in[k]\setminus\left(\{j\}\cup I_{\text{odd}}^{\setminus j}\right)$. The number of ss components is odd.

\end{enumerate}

\begin{table*}[tp]
\centering
\begin{tabular}{|c|l|c|c|}
\hline
Case&Symmetry Class Rules&Number of Symmetries&Transformation from $K_{n_1,\cdots,n_k}$\\
\hline
1&\footnotesize$\begin{array}{l}\text{$Q_0$ is c}\\\text{Even number of $Q_i$ are ss}\\\text{All other $Q_i$ are sc}\end{array}$ &$\mathlarger{\sum_{I_{\text{even}}\subseteq [k]}\left(\prod_{i\in I_{\text{even}}}n_i\right)}$&$\mathlarger{f=\bigcomp_{I_s=\{i,j\}}^{ I_1\sqcup\cdots\sqcup I_t=I_{\text{even}}}\text{ep}([n_i],[n_j])}$\\
\hline
2(j)&\footnotesize$\begin{array}{l}\text{$Q_0$ is $\text{sc}_j$ and $Q_j$ is sc}\\\text{Even number of $Q_i$ are ss}\\\text{All other $Q_i$ are c}\end{array}$&$\mathlarger{\sum_{I_{\text{even}}^{\setminus j}\subseteq[k]\setminus\{j\}}\left(\prod_{i\in I_{\text{even}}^{\setminus j}}n_i\right)}$&$\mathlarger{f=\left(\bigcomp_{i\in I_{\text{even}}^{\setminus j}}c_{\ell}^{n_i}\right)\circ c_{\ell'}^{n_j}}$\\
\hline
3(j)&\footnotesize$\begin{array}{l}\text{$Q_0$ is $\text{sc}_j$ and $Q_j$ is c}\\\text{Odd number of $Q_i$ are ss}\\\text{All other $Q_i$ are c}\end{array}$&$\mathlarger{\sum_{I_{\text{odd}}^{\setminus j}\subseteq[k]\setminus\{j\}}\left(\prod_{i\in I_{\text{odd}}^{\setminus j}}n_i\right)}$&$\mathlarger{f=\left(\bigcomp_{i\in I_{\text{odd}}^{\setminus j}}c_{\ell}^{n_i}\right)\circ c_{\ell'}^{n_j}}$\\
\hline
\end{tabular}
\caption{
The three basic QASST symmetry classes in the LC orbit of $K_{n_1,\cdots,n_k}$, including the number of distinct graphs in this symmetry class and an explicit transformation to one such graph. The transformations are defined with respect to a choice of fixed subset of indices $I\subseteq [k]$ matching the conditions specified.
The exact order of the compositions across $I$ is unimportant; a different choice would define a different expression for $f$, but it would have the same effect on the graph. 
In Case 1, $I_{\text{even}}\subseteq[k]$ is an even size subset of indices and hence can be partitioned into a total of $t=|I_{\text{even}}|/2$ subsets of size 2; $I_{\text{even}}=I_1\sqcup\cdots\sqcup I_t$ is any choice of such a partition. The transformation is defined by a corresponding sequence of $t$ edge pivots via Equation~\ref{eq:component_edge_pivot}. Cases 2 and 3 are defined with respect to any choice of $j\in[k]$. In general, for cases where $Q_i$ is ss, the index $\ell\in[n_i]$ of the primitive local complement $c_{\ell}=c_{\ell}^{n_i}$ used in the transformation defines the center leaf-node of the star.
}
\label{tab:complete_k-partite_quasst_equivalences}
\end{table*}

\begin{figure*}[t]
\centering
\begin{subfigure}{0.33\textwidth}
\centering
\includegraphics[width=0.9\textwidth,page=7]{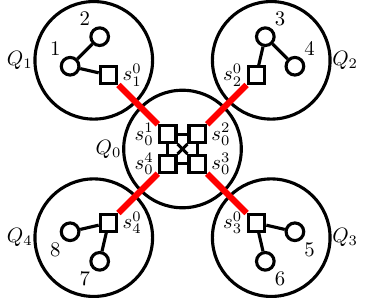}\\
\medskip
\includegraphics[width=0.7\textwidth,page=8]{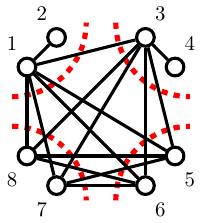}
\begin{eqnarray*}
I_{\text{even}}&=&\{1,2\}\subseteq[4]\\
f&=&\underbrace{c_1\circ c_3\circ c_1}_{\text{ep}([n_1],[n_2])}
\end{eqnarray*}
\caption{Case 1}
\label{fig:K2222_QASST_symmetry_examples_case1}
\end{subfigure}\begin{subfigure}{0.33\textwidth}
\centering
\includegraphics[width=0.9\textwidth,page=9]{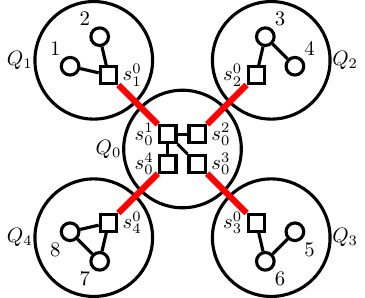}\\
\medskip
\includegraphics[width=0.7\textwidth,page=10]{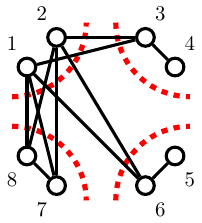}\\
\begin{eqnarray*}
I_{\text{even}}^{\setminus1}&=&\{2,3\}\subseteq[4]\setminus\{1\}\\
f&=&c_5\circ c_3\circ c_1\\
\end{eqnarray*}
\caption{Case 2}
\label{fig:K2222_QASST_symmetry_examples_case2}
\end{subfigure}\begin{subfigure}{0.33\textwidth}
\centering
\includegraphics[width=0.9\textwidth,page=11]{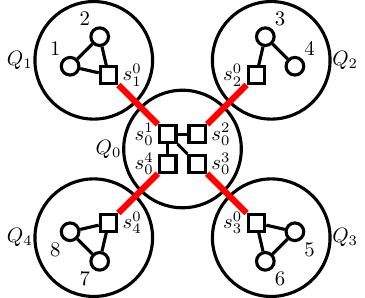}
\medskip
\includegraphics[width=0.7\textwidth,page=12]{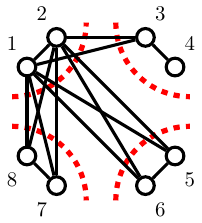}
\begin{eqnarray*}
I_{\text{odd}}^{\setminus1}&=&\{2\}\subseteq[4]\setminus\{1\}\\
f&=&c_3\circ c_1\\
\end{eqnarray*}
\caption{Case 3}
\label{fig:K2222_QASST_symmetry_examples_case3}
\end{subfigure}

\caption{Examples of graphs locally equivalent to $K_{2,2,2,2}$ from each of the three symmetry classes outlined in Table~\ref{tab:complete_k-partite_quasst_equivalences}. This includes the QASST decomposition, the specific choice of index set $I\subseteq [4]=\{1,2,3,4\}$ from among the quotient graph indices, and a corresponding explicit transformation $f$ for which $G=f(K_{2,2,2,2})$. The lower bound on the size of the full LC orbit as computed by the formula in Equation~\ref{eq:complete_k-partite_lower_bound} is $\phi(K_{2,2,2,2})=149$.}
\label{fig:K2222_QASST_symmetry_examples}
\end{figure*}

These rules are summarized in Table~\ref{tab:complete_k-partite_quasst_equivalences}.
Figure~\ref{fig:K2222_QASST_symmetry_examples} shows examples of graphs for each of these three cases which are locally equivalent to the complete $4$-partite graph $K_{2,2,2,2}$ illustrated in Figure~\ref{fig:complete_k-partite_split_decomposition}.
Note that the requirement for an even number of terms allows for zero terms, which corresponds to an empty subset of indices $I_{\text{even}}=\emptyset$. In this case, a product over the indices in an empty set is taken to be 1.
Table~\ref{tab:complete_k-partite_quasst_equivalences} counts all graphs falling into one these three cases and provides an explicit transformation to one such graph.

The transformations included in the table are specified by a given subset of vertices $I\in[k]$, up to a choice of indices for primitive local complements. For a given set $I$, there are many possible such formulas, but each will yield an isomorphic graph.
In general these transformations are non-unique.
Across all symmetry classes in Table~\ref{tab:complete_k-partite_quasst_equivalences} and all choices of index $j$, the total number of distinct graphs in the LC orbit of $K_{n_1,\cdots,n_k}$ is obtained by summing the three cases in the table across all choices of $j\in[k]$.
We compute this formula as follows.
\begin{widetext}
\begin{eqnarray}\label{eq:complete_k-partite_lower_bound}
\phi(K_{n_1,\cdots,n_k})&=&\sum_{I_{\text{even}}\subseteq [k]}\left(\prod_{i\in I_{\text{even}}}n_i\right)+\sum_{j=1}^k\left[\sum_{I_{\text{even}}^{\setminus j}\subseteq[k]\setminus\{j\}}\left(\prod_{i\in I_{\text{even}}^{\setminus j}}n_i\right)+\sum_{I_{\text{odd}}^{\setminus j}\subseteq[k]\setminus\{j\}}\left(\prod_{i\in I_{\text{odd}}^{\setminus j}}n_i\right)\right]\nonumber\\
&=&\sum_{I_{\text{even}}\subseteq [k]}\left(\prod_{i\in I_{\text{even}}}n_i\right)+\sum_{j=1}^k\left[\sum_{I^{\setminus j}\subseteq[k]\setminus\{j\}}\left(\prod_{i\in I^{\setminus j}}n_i\right)\right]\nonumber\\
&=&\underbrace{\prod_{i=1}^k(n_i+1)}_{\text{only terms with even number of products}}+\sum_{j=1}^k\left(\prod_{i\in[k]\setminus\{j\}}(n_i+1)\right)
\end{eqnarray}
\end{widetext}

Since Table~\ref{tab:complete_k-partite_quasst_equivalences} gives an explicit sequence of local complements to transform $K_{n_1,\cdots,n_k}$ into any QASST equivalent graph described by one of these three cases, we may conclude that all such graphs are locally equivalent.
Hence, by counting all such graphs, Equation~\ref{eq:complete_k-partite_lower_bound} gives a lower bound on the size of the LC orbit.
Let $\phi(K_{n_1,\cdots,n_k})$ denote the bound computed in this way.

\subsection{A Lower Bound on the Clique-Star LC Orbit}
\label{app:clique-star_LC_orbit_lower_bound}

Now consider the clique-star $CS^r_{n_1,\cdots,n_k}$ whose clique-center corresponds to the quotient graph $Q_r$, where $r\in[k]$. Given another index $s\in[k]$, the clique-center can be changed from $Q_r$ to $Q_s$ via the edge pivot $\text{ep}([n_r],[n_s])$ of Equation~\ref{eq:component_edge_pivot}.
To define an explicit transformation starting from a given clique-star, in some cases it will be useful to start by changing the clique-center in this way.
We will keep with the index $r\in[k]$ to describe the clique-center of a generic clique-star when writing explicit LC transformations, but this choice is arbitrary and nothing is lost by choosing $r=1$.

As we did for the complete $k$-partite graph, the symmetry classes for the clique-star can be divided into three broad cases, which we enumerate here.
Again, these can be distinguished from each other by the behavior of $Q_0$.

\begin{table*}[tp]
\centering
\begin{tabular}{|c|l|c|c|}
\hline
Case&Symmetry Class Rules&Number of Symmetries&Transformation from $CS^r_{n_1,\cdots,n_k}$\\
\hline
1&\footnotesize$\begin{array}{l}\text{$Q_0$ is c}\\\text{Odd number of $Q_i$ are ss}\\\text{All other $Q_i$ are sc}\end{array}$ &$\mathlarger{\sum_{I_{\text{odd}}\subseteq [k]}\left(\prod_{i\in I_{\text{odd}}}n_i\right)}$&$\mathlarger{f=c_{\ell'}^{n_r}\circ\left(\bigcomp_{i\in I_{\text{odd}}\setminus\{r\}}c_{\ell}^{n_i}\right)}$\\
\hline
2(j)&\footnotesize$\begin{array}{l}\text{$Q_0$ is $\text{sc}_j$ and $Q_j$ is sc}\\\text{Odd number of $Q_i$ are ss}\\\text{All other $Q_i$ are c}\end{array}$&$\mathlarger{\sum_{I_{\text{odd}}^{\setminus j}\subseteq[k]\setminus\{j\}}\left(\prod_{i\in I_{\text{odd}}^{\setminus j}}n_i\right)}$&$\mathlarger{f=\left(\bigcomp_{i\in I_{\text{odd}}^{\setminus j}}c_{\ell}^{n_i}\right)\circ\text{ep}([n_r],[n_j])}$\\
\hline
3(j)&\footnotesize$\begin{array}{l}\text{$Q_0$ is $\text{sc}_j$ and $Q_j$ is c}\\\text{Even number of $Q_i$ are ss}\\\text{All other $Q_i$ are c}\end{array}$&$\mathlarger{\sum_{I_{\text{even}}^{\setminus j}\subseteq[k]\setminus\{j\}}\left(\prod_{i\in I_{\text{even}}^{\setminus j}}n_i\right)}$&$\mathlarger{f=\left(\bigcomp_{i\in I_{\text{even}}^{\setminus j}}c_{\ell}^{n_i}\right)\circ\text{ep}([n_r],[n_j])}$\\
\hline
\end{tabular}
\caption{
The three basic QASST symmetry classes in the LC orbit of $CS^r_{n_1,\cdots,n_k}$, a $(k-1)$-spoke clique-star with clique-center $Q_r$.
We count the total number of distinct graphs in each symmetry class and give the general formula for an explicit transformation to such a graph based on a choice of index set $I\in[k]$ matching the specified conditions.
Again, the order of the composition is unimportant; the expression for $f$ is non-unique, but will transform the graph in the same way.
In the transformation for Case 1, note that the final primitive local complement in the composition is $c_{\ell'}^{n_r}$, regardless of whether or not $r\in I_{\text{odd}}$. This occurs because the clique-center is chosen to be $r\in[k]$; a different choice of clique-center would substitute the corresponding index in $[k]$. For Cases 2 and 3, the first step is to change the clique-center from $Q_r$ to $Q_j$ via the edge pivot of Equation~\ref{eq:component_edge_pivot}, if necessary.
For quotient graphs $Q_i$ which are ss, the center of the star is the leaf-node with index $\ell\in[n_i]$ of the primitive local complement $c_\ell=c_{\ell}^{n_i}$ used in the transformation.
}
\label{tab:clique-star_quasst_equivalences}
\end{table*}

\begin{figure*}[t]
\centering

\begin{subfigure}{0.33\textwidth}
\centering
\includegraphics[width=0.9\textwidth,page=13]{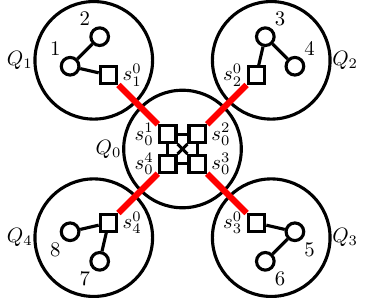}\\
\medskip
\includegraphics[width=0.7\textwidth,page=14]{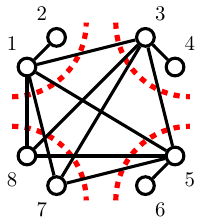}
\begin{eqnarray*}
I_{\text{odd}}&=&\{1,2,3\}\subseteq[4]\\
f&=&c_1\circ c_5\circ c_3
\end{eqnarray*}
\caption{Case 1}
\label{fig:CS2222_QASST_symmetry_examples_case1}
\end{subfigure}\begin{subfigure}{0.33\textwidth}
\centering
\includegraphics[width=0.9\textwidth,page=15]{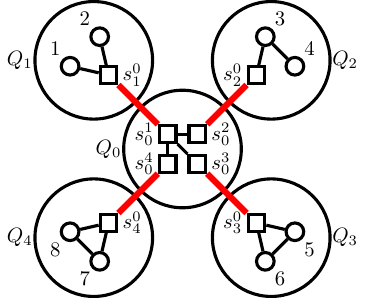}\\
\medskip
\includegraphics[width=0.7\textwidth,page=16]{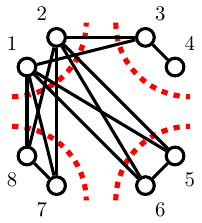}
\begin{eqnarray*}
I_{\text{odd}}^{\setminus1}&=&\{2\}\subseteq[4]\setminus\{1\}\\
f&=&c_3
\end{eqnarray*}
\caption{Case 2}
\label{fig:CS2222_QASST_symmetry_examples_case2}
\end{subfigure}\begin{subfigure}{0.33\textwidth}
\centering
\includegraphics[width=0.9\textwidth,page=17]{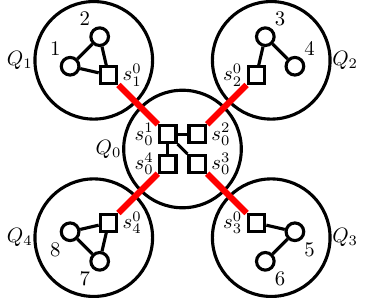}\\
\medskip
\includegraphics[width=0.7\textwidth,page=18]{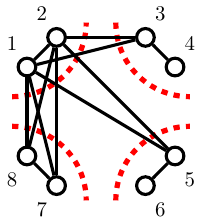}
\begin{eqnarray*}
I_{\text{even}}^{\setminus1}&=&\{2,3\}\subseteq[4]\setminus\{1\}\\
f&=&c_5\circ c_3
\end{eqnarray*}
\caption{Case 3}
\label{fig:CS2222_QASST_symmetry_examples_case3}
\end{subfigure}

\caption{Examples of graphs locally equivalent to $CS^1_{2,2,2,2}$ from each of the three symmetry classes outlined in Table~\ref{tab:clique-star_quasst_equivalences}. This includes the QASST decomposition, the specific choice of index set $I\subseteq [4]=\{1,2,3,4\}$ from among the quotient graph indices, and a corresponding explicit transformation $f$ for which $G=f(CS^1_{2,2,2,2})$. The lower bound on the size of the full LC orbit as computed by the formula in Equation~\ref{eq:clique-star_lower_bound} is $\phi(CS^1_{2,2,2,2})=148$.}
\label{fig:CS2222_QASST_symmetry_examples}
\end{figure*}

\begin{enumerate}
\item \textit{$Q_0$ is c.}

For any odd size subset of indices $I_{\text{odd}}\subseteq[k]$, $Q_i$ is ss for each $i\in I_{\text{odd}}$ and sc for each $i\in[k]\setminus I_{\text{odd}}$. The total number of ss components is hence odd.

\item \textit{For fixed $j\in[k]$, $Q_0$ is $\text{sc}_j$ and $Q_j$ is sc.}

For any odd size subset of indices $I_{\text{odd}}^{\setminus j}\subseteq[k]\setminus\{j\}$, $Q_i$ is ss for each $i\in I_{\text{odd}}^{\setminus j}$ and c for each $i\in[k]\setminus\left(\{j\}\cup I_{\text{odd}}^{\setminus j}\right)$. The number of ss components is odd.

\item \textit{For fixed $j\in[k]$, $Q_0$ is $\text{sc}_j$ and $Q_j$ is c.}

For any even size subset of indices $I_{\text{even}}^{\setminus j}\subseteq[k]\setminus\{j\}$, $Q_i$ is ss for each $i\in I_{\text{even}}^{\setminus j}$ and c for each $i\in[k]\setminus\left(\{j\}\cup I_{\text{even}}^{\setminus j}\right)$. The number of ss components is even.
\end{enumerate}

These three rules for the clique-star LC symmetry classes are summarized in Table~\ref{tab:clique-star_quasst_equivalences}. Figure~\ref{fig:CS2222_QASST_symmetry_examples} shows examples of graphs for each of these three cases which are locally equivalent to the clique-star $CS^1_{2,2,2,2}$ illustrated in Figure~\ref{fig:clique-star_split_decomposition}.
Again, we allow for empty index sets $I_{\text{even}}=\emptyset$ among the even size subsets; these are counted as subsets of size 0 and products over an empty set of indices are taken to be 1.
Table~\ref{tab:clique-star_quasst_equivalences} enumerates all of the symmetries falling into one of these three cases along with a corresponding explicit transformation (non-unique) starting from a clique-star.

The index sets $I\in[k]$ used to count the transformations in the table are also used to specify the explicit LC transformations.
For components $Q_i$ which are ss, the center of the star corresponds to the leaf-node with index $\ell\in[n_i]$ matching the primitive local complement $c_\ell=c_\ell^{n_i}$ used in the transformation.
As in the complete k-partite case, a different choice of $\ell\in[n_i]$ would yield a distinct but isomorphic graph.
These transformations are non-unique in general, but showing that at least one exists is sufficient to conclude local equivalence.
We compute a lower bound $\phi(CS^r_{n_1,\cdots,n_k})$ on the total number of graphs LC equivalent to the clique-star by combining the number of symmetries as follows.

\begin{widetext}
\begin{eqnarray}\label{eq:clique-star_lower_bound}
\phi(CS^r_{n_1,\cdots,n_k})&=&\sum_{I_{\text{odd}}\subseteq [k]}\left(\prod_{i\in I_{\text{odd}}}n_i\right)+\sum_{j=1}^k\left[\sum_{I_{\text{odd}}^{\setminus j}\subseteq[k]\setminus\{j\}}\left(\prod_{i\in I_{\text{odd}}^{\setminus j}}n_i\right)+\sum_{I_{\text{even}}^{\setminus j}\subseteq[k]\setminus\{j\}}\left(\prod_{i\in I_{\text{even}}^{\setminus j}}n_i\right)\right]\nonumber\\
&=&\sum_{I_{\text{odd}}\subseteq [k]}\left(\prod_{i\in I_{\text{odd}}}n_i\right)+\sum_{j=1}^k\left[\sum_{I^{\setminus j}\subseteq[k]\setminus\{j\}}\left(\prod_{i\in I^{\setminus j}}n_i\right)\right]\nonumber\\
&=&\underbrace{\prod_{i=1}^k(n_i+1)}_{\text{only terms with odd number of products}}+\sum_{j=1}^k\left(\prod_{i\in[k]\setminus\{j\}}(n_i+1)\right)
\end{eqnarray}
\end{widetext}

\subsection{Sizes of the LC Orbits}

Since the complete $k$-partite graph and the clique-star are QASST equivalent but known to be in different LC orbits, the sum of the two lower bounds $\phi(K_{n_1,\cdots,n_k})$ and $\phi(CS^r_{n_1,\cdots,n_k})$ taken together must be bounded above by the size of the QASST equivalence class:
$$\phi(K_{n_1,\cdots,n_k})+\phi(CS^r_{n_1,\cdots,n_k})\leq\Phi(K_{n_1,\cdots,n_k}).$$
We may compare these two lower bounds with the formula for the size of the QASST equivalence class computed in Equation~\ref{eq:complete_k-partite_QASST_upper_bound}.
Combining Equations~\ref{eq:complete_k-partite_lower_bound} and \ref{eq:clique-star_lower_bound}, we see that we reach the upper bound.

\begin{widetext}
\begin{eqnarray}
\phi(K_{n_1,\cdots,n_k})+\phi(CS^r_{n_1,\cdots,n_k})&=&\left[\underbrace{\prod_{i=1}^k(n_i+1)}_{\text{only terms with even number of products}}+\sum_{j=1}^k\left(\prod_{i\in[k]\setminus\{j\}}(n_i+1)\right)\right]+\nonumber\\
&&\left[\underbrace{\prod_{i=1}^k(n_i+1)}_{\text{only terms with odd number of products}}+\sum_{j=1}^k\left(\prod_{i\in[k]\setminus\{j\}}(n_i+1)\right)\right]\nonumber\\
&=&\prod_{i=1}^k(n_i+1)+2\sum_{j=1}\left(\prod_{i=1\neq j}^k(n_i+1)\right)\\
&=&\Phi(K_{n_1,\cdots,n_k})\nonumber
\end{eqnarray}
\end{widetext}

In other words, the LC equivalence classes of the complete $k$-partite graph and the clique-star together exhaust the QASST equivalence class.
In particular, we may infer that the two lower bounds computed in Equations~\ref{eq:complete_k-partite_lower_bound} and \ref{eq:clique-star_lower_bound} correspond to the actual sizes of the LC orbits, respectively, since there is no room for additional graphs.
We state these here as theorems.
\begin{theorem}\label{thm:complete_k-partite_LC_orbit_size}
The size of the LC orbit of the complete $k$-partite graph $K_{n_1,\cdots,n_k}$, where $k\geq 3$ and each $n_i\geq2$, is
\begin{eqnarray}\label{eq:complete_k-partite_LC_orbit_size}
|{\mathcal O}(K_{n_1,\cdots,n_k})|&=&\underbrace{\prod_{i=1}^k(n_i+1)}_{\text{even products}}+\sum_{j=1}^k\left(\prod_{i\in[k]\setminus\{j\}}(n_i+1)\right).\nonumber\\
\end{eqnarray}
\end{theorem}
\begin{theorem}\label{thm:clique-star_LC_orbit_size}
The size of the LC orbit of the clique-star $CS^r_{n_1,\cdots,n_k}$, where $k\geq3$ and each $n_i\geq2$, is
\begin{eqnarray}\label{eq:clique-star_LC_orbit_size}
|{\mathcal O}(CS^r_{n_1,\cdots,n_k})|&=&\underbrace{\prod_{i=1}^k(n_i+1)}_{\text{odd products}}+\sum_{j=1}^k\left(\prod_{i\in[k]\setminus\{j\}}(n_i+1)\right).\nonumber\\
\end{eqnarray}
\end{theorem}

Furthermore, any graph QASST equivalent to a complete $k$-partite graph must fall into one of these two orbits.
This last observation in particular can be applied to the special cases of repeater graphs and multi-leaf repeater graphs.

\subsection{Local Equivalence of Multi-leaf Repeater Graphs}
\label{app:local_equivalence_of_MLR}

Let $MR_{n_1,\cdots,n_k}$ denote a multi-leaf repeater graph, which consists of a complete graph on $k$ nodes, each of which has one or more leaves attached to it (Figure~\ref{fig:multi-leaf_repeater_generic_graph}).
We assume that $k\geq 3$ and that $n_i\geq2$, so that the $i^{\text{th}}$ node in the central complete graph has an additional $n_i-1$ degree 1 nodes attached to it.
Note that $MR_{n_1,\cdots,n_k}$ is distance-hereditary.
The multi-leaf repeater graph has a QASST consisting of a central complete graph $Q_0$ on $k$ split-nodes, surrounded by a ring of star-spoke graphs $Q_1,\cdots,Q_k$, each containing a single split-node and $n_i$ leaf-nodes (Figure~\ref{fig:multi-leaf_repeater_generic_QASST}).
In other words, $MR_{n_1,\cdots,n_k}$ is QASST equivalent to $K_{n_1,\cdots,n_k}$, but it may or may not be LC equivalent depending on the parity of $k$.
Note that this QASST falls into Case 1 of Table~\ref{tab:complete_k-partite_quasst_equivalences} when $k$ is even, and into Case 1 of Table~\ref{tab:clique-star_quasst_equivalences} when $k$ is odd.
We state this as a theorem.

\begin{figure*}[t]
\centering

\begin{subfigure}{0.66\textwidth}
\centering
\includegraphics[width=0.9\textwidth,page=19]{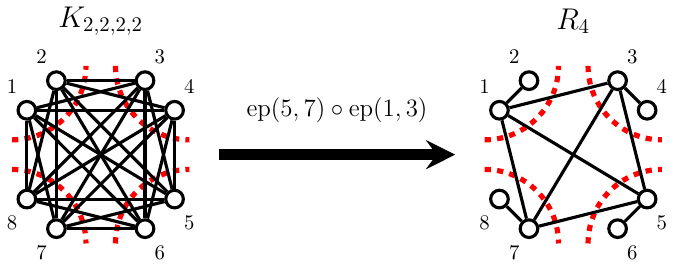}
\caption{
Local equivalence of $K_{2,2,2,2}$ and $R_4$.
}
\label{fig:Transform_K2222_to_R4}
\end{subfigure}\begin{subfigure}{0.33\textwidth}
\centering
\includegraphics[width=0.9\textwidth,page=20]{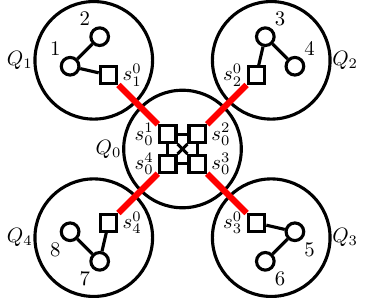}
\caption{$QASST(R_4)$}
\label{fig:R4_QASST_decomposition}
\end{subfigure}

\bigskip

\begin{subfigure}{0.66\textwidth}
\centering
\includegraphics[width=0.9\textwidth,page=21]{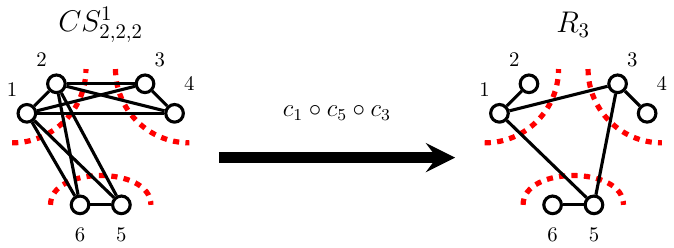}
\caption{
Local equivalence of $CS^1_{2,2,2}$ and $R_3$.
}
\label{fig:Transform_CS222_to_R3}
\end{subfigure}\begin{subfigure}{0.33\textwidth}
\centering
\includegraphics[width=0.9\textwidth,page=22]{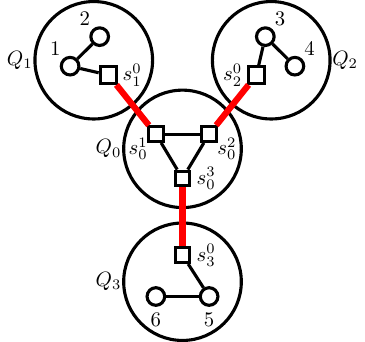}
\caption{$QASST(R_3)$}
\label{fig:R3_QASST_decomposition}
\end{subfigure}

\caption{Illustration of the local equivalence between repeater graphs and complete $k$-partite graphs (when $k$ is even) or clique-stars (when $k$ is odd). The explicit transformations and QASST decompositions match the rules of Tables~\ref{tab:complete_k-partite_quasst_equivalences} and \ref{tab:clique-star_quasst_equivalences}.}
\label{fig:MLR_LC_equivalence_examples}
\end{figure*}

\begin{theorem}\label{thm:multi-leaf_repeater_graph_LC_equivalence}
Let $MR_{n_1,\cdots,n_k}$ denote a multi-leaf repeater graph, where $k\geq3$ and each $n_i\geq 2$.
If $k$ is even, then $MR_{n_1,\cdots,n_k}\in{\mathcal O}(K_{n_1,\cdots,n_k})$.
If $k$ is odd, then $MR_{n_1,\cdots,n_k}\in{\mathcal O}(CS^r_{n_1,\cdots,n_k})$.
\end{theorem}

Local equivalence of the complete bipartite graph and certain variations of the repeater graph is already known in the literature~\cite{tzitrin2018local}, but Theorem~\ref{thm:multi-leaf_repeater_graph_LC_equivalence} applies to the general class of multi-leaf repeater graphs. This includes the standard repeater graph $R_k$ (consisting of a complete graph on $k$ vertices with a single extra leaf on each vertex; these were originally introduced as quantum repeaters~\cite{azuma2015all}). This occurs in the special case when $n_1=n_2=\cdots=n_k=2$, in which case $MR_{2,\cdots,2}=R_k$.
Figure~\ref{fig:MLR_LC_equivalence_examples} shows examples of repeater graphs and their local equivalence for even and odd choices of $k$.
From the explicit constructions, $\phi(K_{n_1,\cdots,n_k})\le |O(K_{n_1,\cdots,n_k})|$ and $\phi(CS^r_{n_1,\cdots,n_k})\le |O(CS^r_{n_1,\cdots,n_k})|$. Since the two orbits are disjoint and contained in the same QASST-equivalence class, $|O(K_{n_1,\cdots,n_k})|+|O(CS^r_{n_1,\cdots,n_k})|\le \Phi(K_{n_1,\cdots,n_k})$. Together with $\phi(K_{n_1,\cdots,n_k})+\phi(CS^r_{n_1,\cdots,n_k})=\Phi(K_{n_1,\cdots,n_k})$, we conclude equality throughout, hence $|O(K_{n_1,\cdots,n_k})|=\phi(K_{n_1,\cdots,n_k})$ and $|O(CS^r_{n_1,\cdots,n_k})|=\phi(CS^r_{n_1,\cdots,n_k})$.

\section{LC Non-Equivalence of Complete $k$-Partite and the Clique-Star}
\label{app:LC_non-equivalence}

It remains to verify that the complete $k$-partite graph $K_{n_1,\cdots,n_k}$ and clique-star $CS^r_{n_1,\cdots,n_k}$ indeed belong to disjoint local equivalence classes.
One way to prove this is by showing that the three symmetry classes introduced in Tables~\ref{tab:complete_k-partite_quasst_equivalences} and \ref{tab:clique-star_quasst_equivalences} are closed under LC operations.

\subsection{LC Closure of the Complete $k$-Partite QASST Symmetry Classes}

Suppose that $G\in{\mathcal O}(K_{n_1,\cdots,n_k})$ belongs to any of the three symmetry classes outlined in Table~\ref{tab:complete_k-partite_quasst_equivalences}; we will refer to these cases as 1, 2(j), and 3(j). It is sufficient to show that the graph $c_v(G)$ obtained from a primitive local complement $c_v\in{\mathcal L}_{|G(V)|}$ for any choice of vertex $v\in(G)$ also belongs to one of these three cases.

If $v\in V(G)$ is any vertex, then it also corresponds to a leaf-node $v\in V(Q_i)$ for some quotient graph $Q_i$ of $QASST(G)$.
If $Q_i$ is complete, then local complement on $G$ with respect to any choice of leaf-node in $Q_i$ yields a star graph, and all such stars are isomorphic, and hence lie in the same QASST symmetry class. A similar fact holds if $Q_i$ is star-center since all of the leaf-nodes are spoke. If $Q_i$ is star-spoke, we only need to distinguish between the center leaf-node and spoke leaf-nodes.
Hence, for any of the three cases for $G$, we only need to consider a few subcases distinguishing between these possibilities.
These subcases are outlined in Table~\ref{tab:LC_closure_complete_k-partite}.
Figure~\ref{fig:LC_closure_complete_k-partite} illustrates each of the subcases in this table for simple examples of graphs from ${\mathcal O}(K_{2,2,2,2})$.

For any $G\in\mathcal{O}(K_{n_1,\cdots,n_k})$ with symmetry class belonging to case 1, 2(j), or 3(j), we see that local complement with respect to any choice of vertex $v\in V(G)$ yields a graph also belonging to one of these three symmetry classes.
That is to say, no operation transforms $G$ into a graph with QASST symmetry class belonging to one of the three cases for the clique-star specified in Table~\ref{tab:clique-star_quasst_equivalences}, and so we may conclude that $G$ is not locally equivalent to any of these. 

\begin{table*}[tp]
\centering
\begin{tabular}{|c|c|l|c|}
\hline
$QASST(G)$&Subcase&Condition for Vertex $v\in V(G)$&$QASST(c_v(G))$\\
\hline
Case 1&1-a&$v\in V(Q_i)$ is center of ss $Q_i$&Case 3(i)\\
&1-b&$v\in V(Q_i)$ is spoke of ss $Q_i$&Case 1\\
&1-c&$v\in V(Q_i)$ is any node of sc $Q_i$&Case 2(i)\\
\hline
Case 2(j)&2(j)-a&$v\in V(Q_j)$ is any node of sc $Q_j$&Case 1\\
&2(j)-b&$v\in V(Q_i)$ is center of ss $Q_i$&Case 3(j)\\
&2(j)-c&$v\in V(Q_i)$ is spoke of ss $Q_i$&Case 2(j)\\
&2(j)-d&$v\in V(Q_i)$ is any node of c $Q_i$&Case 3(j)\\
\hline
Case 3(j)&3(j)-a&$v\in V(Q_j)$ is any node of c $Q_j$&Case 1\\
&3(j)-b&$v\in V(Q_i)$ is center of ss $Q_i$&Case 2(j)\\
&3(j)-c&$v\in V(Q_i)$ is spoke of ss $Q_i$&Case 3(j)\\
&3(j)-d&$v\in V(Q_i)$ is any node of c $Q_i$&Case 2(j)\\
\hline
\end{tabular}
\caption{
The 11 subcases show that any graph in ${\mathcal O}(K_{n_1,\cdots,n_k})$ with a QASST structure belonging to one of the three symmetry classes introduced in Table~\ref{tab:complete_k-partite_quasst_equivalences} remains in one of these classes. 
}
\label{tab:LC_closure_complete_k-partite}
\end{table*}

\begin{figure*}[tp]
\centering
\includegraphics[height=0.95\textheight,page=23]{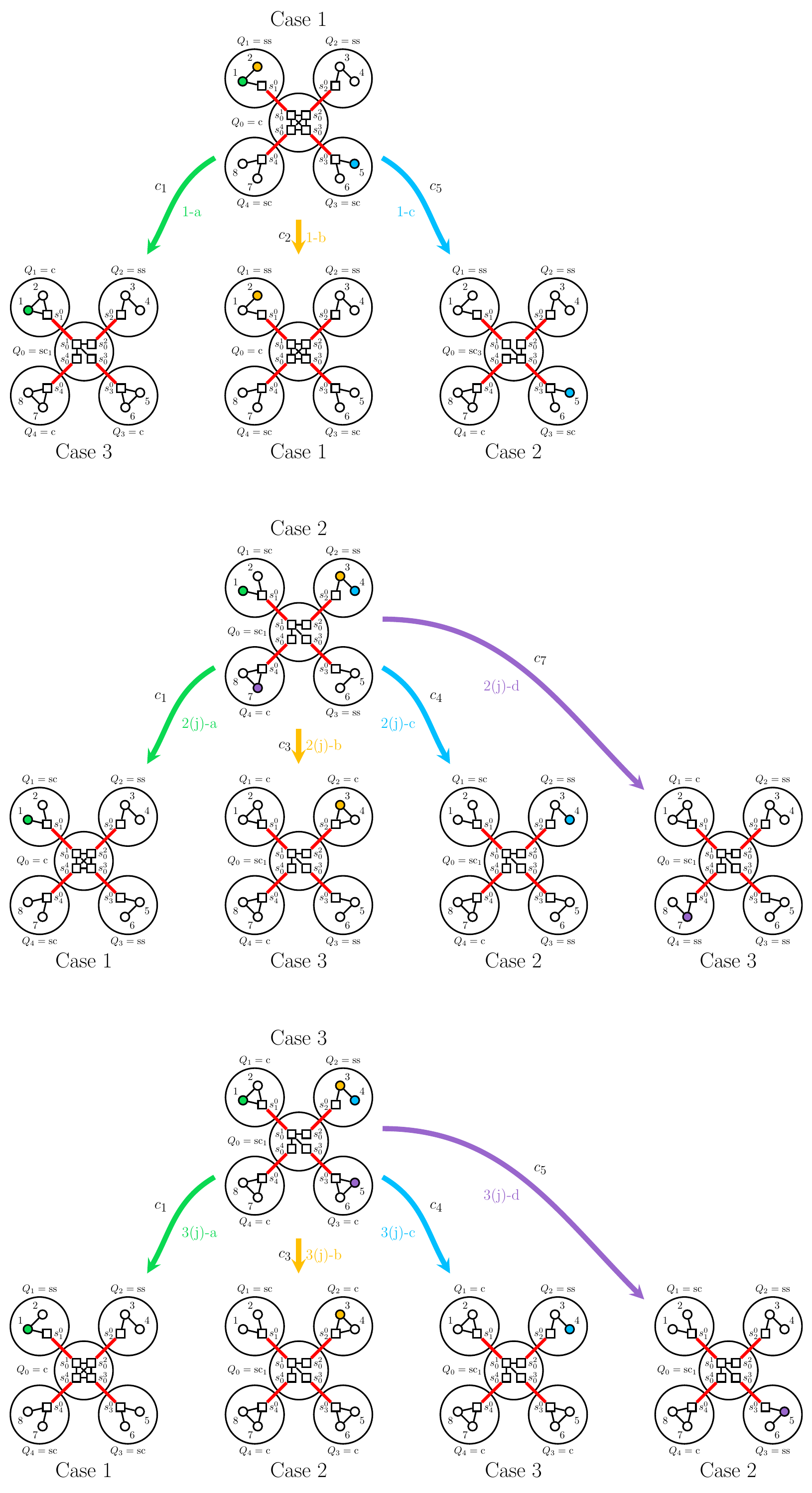}
\caption{Examples of local complements on simple graphs from $\mathcal{O}(K_{2,2,2,2})$ falling into each of the 11 subcases described in Table~\ref{tab:LC_closure_complete_k-partite}.}
\label{fig:LC_closure_complete_k-partite}
\end{figure*}

\subsection{LC Closure of the Clique-Star QASST Symmetry Classes}

Next we consider the LC orbit of the clique-star, supposing that $G\in{\mathcal O}(CS^r_{n_1,\cdots,n_k})$ belongs to any of the three symmetry classes outlined in Table~\ref{tab:clique-star_quasst_equivalences}. Again, we refer to these cases as 1, 2(j), and 3(j).
We proceed in the same way as we did for the complete $k$-partite graph by identifying the symmetry class of $QASST(c_v(G))$, where $c_v\in{\mathcal L}_{|V(G)|}$ is the primitive local complement corresponding to any choice of a vertex $v\in V(G)$.
The resulting analysis is shown in Table~\ref{tab:LC_closure_clique-star}, with simple examples of the subcases shown in Figure~\ref{fig:LC_closure_clique-star} taken from ${\mathcal O}(CS^1_{2,2,2,2})$.

\begin{table*}[tp]
\centering
\begin{tabular}{|c|c|l|c|}
\hline
$QASST(G)$&Subcase&Condition for Vertex $v\in V(G)$&$QASST(c_v(G))$\\
\hline
Case 1&1-a&$v\in V(Q_i)$ is center of ss $Q_i$&Case 3(i)\\
&1-b&$v\in V(Q_i)$ is spoke of ss $Q_i$&Case 1\\
&1-c&$v\in V(Q_i)$ is any node of sc $Q_i$&Case 2(i)\\
\hline
Case 2(j)&2(j)-a&$v\in V(Q_j)$ is any node of sc $Q_j$&Case 1\\
&2(j)-b&$v\in V(Q_i)$ is center of ss $Q_i$&Case 3(j)\\
&2(j)-c&$v\in V(Q_i)$ is spoke of ss $Q_i$&Case 2(j)\\
&2(j)-d&$v\in V(Q_i)$ is any node of c $Q_i$&Case 3(j)\\
\hline
Case 3(j)&3(j)-a&$v\in V(Q_j)$ is any node of c $Q_j$&Case 1\\
&3(j)-b&$v\in V(Q_i)$ is center of ss $Q_i$&Case 2(j)\\
&3(j)-c&$v\in V(Q_i)$ is spoke of ss $Q_i$&Case 3(j)\\
&3(j)-d&$v\in V(Q_i)$ is any node of c $Q_i$&Case 2(j)\\
\hline
\end{tabular}
\caption{
The 11 subcases showing that any graph in ${\mathcal O}(CS^r_{n_1,\cdots,n_k})$ with a QASST structure belonging to one of the three symmetry classes introduced in Table~\ref{tab:clique-star_quasst_equivalences} remains in one of these classes.
Note that these exactly match the subcases of Table~\ref{tab:LC_closure_complete_k-partite} and this is not a coincidence.
}
\label{tab:LC_closure_clique-star}
\end{table*}

\begin{figure*}[tp]
\centering
\includegraphics[height=0.95\textheight,page=24]{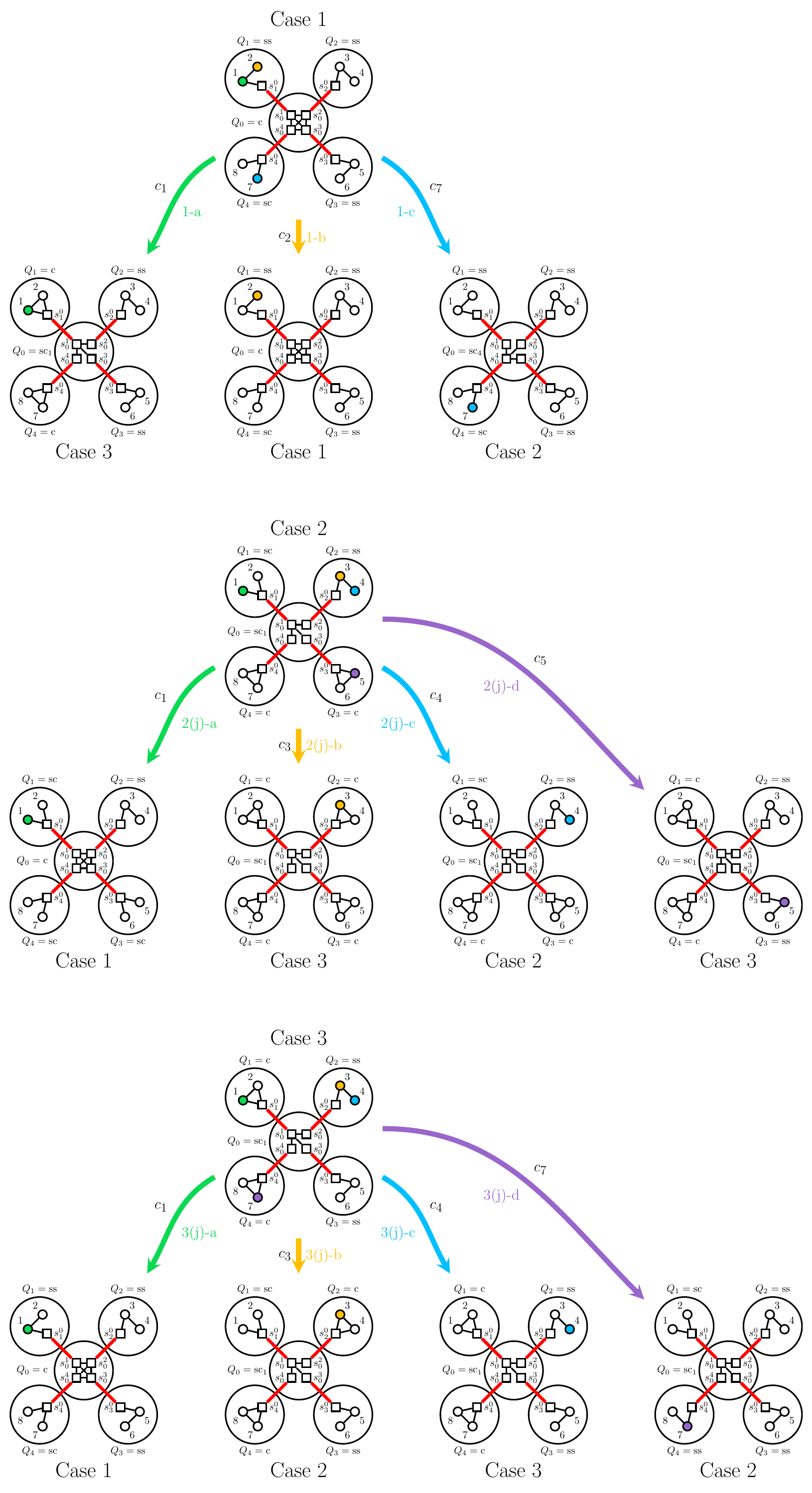}
\caption{Examples of local complements on simple graphs from ${\mathcal O}(CS^1_{2,2,2,2})$ falling into each of the 11 subcases described in Table~\ref{tab:LC_closure_clique-star}.}
\label{fig:LC_closure_clique-star}
\end{figure*}

Interestingly, when enumerated in the same order, we see that the subcases in Table~\ref{tab:LC_closure_clique-star} are identical to those in Table~\ref{tab:LC_closure_complete_k-partite}.
The reader may verify this by comparing the examples of Figure~\ref{fig:LC_closure_complete_k-partite} and Figure~\ref{fig:LC_closure_clique-star}.
This occurs because the only distinction between the three QASST symmetry classes of the complete $k$-partite graph and those of the clique-star is the parity of the number of star-spoke quotient graphs (either even or odd).

\subsection{Disjoint Local Equivalence Classes}

Taken together, Tables~\ref{tab:LC_closure_complete_k-partite} and \ref{tab:LC_closure_clique-star} show that no graph in ${\mathcal O}(K_{n_1,\cdots,n_k})$ can be converted into a graph in ${\mathcal O}(CS^r_{n_1,\cdots,n_k})$ via local complements, establishing the following theorem.

\begin{theorem}\label{thm:LC_non-equivalence}
The complete $k$-partite graph $K_{n_1,\cdots,n_k}$ and clique-star $CS^r_{n_1,\cdots,n_k}$ are not locally equivalent, and so ${\mathcal O}(K_{n_1,\cdots,n_k})\cap{\mathcal O}(CS^r_{n_1,\cdots,n_k})=\emptyset$.
\end{theorem}

\section{Edge Counting for Graphs QASST Equivalent to $K_{n_1,\cdots,n_k}$}
\label{app:edge_counting}

In many practical applications, one is interested in finding a minimum edge representative of the LC orbit~\cite{adcock2020mapping,sharma2025minimizing}.
With this application in mind, we will provide a general formula for the number of edges in any graph locally equivalent to the complete $k$-partite graph or a clique-star, along with some constraints by which a minimum edge representative may be determined. Throughout Appendix~\ref{app:edge_counting}, $n_i$ denotes the number of leaf-nodes in $Q_i$ (excluding the single split-node).

\subsection{Counting Edges via the QASST}

In general, the number of edges in a graph can be inferred from its QASST decomposition by adding up the number of edges contributed by each quotient graph.
Note that this is not as simple as summing all edges from all quotient graphs since split-nodes do not exist in the original graph.
Edges between leaf-nodes in the same quotient graph directly correspond to edges in the original graph, but counting the edges between leaf-nodes from different quotient graphs requires that we correctly account for the behavior of split-nodes.
However, this task is much simpler in the distance-hereditary case since quotient graphs are either star or complete, for which the number of edges only depends on the number of nodes.
Specifically, a star or complete graph with $n$ vertices contains $n-1$ or $\frac{n(n-1)}{2}$ edges, respectively.

For a graph $G$ QASST equivalent to the complete $k$-partite graph, where $QASST(G)$ consists of a central quotient graph $Q_0$ surrounded by a ring of quotient graphs $Q_1,\cdots,Q_k$, the number of edges in the graph can be inferred from the QASST decomposition via the following formula. 

\begin{widetext}
\begin{eqnarray}\label{eq:complete_k-partite_QASST_edge_formula}
|E(G)|&=&\underbrace{\left[\sum_{i=1}^k\begin{cases}\frac{n_i(n_i-1)}{2}&\text{if $Q_i$ is c}\\0&\text{if $Q_i$ is sc}\\n_i-1&\text{if $Q_i$ is ss}\end{cases}\right]}_{\text{edges between leaf-nodes within fixed $Q_{i>0}$}}+\underbrace{\left[\sum_{e=(Q_i,Q_j)}^{E(Q_0)}\begin{cases}n_in_j&\text{if ($Q_i$ is c or sc) and ($Q_j$ is c or sc)}\\ n_i&\text{if ($Q_i$ is c or sc) and ($Q_j$ is ss)}\\ n_j&\text{if ($Q_i$ is ss) and ($Q_j$ is c or sc)}\\ 1&\text{if ($Q_i$ is ss) and ($Q_j$ is ss)}\end{cases}\right]}_{\text{edges between leaf-nodes from different $Q_{i>0}$ and $Q_{j>0}$}}
\end{eqnarray}
\end{widetext}

In this formula, the contribution to the number of edges in $G$ from each quotient graph is divided into two parts: the internal edges between leaf-nodes from each $Q_{i>0}$, and the edges between leaf-nodes from different quotient graphs. Since $Q_0$ contains only split-nodes, no edges from $Q_0$ directly correspond to edges in $G$. Rather, each edge in $Q_0$ represents a connection between two different quotient graphs, and hence the notation $e=(Q_i,Q_j)\in E(Q_0)$ used in the second summation of Equation~\ref{eq:complete_k-partite_QASST_edge_formula}.
The number of edges in $G$ corresponding to each edge in $Q_0$ will depend on the structure of the adjacent quotient graphs, for which there are several cases depending on whether each quotient graph is complete, star-center, or star-spoke.

For example, when $G=K_{n_1,\cdots,n_k}$ is complete $k$-partite, each quotient graph $Q_{i>0}$ is star-center, contributing no edges between leaf-nodes. The central quotient graph $Q_0$ is complete, meaning all pairs of other quotient graphs $Q_i$ and $Q_j$ have a total of $n_in_j$ edges passing between their leaf-nodes in the original graph $G$. This yields a total number of edges:
\begin{eqnarray}
|E(K_{n_1,\cdots,n_k})|&=&\sum_{i=1}^k\sum_{j>i}^kn_in_j.
\end{eqnarray}

By contrast, if $G=CS^r_{n_1,\cdots,n_k}$ is a clique-star, each quotient graph $Q_{i>0}$ is complete, contributing $\frac{n_i(n_i-1)}{2}$ edges. The central quotient graph $Q_0$ is $\text{sc}_r$ pointing towards $Q_r$, and hence the only edges are those which connect $Q_r$ to all other $Q_{i>0}$, each of which represents $n_rn_i$ edges in the original graph $G$. This yields a total number of edges:
\begin{eqnarray}
|E(CS^r_{n_1,\cdots,n_k})|&=&\left(\sum_{i=1}^k\frac{n_i(n_i-1)}{2}\right)+\left(\sum_{i=1\neq r}^kn_in_r\right).\nonumber\\
\end{eqnarray}

For one more comparison if $G=MR_{n_1,\cdots,n_k}$, we may also recover the formula for the number of edges in a multi-leaf repeater graph from Equation~\ref{eq:complete_k-partite_QASST_edge_formula}. In this case, each quotient graph $Q_{i>0}$ is star-spoke and contributes exactly $n_i-1$ edges. The central quotient graph $Q_0$ is complete, but in this case each edge in $E(Q_0)$ only contributes 1 edge to $E(G)$, giving:
\begin{eqnarray}
|E(MR_{n_1,\cdots,n_k})|&=&\left(\sum_{i=1}^k(n_i-1)\right)+\frac{k(k-1)}{2}.\nonumber\\
\end{eqnarray}

\subsection{Minimal Edge Representatives}
\label{sect:complet_k-partite_minimum_edge_representative}

Although these three cases have especially nice symmetry, in general they will not correspond to a minimum edge representative of the local equivalence class.
Among the quotient graphs $Q_1,\cdots,Q_k$, each of these may be complete, star-center, or star-spoke. For a fixed $Q_i$, we infer the minimum number of edges that it will contribute to the original graph based on these three possibilities. This includes both the internal edges between leaf-nodes in $Q_i$ as well as additional edges to other leaf-nodes in other quotient graphs.
The number of additional edges depends on the structure of $Q_0$.
For each edge $(Q_i,Q_j)\in E(Q_0)$, the number of these additional edges is multiplied.

If $Q_i$ is complete, it contributes $\frac{n_i(n_i-1)}{2}$ internal edges and at least $n_i$ additional edges. If $Q_i$ is star-center, it contributes no internal edges but at least $n_i$ additional edges. If $Q_i$ is star-spoke, it contributes $n_i-1$ internal edges and at least 1 additional edge.
Of these possibilities, the lowest overall contribution occurs when $Q_i$ is star-spoke.
Hence, to minimize the number of edges in $G$, it is usually desirable to choose a representative of the LC orbit for which as many quotient graphs $Q_{i>0}$ are star-spoke as possible.
However, whether or not this graph has a minimal number of edges also depends on how large $k$ is relative to $n_1,\cdots,n_k$.

In Tables~\ref{tab:complete_k-partite_quasst_equivalences} and \ref{tab:clique-star_quasst_equivalences}, we enumerated the three broad symmetry classes within the LC orbits of the complete $k$-partite graph and clique-star.
For each of these cases, the quotient graphs which are star-spoke are determined by the indices in a choice of index set $I\subseteq[k]$ meeting certain conditions.
With the goal of identifying QASST decompositions with as many star-spoke quotient graphs as possible (and hence LC orbit representatives with as few edges as possible), we now present modified tables for which $I$ is chosen to be as large as possible in each of the cases considered previously.
We also give the explicit formula for the number of edges in this graph based on Equation~\ref{eq:complete_k-partite_QASST_edge_formula}.

\begin{table*}[t]
\centering
\begin{tabular}{|c|c|l|c|}
\hline
$k$&Case&Rules for $QASST(G)$&Number of Edges in $G$\\
\hline
even&1&$\begin{array}{l}\text{$Q_0$ is c}\\\text{All other $Q_{i\neq0}$ are ss}\end{array}$&$\mathlarger{\tfrac{k(k-1)}{2}+\sum_{i=1}^k(n_i-1)}$\\
&&&\\
&2(j)&$\begin{array}{l}\text{$Q_0$ is $\text{sc}_j$ and $Q_j$ is sc}\\\text{One $Q_{\ell\neq j,0}$ is c}\\\text{All other $Q_{i\neq j,\ell,0}$ are ss}\end{array}$&$\mathlarger{n_j((k-2)+n_\ell)+\tfrac{n_{\ell}(n_{\ell}-1)}{2}+\sum_{i=1\neq j,\ell}^k(n_i-1)}$\\
&&&\\
&3(j)&$\begin{array}{l}\text{$Q_0$ is $\text{sc}_j$ and $Q_j$ is c}\\\text{All other $Q_{i\neq j,0}$ are ss}\end{array}$&$\mathlarger{n_j(k-1)+\tfrac{n_j(n_j-1)}{2}+\sum_{i=1\neq j}^k(n_i-1)}$\\
\hline
odd&1&$\begin{array}{l}\text{$Q_0$ is c}\\\text{One $Q_{j\neq0}$ is sc}\\\text{All other $Q_{i\neq j,0}$ are ss}\end{array}$&$\mathlarger{n_j(k-1)+\tfrac{(k-1)(k-2)}{2}+\sum_{i=1\neq j}^k(n_i-1)}$\\
&&&\\
&2(j)&$\begin{array}{l}\text{$Q_0$ is $\text{sc}_j$ and $Q_j$ is sc}\\\text{All other $Q_{i\neq j,0}$ are ss}\end{array}$&$\mathlarger{n_j(k-1)+\sum_{i=1\neq j}^k(n_i-1)}$\\
&&&\\
&3(j)&$\begin{array}{l}\text{$Q_0$ is $\text{sc}_j$ and $Q_j$ is c}\\\text{One $Q_{\ell\neq j,0}$ is c}\\\text{All other $Q_{i\neq j,\ell,0}$ are ss}\end{array}$&$\mathlarger{\tfrac{n_j(n_j-1)}{2}+\tfrac{n_{\ell}(n_{\ell}-1)}{2}+n_j((k-2)+n_{\ell})+\sum_{i=1\neq j,\ell}^k(n_i-1)}$\\
\hline
\end{tabular}
\caption{Candidate graphs for minimum edge representatives of ${\mathcal O}(K_{n_1,\cdots,n_k})$ based on the cases from Table~\ref{tab:complete_k-partite_quasst_equivalences}, where Equation~\ref{eq:complete_k-partite_QASST_edge_formula} computes the number of edges in each graph.}
\label{tab:complete_k-partite_LC_orbit_minimum_edge_representatives}
\end{table*}

\begin{table*}[t]
\centering
\begin{tabular}{|c|c|l|c|}
\hline
$k$&Case&Rules for $QASST(G)$&Number of Edges in $G$\\
\hline
even&1&$\begin{array}{l}\text{$Q_0$ is c}\\\text{One $Q_{j\neq0}$ is sc}\\\text{All other $Q_{i\neq0}$ are ss}\end{array}$&$\mathlarger{n_j(k-1)+\tfrac{(k-1)(k-2)}{2}+\sum_{i=1\neq j}^k(n_i-1)}$\\
&&&\\
&2(j)&$\begin{array}{l}\text{$Q_0$ is $\text{sc}_j$ and $Q_j$ is sc}\\\text{All other $Q_{i\neq j,0}$ are ss}\end{array}$&$\mathlarger{n_j(k-1)+\sum_{i=1\neq j}^k(n_i-1)}$\\
&&&\\
&3(j)&$\begin{array}{l}\text{$Q_0$ is $\text{sc}_j$ and $Q_j$ is c}\\\text{One $Q_{\ell\neq j,0}$ is c}\\\text{All other $Q_{i\neq j,\ell,0}$ are ss}\end{array}$&$\mathlarger{\tfrac{n_j(n_j-1)}{2}+\tfrac{n_{\ell}(n_{\ell}-1)}{2}+n_j((k-2)+n_{\ell})+\sum_{i=1\neq j,\ell}^k(n_i-1)}$\\
\hline
odd&1&$\begin{array}{l}\text{$Q_0$ is c}\\\text{All other $Q_{i\neq0}$ are ss}\end{array}$&$\mathlarger{\tfrac{k(k-1)}{2}+\sum_{i=1}^k(n_i-1)}$\\
&&&\\
&2(j)&$\begin{array}{l}\text{$Q_0$ is $\text{sc}_j$ and $Q_j$ is sc}\\\text{One $Q_{\ell\neq j,0}$ is c}\\\text{All other $Q_{i\neq j,\ell,0}$ are ss}\end{array}$&$\mathlarger{n_j((k-2)+n_\ell)+\tfrac{n_{\ell}(n_{\ell}-1)}{2}+\sum_{i=1\neq j,\ell}^k(n_i-1)}$\\
&&&\\
&3(j)&$\begin{array}{l}\text{$Q_0$ is $\text{sc}_j$ and $Q_j$ is c}\\\text{All other $Q_{i\neq j,0}$ are ss}\end{array}$&$\mathlarger{n_j(k-1)+\tfrac{n_j(n_j-1)}{2}+\sum_{i=1\neq j}^k(n_i-1)}$\\
\hline
\end{tabular}
\caption{Candidate graphs for minimum edge representatives of ${\mathcal O}(CS^r_{n_1,\cdots,n_k})$ based on the cases from Table~\ref{tab:clique-star_quasst_equivalences}, where Equation~\ref{eq:complete_k-partite_QASST_edge_formula} computes the number of edges in each graph.}
\label{tab:clique-star_LC_orbit_minimum_edge_representatives}
\end{table*}

The basic procedure is to choose an index set $I$ which is as large as possible and which satisfies the rules of the symmetry class.
Since the parity of $|I|$ matters, we cannot choose $I=[k]$ in all cases, and so we cannot always force each quotient graph $Q_1,\cdots, Q_k$ to be star-spoke, but we can choose which indices are excluded from $I$ if necessary.
With this in mind, let $j\in[k]$ be the index for which $n_j=\text{min}\{n_1,\cdots,n_k\}$, and let $\ell\in[k]\setminus\{j\}$ be the index for which $n_\ell=\text{min}\{n_1,\cdots,n_k\}\setminus\{n_j\}$.
That is, $n_j$ and $n_{\ell}$ are the smallest and second smallest numbers from among $n_1,\cdots,n_k$.
Since the parity of $|I|$ also plays a role in the symmetry class rules, we must also divide each of the previous cases into two subcases distinguishing between whether $k$ is even or odd.

Table~\ref{tab:complete_k-partite_LC_orbit_minimum_edge_representatives} shows the possible candidates for a minimum edge representative of ${\mathcal O}(K_{n_1,\cdots,n_k})$ and Table~\ref{tab:clique-star_LC_orbit_minimum_edge_representatives} shows those for ${\mathcal O}(CS^r_{n_1,\cdots,n_k})$, each consisting of six cases.
For both LC orbits, the candidate graph with the fewest edges will also depend on the relationship between $k$ and $n_j$.
We consider when $k$ is even or odd separately for each of the two orbits.

\begin{enumerate}
\item ${\mathcal O}(K_{n_1,\cdots,n_k})$ when $k$ is even.

Let $G_1$, $G_2$, and $G_3$ denote the three candidate graphs in the first part of Table~\ref{tab:complete_k-partite_LC_orbit_minimum_edge_representatives}, for which $k$ is even.
After simplifying, observe that $|E(G_2)|-|E(G_3)|=(n_j-1)(n_{\ell}-1)+\left(\frac{n_{\ell}(n_{\ell}-1)}{2}-\frac{n_j(n_j-1)}{2}\right)$. Since $2\leq n_j\leq n_{\ell}$, this quantity is always positive, and so $G_3$ always has fewer edges than $G_2$.

Trying the same trick to compute $|E(G_3)|-|E(G_1)|$ and then simplifying gives an expression for a hyperbola in variables $k$ and $n_j$:
\begin{eqnarray}\label{eq:even_complete_k-partite_MER_hyperbola}
f(k,n_j)&=&(n_j-1)(k-1)+\tfrac{(n_j-2)(n_j-1)}{2}-\tfrac{(k-2)(k-1)}{2}.\nonumber\\
\end{eqnarray}
When $f(k,n_j)>0$, $G_1$ has fewer edges than $G_3$.
When $f(k,n_j)<0$, $G_3$ has fewer edges than $G_1$.
Hence, the minimum edge representative only depends on the relationship between $n_j=\text{min}\{n_1,\cdots,n_k\}$ and $k$.

The formula given by Equation~\ref{eq:even_complete_k-partite_MER_hyperbola} is exact, but a simple approximation can also be inferred numerically by examining the linear asymptotes of the hyperbola.
For sufficiently large values of $k$ and $n_j$, we can approximate these two cases via linear functions, noting that these may not hold for small values.
If $3\leq k\lessapprox 2.4n_j-0.9$, then $G_1$ is the minimum edge representative; if $k\gtrapprox 2.4n_j-0.9$, then $G_3$ is the minimum edge representative.

Figure~\ref{fig:complete_k-partite_LC_orbit_MER_examples} shows a comparison of the graphs $G_1$ and $G_3$ with two examples, one for $f(k,n_j)>0$ and the other for $f(k,n_j)<0$. The difference in the number of edges is computed by $f(k,n_j)$.

\item ${\mathcal O}(K_{n_1,\cdots,n_k})$ when $k$ is odd.

Now let $G_1$, $G_2$, and $G_3$ denote the three candidate graphs in the second part of Table~\ref{tab:complete_k-partite_LC_orbit_minimum_edge_representatives}, for which $k$ is odd.
Using the same technique to calculate the difference in the number of edges for each candidate graph, we compute $|E(G_1)|-|E(G_2)|=\frac{(k-2)(k-1)}{2}$, which is always positive since $k\geq3$.
Hence, $G_2$ always has fewer edges than $G_1$.

Next, simplifying the difference in the number of edges between $G_3$ and $G_2$, we compute $|E(G_3)|-|E(G_2)|=\left(\frac{n_j(n_j-1)}{2}-n_j\right)+\left(\frac{n_{\ell}(n_{\ell}-1)}{2}-n_{\ell}\right)+n_jn_{\ell}+1$.
Since $2\leq n_j\leq n_{\ell}$, this quantity always positive.
Hence, $G_2$ also always has fewer edges than $G_3$.

We conclude that the graph $G_2$ is always the minimum edge representative of ${\mathcal O}(K_{n_1,\cdots,n_k})$ when $k$ is odd.

\item ${\mathcal O}(CS^r_{n_1,\cdots,n_k})$ when $k$ is even.

In the case of the clique-star, we use exactly the same argument as the complete $k$-partite graph, except referring to Table~\ref{tab:clique-star_LC_orbit_minimum_edge_representatives} instead.
Consider the first three graphs $G_1$, $G_2$, and $G_3$ for which $k$ is even.
Observe that the formulas for the number of edges in this first block of Table~\ref{tab:clique-star_LC_orbit_minimum_edge_representatives} exactly coincide with the second block of formulas from Table~\ref{tab:complete_k-partite_LC_orbit_minimum_edge_representatives}, as do the rules for the quotient graphs in the QASST.
Therefore, we may immediately conclude that $G_2$ defines the minimum edge representative in this case.

\item ${\mathcal O}(CS^r_{n_1,\cdots,n_k})$ when $k$ is odd.

Finally, we let $G_1$, $G_2$, and $G_3$ denote the second batch of graphs in Table~\ref{tab:clique-star_LC_orbit_minimum_edge_representatives}, for which $k$ is odd. Again, observe that the corresponding edge formulas exactly coincide with the formulas from the first block of Table~\ref{tab:complete_k-partite_LC_orbit_minimum_edge_representatives}.
Hence, the minimum edge representative is either $G_1$ or $G_3$ and will depend on the relationship of $k$ and $n_j$ defined by Equation~\ref{eq:even_complete_k-partite_MER_hyperbola}.
If $f(k,n_j)>0$, then $G_1$ has fewer edges than $G_3$.
If $f(k,n_j)<0$, then $G_3$ has fewer edges than $G_1$.

\end{enumerate}

\begin{figure*}[tp]
\centering

\begin{subfigure}{0.5\textwidth}
\centering
\includegraphics[width=0.35\linewidth,page=1]{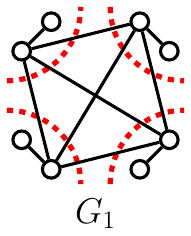}
\includegraphics[width=0.35\linewidth,page=2]{Figures/Appendix_Edge_Counting_and_Isomorphisms.pdf}\\
$f(4,2)=0$ and so both minimal\\
$|E(G_1)|=10$ and $|E(G_3)|=10$
\caption{
$k=4$ and $n_i=2$
}
\label{fig:complete_k-partite_LC_orbit_even_MER_k4n2}
\end{subfigure}\begin{subfigure}{0.5\textwidth}
\centering
\includegraphics[width=0.35\linewidth,page=3]{Figures/Appendix_Edge_Counting_and_Isomorphisms.pdf}
\includegraphics[width=0.35\linewidth,page=4]{Figures/Appendix_Edge_Counting_and_Isomorphisms.pdf}\\
$f(4,3)=4>0$ and so $G_1$ is minimal\\
$|E(G_1)|=14$ and $|E(G_3)|=18$
\caption{
$k=4$ and $n_i=3$
}
\label{fig:complete_k-partite_LC_orbit_even_MER_k4n3}
\end{subfigure}

\begin{subfigure}{0.5\textwidth}
\centering
\includegraphics[width=0.35\linewidth,page=5]{Figures/Appendix_Edge_Counting_and_Isomorphisms.pdf}
\includegraphics[width=0.35\linewidth,page=6]{Figures/Appendix_Edge_Counting_and_Isomorphisms.pdf}\\
$f(6,2)=-5<0$ and so $G_3$ is minimal\\
$|E(G_1)|=21$ and $|E(G_3)|=16$
\caption{
$k=6$ and $n_i=2$
}
\label{fig:complete_k-partite_LC_orbit_even_MER_k6n2}
\end{subfigure}\begin{subfigure}{0.5\textwidth}
\centering
\includegraphics[width=0.4\linewidth,page=7]{Figures/Appendix_Edge_Counting_and_Isomorphisms.pdf}
\includegraphics[width=0.4\linewidth,page=8]{Figures/Appendix_Edge_Counting_and_Isomorphisms.pdf}\\
$f(6,3)=1>0$ and so $G_1$ is minimal\\
$|E(G_1)|=27$ and $|E(G_3)|=28$
\caption{
$k=6$ and $n_i=3$
}
\label{fig:complete_k-partite_LC_orbit_even_MER_k6n3}
\end{subfigure}

\begin{subfigure}{0.25\textwidth}
\centering
\includegraphics[width=0.8\textwidth,page=9]{Figures/Appendix_Edge_Counting_and_Isomorphisms.pdf}
$|E(G_2)|=6$
\caption{
$k=3$ and $n_i=2$
}
\label{fig:complete_k-partite_LC_orbit_odd_MER_k3n2}
\end{subfigure}\begin{subfigure}{0.25\textwidth}
\centering
\includegraphics[width=0.8\textwidth,page=10]{Figures/Appendix_Edge_Counting_and_Isomorphisms.pdf}
$|E(G_2)|=10$
\caption{
$k=3$ and $n_i=3$
}
\label{fig:complete_k-partite_LC_orbit_odd_MER_k3n3}
\end{subfigure}\begin{subfigure}{0.25\textwidth}
\centering
\includegraphics[width=0.8\textwidth,page=11]{Figures/Appendix_Edge_Counting_and_Isomorphisms.pdf}
$|E(G_2)|=12$
\caption{
$k=5$ and $n_i=2$
}
\label{fig:complete_k-partite_LC_orbit_odd_MER_k5n2}
\end{subfigure}\begin{subfigure}{0.25\textwidth}
\centering
\includegraphics[width=0.8\textwidth,page=12]{Figures/Appendix_Edge_Counting_and_Isomorphisms.pdf}
$|E(G_2)|=20$
\caption{
$k=5$ and $n_i=3$
}
\label{fig:complete_k-partite_LC_orbit_odd_MER_k5n3}
\end{subfigure}

\caption{
Examples of minimal edge representatives for ${\mathcal O}(K_{n_1,\cdots,n_k})$ for various values of $k$ and $n_1=\cdots=n_k$.
When $k$ is even (a--d), the minimal edge graph is either $G_1$ or $G_3$ from the first block of Table~\ref{tab:complete_k-partite_LC_orbit_minimum_edge_representatives} depending on the value of $f(k,n_j)$ by Equation~\ref{eq:even_complete_k-partite_MER_hyperbola}.
When $k$ is odd (e--h), $G_2$ from this table always has the fewest edges.
}
\label{fig:complete_k-partite_LC_orbit_MER_examples}
\end{figure*}

We formalize these cases with a theorem to characterize the minimum edge representatives of the two local equivalence classes.
Figure~\ref{fig:complete_k-partite_LC_orbit_MER_examples} shows a number of examples of these graphs for small values of $k$ and $n_i$.
Note that such a graph is non-unique in general; there could be many isomorphic locally equivalent graphs which achieve this minimum.
Furthermore, there may be non-isomorphic graphs in the LC orbit which are both minimal with respect to the number of edges when $f(k,n_j)=0$, as in the example of Figure~\ref{fig:complete_k-partite_LC_orbit_even_MER_k4n2}.

\begin{theorem}
\label{thm:complete_k-partite_minimal_edge_representative}
Let $K_{n_1,\cdots,n_k}$ be a complete $k$-partite graph.
Take the smallest term to be $n_j=\text{min}\{n_1,\cdots,n_k\}$ and let $f(k,n_j)$ be as specified by Equation~\ref{eq:even_complete_k-partite_MER_hyperbola}.
Let $G\in{\mathcal O}(K_{n_1,\cdots,n_k})$ be a minimum edge representative of the local equivalence class with a QASST decomposition defined by the quotient graphs $Q_0,Q_1,\cdots,Q_k$.
The structure of $G$ can be identified via its quotient graphs according to the following cases.
\begin{enumerate}
\item If $k$ is even and $f(k,n_j)>0$, then $Q_0$ is complete and each $Q_{i\neq0}$ is star-spoke (in this case, $G$ is a multi-leaf repeater graph). The number of edges is
\begin{eqnarray}
|E(G)|=\frac{k(k-1)}{2}+\sum_{i=1}^k(n_i-1).
\end{eqnarray}
\item If $k$ is even and $f(k,n_j)<0$, then $Q_0$ is a star pointing towards $Q_j$, $Q_j$ is complete, and all other $Q_{i\neq j,0}$ are star-spoke. The number of edges is
\begin{eqnarray}
|E(G)|&=&n_j(k-1)+\frac{n_j(n_j-1)}{2}+\sum_{i=1\neq j}^k(n_i-1).\nonumber\\
\end{eqnarray}
If $f(k,n_j)=0$, then Case 1 and 2 define graphs with the same number of edges, both of which are minimal in the LC orbit.
\item If $k$ is odd, then $Q_0$ is a star pointing towards $Q_j$, $Q_j$ is star-center, and all other $Q_{i\neq j,0}$ are star-spoke. The number of edges is
\begin{eqnarray}
|E(G)|=n_j(k-1)+\sum_{i=1\neq j}^k(n_i-1).
\end{eqnarray}
\end{enumerate}
\end{theorem}

\begin{theorem}
\label{thm:clique-star_minimal_edge_representative}
Let $CS^r_{n_1,\cdots,n_k}$ be a clique-star.
Take the smallest term to be $n_j=\text{min}\{n_1,\cdots,n_k\}$ and let $f(k,n_j)$ be as specified by Equation~\ref{eq:even_complete_k-partite_MER_hyperbola}.
Let $G\in{\mathcal O}(CS^r_{n_1,\cdots,n_k})$ be a minimum edge representative of the local equivalence class with a QASST decomposition defined by the quotient graphs $Q_0,Q_1,\cdots,Q_k$.
The structure of $G$ can be identified via its quotient graphs according to the following cases.
\begin{enumerate}
\item If $k$ is even, then $Q_0$ is a star pointing towards $Q_j$, $Q_j$ is star-center, and all other $Q_{i\neq j,0}$ are star-spoke. The number of edges is
\begin{eqnarray}
|E(G)|=n_j(k-1)+\sum_{i=1\neq j}^k(n_i-1).
\end{eqnarray}
\item If $k$ is odd and $f(k,n_j)>0$, then $Q_0$ is complete and all other $Q_{i\neq0}$ are star-spoke (in this case, $G$ is a multi-leaf repeater graph). The number of edges is
\begin{eqnarray}
|E(G)|=\frac{k(k-1)}{2}+\sum_{i=1}^k(n_i-1).
\end{eqnarray}
\item If $k$ is odd and $f(k,n_j)<0$, then $Q_0$ is a star pointing towards $Q_j$, $Q_j$ is complete, and all other $Q_{i\neq j,0}$ are star-spoke. The number of edges is
\begin{eqnarray}
|E(G)|&=&n_j(k-1)+\frac{n_j(n_j-1)}{2}+\sum_{i=1\neq j}^k(n_i-1).\nonumber\\
\end{eqnarray}
If $f(k,n_j)=0$, then Case 2 and 3 define graphs with the same number of edges, both of which are minimal in the LC orbit.
\end{enumerate}
\end{theorem}

\section{Maximum Vertex Degrees for Graphs QASST Equivalent to $K_{n_1,\cdots,n_k}$}
\label{app:max_vertex_degrees}

Maximum vertex degree is another graph parameter relevant for some practical applications. For example, this parameter can be used to determine the lower bound of the circuit depth required to implement a corresponding graph state~\cite{cabello2011optimal}.
This value is equivalent to the \textit{edge chromatic number} of the graph (the minimal number of distinct colors needed to label each edge in the graph such that no two edges incident to the same vertex have the same color).
Vizing proved that this number is either $\Delta$ (the maximum vertex degree of the graph) or $\Delta+1$~\cite{vizing1964estimate}.
Much like with edge counting, we can derive a general formula for the maximum vertex degree for graphs locally equivalent to a complete $k$-partite graph or a clique-star by using the QASST, and then search for the minimum $\Delta$ across the LC orbit.

\subsection{Counting Vertex Degrees via the QASST}

Using the QASST, we will derive a general formula for the maximum vertex degree.
We do this by computing the maximum vertex degrees of leaf-nodes for each quotient graph individually (split-nodes are excluded and degrees are with respect to the entire graph, not just the quotient graph).
The maximum vertex degree of the entire graph is then taken to be the maximum leaf-node degree across all quotient graphs.

Before defining a general function, we introduce the following notation.
For a graph $G$, let the maximum vertex degree of this graph be denoted as $\Delta(G)$.
Let $G$ be any graph QASST equivalent to $K_{n_1,\cdots,n_k}$ with split decomposition given by $QASST(G)=\{Q_0,Q_1,\cdots,Q_k\}$.
Recall that the edges in $Q_0$ represent connections between other quotient graphs, and so let $e=(Q_i,Q_j)\in E(Q_0)$ be a shorthand for denoting these.
For a fixed choice of $Q_0$, define a function $\delta_{e=(Q_i,Q_j)}^{E(Q_0)}$ according to the following rule:
\begin{eqnarray}
\delta_{e=(Q_i,Q_j)}^{E(Q_0)}&=&\begin{cases}1&\text{if there exists }e=(Q_i,Q_j)\in E(Q_0)\\0&\text{otherwise}\end{cases}.\nonumber\\
\end{eqnarray}

Let $Q_i\in\{Q_1,\cdots,Q_k\}$ be a quotient graph. The degree of any leaf-node $v\in V(Q_i)$ is the sum of the internal contributions to degree from other leaf-nodes in $Q_i$ and the external contributions to degree from leaf-nodes in other quotient graphs $Q_j$ connected to $Q_i$ by an edge $(Q_i,Q_j)\in E(Q_0)$.
There are three cases for the internal degree, depending on the structure of $Q_i$ (c, sc, or ss). If $Q_i$ is complete, all leaf-nodes have the same internal degree $n_i-1$. If $Q_i$ is star-center, all leaf-nodes have the same internal degree 0. In these two cases, the leaf-nodes of $Q_i$ have identical external neighborhoods as determined by $Q_0$. If $Q_i$ is star-spoke, the center leaf-node of the star has internal degree $n_i-1$, but the spoke leaf-nodes each have degree 1. In other words, the center leaf-node has the maximum degree of the quotient graph in this case.
The contributions to external leaf-node degree are likewise determined by the structure of any connected quotient graphs $Q_j$.
This gives rise to the following formula for the maximum vertex degree of any \textit{leaf-node} in $Q_i$ (we exclude split-nodes).
\begin{widetext}
\begin{eqnarray}\label{eq:mvd_quotient}
\Delta(Q_i)&=&\underbrace{\begin{cases}0&\text{if $Q_i$ is sc}\\n_i-1&\text{if $Q_i$ is c or ss}\end{cases}}_{\text{internal contribution to degree}}+
\underbrace{\sum_{j=1\neq i}^k\delta_{e=(Q_i,Q_j)}^{E(Q_0)}\begin{cases}1&\text{if $Q_j$ is ss}\\ n_j&\text{if $Q_j$ is c or sc}\end{cases}}_{\text{external contribution to degree}}
\end{eqnarray}
\end{widetext}
Hence, the maximum vertex degree of $G$ is computed as
\begin{eqnarray}
\Delta(G)&=&\max\{\Delta(Q_1),\cdots,\Delta(Q_k)\}.
\end{eqnarray}

\subsection{Simplifying Quotient Graph Maximum Vertex Degree by Cases}

When the QASST structure of $G$ is known, the formula can be simplified.
We derive three simplifications for $\Delta(G)$ based on the three broad symmetry classes described in Tables~\ref{tab:complete_k-partite_quasst_equivalences} and \ref{tab:clique-star_quasst_equivalences}.
Let $I\subseteq[k]$ be a set of indices corresponding to the star-spoke quotient graphs of $G$. Let $I^C=[k]\setminus I$ denote its complement.
For a fixed index $j\in[k]$, let $I^{\setminus j}\subseteq[k]\setminus\{j\}$ with complement $(I^{\setminus j})^C=[k]\setminus(I\cup\{j\})$.
Define the following three cases.
\begin{enumerate}
\item $Q_0$ is c, $Q_i$ is ss for $i\in I$, and $Q_i$ is sc for $i\in I^C$.
\item For $j\in[k]$, $Q_0$ is $\text{sc}_j$, $Q_j$ is sc, $Q_i$ is ss for $i\in I^{\setminus j}$, and $Q_i$ is c for $i\in(I^{\setminus j})^C$.
\item For $j\in[k]$, $Q_0$ is $\text{sc}_j$, $Q_j$ is c, $Q_i$ is ss for $i\in I^{\setminus j}$, and $Q_i$ is c for $i\in(I^{\setminus j})^C$.
\end{enumerate}
These correspond exactly to the symmetry classes 1, 2(j), and 3(j) considered earlier, where we do not make a distinction between $|I|$ being even or odd.

Assume case 1, so that $G$ has split decomposition $QASST(G)=\{Q_0,Q_1,\cdots,Q_k\}$ with $Q_0=\text{c}$, $Q_i=\text{ss}$ for $i\in I\subseteq[k]$, and $Q_i=\text{sc}$ for $i\in I^C$.
Since $Q_0$ is complete, it contains all edges and so $\delta_{e=(Q_i,Q_j)}^{E(Q_0)}=1$ for all choices of indices.
For any quotient graph $Q_{i\neq0}$, we may simplify Equation~\ref{eq:mvd_quotient} as follows. Here, we indicate in a strike-through in red the conditions of Equation~\ref{eq:mvd_quotient} that cannot occur in case 1.
\begin{widetext}
\begin{eqnarray}\label{eq:mvd1_Qi}
\Delta_1(Q_i)&=&\begin{cases}0&\text{if $Q_i$ is sc}\\n_i-1&\text{if $Q_i$ is ss or {\color{red}\sout{c}}}\end{cases}+\sum_{j=1\neq i}^k\begin{cases}1&\text{if $Q_j$ is ss}\\ n_j&\text{if $Q_j$ is {\color{red}\sout{c}} or sc}\end{cases}\nonumber\\
&=&\begin{cases}0&\text{if $i\in I^C$}\\n_i-1&\text{if $i\in I$}\end{cases}+\sum_{j=1\neq i}^k\begin{cases}1&\text{if $j\in I$}\\ n_j&\text{if $j\in I^C$}\end{cases}\nonumber\\
&=&\left(\sum_{j\in I^C\setminus\{i\}}n_j\right)+|I\setminus\{i\}|+\begin{cases}0&\text{if $i\in I^C$}\\n_i-1&\text{if $i\in I$}\end{cases}
\end{eqnarray}
\end{widetext}

Assume case 2($j$) for some fixed choice of $j\in[k]$, so that $Q_0=\text{sc}_j$, $Q_j=\text{sc}$, $Q_i=ss$ for $i\in I^{\setminus j}$ and $Q_i=\text{c}$ for $i\in(I^{\setminus j})^C$. In this case, the edges of $Q_0$ represent connections with $Q_j$ with all other quotient graphs, but there exist no connections between other pairs of quotient graphs.
Hence, in this case we have that $\delta_{e=(Q_i,Q_{\ell})}^{E(Q_0)}=\delta_{\ell}^j$.
Now we may simplify Equation~\ref{eq:mvd_quotient} for $Q_j$ and $Q_{i\neq j}$ as follows.
\begin{widetext}
\begin{eqnarray}\label{eq:mvd2j_Qj}
\Delta_{\text{2($j$)}}(Q_j)&=&\begin{cases}0&\text{if $Q_j$ is sc}\\n_j-1&\text{if $Q_j$ is {\color{red}\sout{ss or c}}}\end{cases}+\sum_{\ell=1\neq j}^k\begin{cases}1&\text{if $Q_\ell$ is ss}\\ n_\ell&\text{if $Q_\ell$ is c or {\color{red}\sout{sc}}}\end{cases}\nonumber\\
&=&0+\sum_{\ell=1\neq j}^k\begin{cases}1&\text{if $\ell\in I^{\setminus j}$}\\ n_{\ell}&\text{if $\ell\in (I^{\setminus j})^C$}\end{cases}\nonumber\\
&=&|I^{\setminus j}|+\sum_{\ell\in(I^{\setminus j})^C}n_{\ell}
\end{eqnarray}
\begin{eqnarray}\label{eq:mvd2j_Qi}
\Delta_{\text{2($j$)}}(Q_{i\neq j})&=&\begin{cases}0&\text{if $Q_i$ is {\color{red}\sout{sc}}}\\n_{i}-1&\text{if $Q_i$ is ss or c}\end{cases}+\delta_{\ell}^j\sum_{\ell=1\neq i}^k\begin{cases}1&\text{if $Q_\ell$ is ss}\\ n_{\ell}&\text{if $Q_\ell$ is c or sc}\end{cases}\nonumber\\
&=&n_i-1+n_j
\end{eqnarray}

Finally, assume case 3(j) for some fixed choice of $j\in[k]$, so that $Q_0=\text{sc}_j$, $Q_j=\text{c}$, $Q_i=ss$ for $i\in I^{\setminus j}$ and $Q_i=\text{c}$ for $i\in(I^{\setminus j})^C$.
Again we have that $\delta_{e=(Q_i,Q_{\ell})}^{E(Q_0)}=\delta_{\ell}^j$.
As in the preceding case, we compute the maximum vertex degree for $Q_j$ and $Q_{i\neq j}$ separately to derive the following.
\begin{eqnarray}\label{eq:mvd3j_Qj}
\Delta_{\text{3($j$)}}(Q_j)&=&\begin{cases}0&\text{if $Q_j$ is {\color{red}\sout{sc}}}\\n_j-1&\text{if $Q_j$ is {\color{red}\sout{ss}} or c}\end{cases}+\sum_{\ell=1\neq j}^k\begin{cases}1&\text{if $Q_\ell$ is ss}\\ n_\ell&\text{if $Q_\ell$ is c or {\color{red}\sout{sc}}}\end{cases}\nonumber\\
&=&(n_j-1)+\sum_{\ell=1\neq j}^k\begin{cases}1&\text{if $\ell\in I^{\setminus j}$}\\ n_{\ell}&\text{if $\ell\in (I^{\setminus j})^C$}\end{cases}\nonumber\\
&=&(n_j-1)+|I^{\setminus j}|+\sum_{\ell\in(I^{\setminus j})^C}n_{\ell}
\end{eqnarray}
\begin{eqnarray}\label{eq:mvd3j_Qi}
\Delta_{\text{3($j$)}}(Q_{i\neq j})&=&\begin{cases}0&\text{if $Q_i$ is {\color{red}\sout{sc}}}\\n_{i}-1&\text{if $Q_i$ is ss or c}\end{cases}+\delta_{\ell}^j\sum_{\ell=1\neq i}^k\begin{cases}1&\text{if $Q_\ell$ is ss}\\ n_{\ell}&\text{if $Q_\ell$ is c or {\color{red}\sout{sc}}}\end{cases}\nonumber\\
&=&n_i-1+n_j
\end{eqnarray}
\end{widetext}
Notice that $\Delta(Q_{i\neq j})$ is the same for both cases 2(j) and 3(j).

\subsection{Minimizing the Maximum Vertex Degree across the LC Orbit}

In \ref{app:edge_counting}, we search the LC orbit for a representative graph with a minimal number of edges.
Now we will do something similar for the maximum vertex degree.
Examining the preceding equations, we see that the maximum vertex degree for any given quotient graph can be minimized by choosing $I$ or $I^{\setminus j}$ to be as large as possible; this corresponds to maximizing the number of star-spoke components. 

By Equation~\ref{eq:mvd_quotient}, the external contribution to the degree of a leaf-node in a quotient graph $Q_i$ is determined solely by the adjacencies of $Q_i$ in the central quotient graph $Q_0$: an adjacency to a star-spoke quotient graph contributes 1, whereas an adjacency to a complete or star-center quotient graph $Q_j$ contributes $n_j$. Consequently, for a fixed choice of the type of each $Q_i$, enlarging $I$ (equivalently, reducing $I^C$) replaces contributions of the form $n_j$ by 1, thereby decreasing the external contribution and tending to reduce $\Delta(G)$. In the global minimization of $\Delta(G)$, this decrease must be weighed against the possible increase in the internal contribution when a star-center component is replaced by a star-spoke component.

There will be several cases to consider depending on the local equivalence class and parity of $k$.
With this in mind, let $j,t,\ell\in[k]$ be indices such that $n_j=\min\{n_1,\cdots,n_k\}$, $n_t=\min\{n_1,\cdots,n_k\}\setminus\{n_j\}$, and $n_{\ell}=\max\{n_1,\cdots,n_k\}$ (that is, the indices corresponding to the lowest, the second lowest, and maximum $n_i$).
Note that $k\geq3$, so these three indices exist and are distinct.

Consider a graph $G\cong_{LC}K_{n_1,\cdots,n_k}$, where we assume $k$ is even. We consider separately choices of the index set $I$ for the three cases of Table~\ref{tab:complete_k-partite_quasst_equivalences}.
\begin{enumerate}
\item $I=[k]$, so that $|I|=k$ is even, with quotient graphs satisfying $Q_0=\text{c}$ and $Q_{i\neq0}=\text{sc}$. Note that $G$ is a multi-leaf repeater graph if it has this form.
We compute the maximum vertex degree of $G$ by taking a maximum across each of the quotient graphs using Equation~\ref{eq:mvd1_Qi}.
\begin{eqnarray}\label{eq:mvd1_G_ckp_ek}
\Delta_1(G)&=&\max_{i\in[k]}\{\Delta_1(Q_i)\}\nonumber\\
&=&\max_{i\in[k]}\{(k-1)+(n_i-1)\}\nonumber\\
&=&n_{\ell}+k-2
\end{eqnarray}
\item $I^{\setminus j}=[k]\setminus\{j,t\}$, so that $|I^{\setminus j}|$ is even, with quotient graphs satisfying $Q_0=\text{sc}_j$, $Q_j=\text{sc}$, $Q_t=\text{c}$, and $Q_{i\neq0,j,t}=\text{ss}$.
Equations~\ref{eq:mvd2j_Qj} and \ref{eq:mvd2j_Qi} simplify as follows.
\begin{eqnarray}
\Delta_{2(j)}(Q_j)&=&k-2+n_t\\
\Delta_{2(j)}(Q_{i\neq j})&=&n_i-1+n_j
\end{eqnarray}
We use these to compute the maximum vertex degree of $G$.
\begin{eqnarray}\label{eq:mvd2j_G_ckp_ek}
\Delta_{2(j)}(G)&=&\max_{i\in[k]}\{\Delta_{2(j)}(Q_i)\}\nonumber\\
&=&\max\{\Delta_{2(j)}(Q_j),\Delta_{2(j)}(Q_\ell)\}\nonumber\\
&=&\max\{(k-2+n_t),(n_{\ell}-1+n_j)\}
\end{eqnarray}
\item $I^{\setminus j}=[k]\setminus\{j\}$, so that $|I^{\setminus j}|$ is odd with $Q_0=\text{sc}_j$, $Q_j=\text{c}$, and $Q_{i\neq0,j}=\text{ss}$. Equations~\ref{eq:mvd3j_Qj} and \ref{eq:mvd3j_Qi} simplify as follows.
\begin{eqnarray}
\Delta_{3(j)}(Q_j)&=&n_j+k-2\\
\Delta_{3(j)}(Q_{i\neq j})&=&n_i-1+n_j
\end{eqnarray}
Again we use these to compute the maximum vertex degree of $G$.
\begin{eqnarray}\label{eq:mvd3j_G_ckp_ek}
\Delta_{3(j)}(G)&=&\max_{i\in[k]}\{\Delta_{3(j)}(Q_i)\}\nonumber\\
&=&\max\{\Delta_{3(j)}(Q_j),\Delta_{3(j)}(Q_\ell)\}\nonumber\\
&=&\max\{(n_j+k-2),(n_{\ell}-1+n_j)\}
\end{eqnarray}
\end{enumerate}

Now consider a graph $G\cong_{LC}K_{n_1,\cdots,n_k}$, where we assume $k$ is odd. We derive an analogous computation based on the index set $I$ for the three cases of Table~\ref{tab:complete_k-partite_quasst_equivalences}.
\begin{enumerate}
\item $I=[k]\setminus\{j\}$, so that $|I|$ is even, with quotient graphs satisfying $Q_0=\text{c}$, $Q_j=\text{sc}$, and $Q_{i\neq0,j}=\text{ss}$. Equation~\ref{eq:mvd1_Qi} simplifies as follows for $Q_j$ and $Q_{i\neq j}$.
\begin{eqnarray}
\Delta_1(Q_j)&=&k-1\\
\Delta_1(Q_{i\neq j})&=&n_j+(k-2)+(n_i-1)
\end{eqnarray}
We use these to compute the maximum vertex degree of $G$.
\begin{eqnarray}\label{eq:mvd1_G_ckp_ok}
\Delta_1(G)&=&\max_{i\in[k]}\{\Delta_1(Q_i)\}\nonumber\\
&=&\max\{\Delta_1(Q_j),\Delta_1(Q_\ell)\}\nonumber\\
&=&\max\{(k-1),(n_i+n_j+k-3)\}\nonumber\\
&=&n_i+n_j+k-3
\end{eqnarray}
$I^{\setminus j}=[k]\setminus\{j\}$, so that $|I^{\setminus j}|$ is even with $Q_0=\text{sc}_j$, $Q_j=\text{sc}$, and $Q_{i\neq j}=\text{ss}$. Equations~\ref{eq:mvd2j_Qj} and \ref{eq:mvd2j_Qi} simplify as follows.
\begin{eqnarray}
\Delta_{2(j)}(Q_j)&=&k-1\\
\Delta_{2(j)}(Q_{i\neq j})&=&n_i-1+n_j
\end{eqnarray}
We use these to compute the maximum vertex degree of $G$.
\begin{eqnarray}\label{eq:mvd2j_G_ckp_ok}
\Delta_{2(j)}(G)&=&\max_{i\in[k]}\{\Delta_{2(j)}(Q_i)\}\nonumber\\
&=&\max\{\Delta_{2(j)}(Q_j),\Delta_{2(j)}(Q_{\ell})\}\nonumber\\
&=&\max\{(k-1),(n_{\ell}-1+n_j)\}
\end{eqnarray}
\item $I^{\setminus j}=[k]\setminus\{j,t\}$, so that $|I^{\setminus j}|$ is odd with $Q_0=\text{sc}_j$, $Q_j=\text{c}$, $Q_t=\text{c}$, and $Q_{i\neq0,j,t}=\text{ss}$. Equations~\ref{eq:mvd3j_Qj} and \ref{eq:mvd3j_Qi} simplify as follows.
\begin{eqnarray}
\Delta_{3(j)}(Q_j)&=&(n_j-1)+(k-2)+n_t\nonumber\\
&=&n_j+n_t+k-3\\
\Delta_{3(j)}(Q_{i\neq j})&=&n_i-1+n_j
\end{eqnarray}
Finally, we use these to compute the maximum vertex degree of $G$.
\begin{eqnarray}
\Delta_{3(j)}(G)&=&\max_{i\in[k]}\{\Delta_{3(j)}(Q_i)\}\nonumber\\
&=&\max\{\Delta_{3(j)}(Q_j),\Delta_{3(j)}(Q_\ell)\}\nonumber\\
&=&\max\{(n_j+n_t+k-3),(n_{\ell}-1+n_j)\}\nonumber\\
\end{eqnarray}
\end{enumerate}

If $G$ is a graph locally equivalent to $K_{n_1,\cdots,n_k}$ which minimizes the maximum vertex degree across the LC orbit, then it must belong to one of the three QASST symmetry cases 1, 2(j), or 3(j) considered above.
However, the values of $n_j,n_t,n_{\ell},k$ and the parity of $k$ determine exactly which case achieves the minimum and the value of this minimum.
Table~\ref{tab:mvd_cases_complete_k-partite} summarizes these possibilities and their conditions.
An analogous approach can be used to examine the maximum vertex degrees of graphs locally equivalent to a clique-star. The derived formulas are identical except reversed for the parity of $k$.
These results are included in Table~\ref{tab:mvd_cases_clique-star} for completion but without an explicit derivation shown here.
We state these facts here as theorems.

\begin{table*}[t]
\centering
\begin{tabular}{|c|c|c|l|l|}
\hline
LC Orbit of $G$&$k$&Case&Rules for $QASST(G)$&Maximum Vertex Degree of $G$\\
\hline
&&&&\\
${\mathcal O}(K_{n_1,\cdots,n_k})$&even&$1$&$\begin{array}{l}\text{$Q_0$ is c}\\\text{All other $Q_{i}$ are ss}\end{array}$&$n_{\ell}+k-2$\\
&&&&\\
&&$2(j)$&$\begin{array}{l}\text{$Q_0$ is $\text{sc}_j$ and $Q_{j}$ is sc}\\\text{$Q_t$ is c}\\\text{All other $Q_i$ are ss}\end{array}$&$\max\left\{\begin{array}{l}(k-1+n_t),\\(n_\ell-1+n_j)\end{array}\right\}$\\
&&&&\\
&&$3(j)$&$\begin{array}{l}\text{$Q_0$ is $\text{sc}_j$ and $Q_{j}$ is c}\\\text{All other $Q_i$ are ss}\end{array}$&$\max\left\{\begin{array}{l}(n_j+k-2),\\(n_\ell-1+n_j)\end{array}\right\}$\\
&&&&\\
\hline
&&&&\\
${\mathcal O}(K_{n_1,\cdots,n_k})$&odd&1&$\begin{array}{l}\text{$Q_0$ is c and $Q_j$ is sc}\\\text{All other $Q_{i}$ are ss}\end{array}$&$n_{\ell}+n_j+k-3$\\
&&&&\\
&&$2(j)$&$\begin{array}{l}\text{$Q_0$ is $\text{sc}_j$ and $Q_{j}$ is sc}\\\text{All other $Q_i$ are ss}\end{array}$&$\max\left\{\begin{array}{l}(k-1),\\(n_\ell-1+n_j)\end{array}\right\}$\\
&&&&\\
&&$3(j)$&$\begin{array}{l}\text{$Q_0$ is $\text{sc}_j$ and $Q_{j}$ is c}\\\text{$Q_t$ is c}\\\text{All other $Q_i$ are ss}\end{array}$&$\max\left\{\begin{array}{l}(n_j+n_t+k-3),\\(n_\ell-1+n_j)\end{array}\right\}$\\
&&&&\\
\hline
\end{tabular}
\caption{
The possibilities for a graph $G$ locally equivalent to a complete $k$-partite graph $K_{n_1,\cdots,n_k}$ which is minimal with respect to maximum vertex degree across the entire LC orbit.
The exact case depends on $k$ and its parity and the values of $n_j=\min\{n_1,\cdots,n_k\}$, $n_t=\min\{n_1,\cdots,n_k\}\setminus\{n_j\}$, and $n_\ell=\max\{n_1,\cdots,n_k\}$.
}
\label{tab:mvd_cases_complete_k-partite}
\end{table*}

\begin{table*}[t]
\centering
\begin{tabular}{|c|c|c|l|l|}
\hline
LC Orbit of $G$&$k$&Case&Rules for $QASST(G)$&Maximum Vertex Degree of $G$\\
\hline
&&&&\\
${\mathcal O}(CS^r_{n_1,\cdots,n_k})$&even&1&$\begin{array}{l}\text{$Q_0$ is c and $Q_j$ is sc}\\\text{All other $Q_{i}$ are ss}\end{array}$&$n_{\ell}+n_j+k-3$\\
&&&&\\
&&$2(j)$&$\begin{array}{l}\text{$Q_0$ is $\text{sc}_j$ and $Q_{j}$ is sc}\\\text{All other $Q_i$ are ss}\end{array}$&$\max\left\{\begin{array}{l}(k-1),\\(n_\ell-1+n_j)\end{array}\right\}$\\
&&&&\\
&&$3(j)$&$\begin{array}{l}\text{$Q_0$ is $\text{sc}_j$ and $Q_{j}$ is c}\\\text{$Q_t$ is c}\\\text{All other $Q_i$ are ss}\end{array}$&$\max\left\{\begin{array}{l}(n_j+n_t+k-3),\\(n_\ell-1+n_j)\end{array}\right\}$\\
&&&&\\
\hline
&&&&\\
${\mathcal O}(CS^r_{n_1,\cdots,n_k})$&odd&1&$\begin{array}{l}\text{$Q_0$ is c}\\\text{All other $Q_{i}$ are ss}\end{array}$&$n_{\ell}+k-2$\\
&&&&\\
&&$2(j)$&$\begin{array}{l}\text{$Q_0$ is $\text{sc}_j$ and $Q_{j}$ is sc}\\\text{$Q_t$ is c}\\\text{All other $Q_i$ are ss}\end{array}$&$\max\left\{\begin{array}{l}(k-1+n_t),\\(n_\ell-1+n_j)\end{array}\right\}$\\
&&&&\\
&&$3(j)$&$\begin{array}{l}\text{$Q_0$ is $\text{sc}_j$ and $Q_{j}$ is c}\\\text{All other $Q_i$ are ss}\end{array}$&$\max\left\{\begin{array}{l}(n_j+k-2),\\(n_\ell-1+n_j)\end{array}\right\}$\\
&&&&\\
\hline
\end{tabular}
\caption{
The possibilities for a graph $G$ locally equivalent to a clique-star $CS^r_{n_1,\cdots,n_k}$ which is minimal with respect to maximum vertex degree across the entire LC orbit.
The exact case depends on $k$ and its parity and the values of $n_j=\min\{n_1,\cdots,n_k\}$, $n_t=\min\{n_1,\cdots,n_k\}\setminus\{n_j\}$, and $n_\ell=\max\{n_1,\cdots,n_k\}$.
These are identical to the complete $k$-partite case, but with the parity of $k$ reversed.
}
\label{tab:mvd_cases_clique-star}
\end{table*}

\begin{widetext}
\begin{theorem}\label{thm:ckp_min_mvd}
Suppose that $G\in{\mathcal O}(K_{n_1,\dots,n_k})$ is a graph locally equivalent to a complete $k$-partite graph and which has minimal maximum vertex degree across the LC orbit.
Let $n_j=\min\{n_1,\cdots,n_k\}$, $n_t=\min\{n_1,\cdots,n_k\}\setminus\{n_j\}$, and $n_\ell=\max\{n_1,\cdots,n_k\}$.
\begin{enumerate}
\item If $k$ is even, then
\begin{eqnarray}
\Delta(G)&=&\min\begin{cases}n_\ell+k-2&\text{(Case 1)}\\\max\{(k-1+n_t),(n_\ell-1+n_j)\}&\text{(Case 2)}\\\max\{(n_j+k-2),(n_\ell-1+n_j)\}&\text{(Case 3)}\end{cases}.
\end{eqnarray}
\item If $k$ is odd, then
\begin{eqnarray}
\Delta(G)&=&\min\begin{cases}n_\ell+n_j+k-3&\text{(Case 1)}\\\max\{(k-1),(n_\ell-1+n_j)\}&\text{(Case 2)}\\\max\{(n_j+n_t+k-3),(n_\ell-1+n_j)\}&\text{(Case 3)}\end{cases}.
\end{eqnarray}
\end{enumerate}
The structure of $QASST(G)$ matches the corresponding case of Table~\ref{tab:mvd_cases_complete_k-partite}.
\end{theorem}

\begin{theorem}\label{thm:cs_min_mvd}
Suppose that $G\in{\mathcal O}(CS^r_{n_1,\dots,n_k})$ is a graph locally equivalent to a clique-star graph and which has minimal maximum vertex degree across the LC orbit.
Let $n_j=\min\{n_1,\cdots,n_k\}$, $n_t=\min\{n_1,\cdots,n_k\}\setminus\{n_j\}$, and $n_\ell=\max\{n_1,\cdots,n_k\}$.
\begin{enumerate}
\item If $k$ is even, then
\begin{eqnarray}
\Delta(G)&=&\min\begin{cases}n_\ell+n_j+k-3&\text{(Case 1)}\\\max\{(k-1),(n_\ell-1+n_j)\}&\text{(Case 2)}\\\max\{(n_j+n_t+k-3),(n_\ell-1+n_j)\}&\text{(Case 3)}\end{cases}.
\end{eqnarray}
\item If $k$ is odd, then
\begin{eqnarray}
\Delta(G)&=&\min\begin{cases}n_\ell+k-2&\text{(Case 1)}\\\max\{(k-1+n_t),(n_\ell-1+n_j)\}&\text{(Case 2)}\\\max\{(n_j+k-2),(n_\ell-1+n_j)\}&\text{(Case 3)}\end{cases}.
\end{eqnarray}
\end{enumerate}
The structure of $QASST(G)$ matches the corresponding case of Table~\ref{tab:mvd_cases_clique-star}.
\end{theorem}
\end{widetext}

\subsection{Some Special Cases of Graphs Minimizing the Maximum Vertex Degree}

For both complete $k$-partite graphs and clique-stars, recall that the conditions $2\leq n_j\leq n_t\leq n_\ell$ and $3\leq k$ must always be satisfied, but there are many possibilities for the value of $k$ relative to $n_j$, $n_t$, and $n_\ell$.
Although the exact structure of a graph minimizing the maximum vertex degree depends on these relationships, we summarize a few special cases in Table~\ref{tab:special_conditions_minimizing_mvd}.
Depending on the condition, we see that it is possible to find a minimizing graph with all three possibilities of the QASST structure.
Finally, Figures~\ref{fig:ex_special_cases_minimizing_mvd_complete_k_partite} and \ref{fig:ex_special_cases_minimizing_mvd_clique_star} show a few illustrations of these for some fixed parameter values.

\begin{table*}[t]
\centering
\begin{tabular}{|c|c|c|c|c|}
\hline
LC Orbit of $G$&$k$&Condition&Minimum $\Delta(G)$&Case\\
\hline
${\mathcal O}(K_{n_1,\cdots,n_k})$&even&$3\leq k\leq n_j\leq n_t\leq n_{\ell}$&$n_{\ell}+k-2$&1\\
&&$2\leq n_j\leq n_t\leq n_{\ell}<k$&$n_j+k-2$&$3(j)$\\
\hline
${\mathcal O}(K_{n_1,\cdots,n_k})$&odd&$3\leq k\leq n_j\leq n_t\leq n_{\ell}$&$n_\ell-1+n_j$&$2(j)$\\
&&$n_\ell+n_j<k$&$k-1$&$2(j)$\\
\hline
${\mathcal O}(CS^r_{n_1,\cdots,n_k})$&even&$3\leq k\leq n_j\leq n_t\leq n_{\ell}$&$n_\ell-1+n_j$&$2(j)$\\
&&$n_\ell+n_j<k$&$k-1$&$2(j)$\\
\hline
${\mathcal O}(CS^r_{n_1,\cdots,n_k})$&odd&$3\leq k\leq n_j\leq n_t\leq n_{\ell}$&$n_{\ell}+k-2$&1\\
&&$2\leq n_j\leq n_t\leq n_{\ell}<k$&$n_j+k-2$&$3(j)$\\
\hline
\end{tabular}
\caption{
Examples of some special cases for graphs minimizing the maximum vertex degree across the LC orbit based on certain conditions for the parameters. Based on these relationships, any of the three QASST structures is possible.
Observe that clique-stars satisfy the same relationships as complete $k$-partite graphs, but reversed for the parity of $k$.
}
\label{tab:special_conditions_minimizing_mvd}
\end{table*}

\begin{figure*}[t]
\centering
\begin{subfigure}{0.25\textwidth}
\centering
\includegraphics[width=0.7\textwidth,page=15]{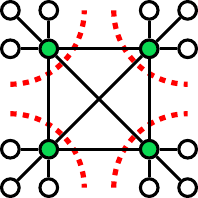}
$\Delta(G)=6$
\caption{
$k=4$ and $n_i=5$
}
\label{fig:mvd_ckp_k4_n5}
\end{subfigure}\begin{subfigure}{0.25\textwidth}
\centering
\includegraphics[width=0.7\textwidth,page=16]{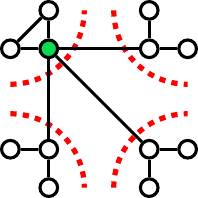}
$\Delta(G)=5$
\caption{
$k=4$ and $n_i=3$
}
\label{fig:mvd_ckp_k4_n3}
\end{subfigure}\begin{subfigure}{0.25\textwidth}
\centering
\includegraphics[width=0.75\textwidth,page=17]{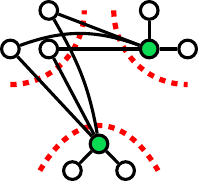}
$\Delta(G)=5$
\caption{
$k=3$ and $n_i=3$
}
\label{fig:mvd_ckp_k3_n3}
\end{subfigure}\begin{subfigure}{0.25\textwidth}
\centering
\includegraphics[width=0.8\textwidth,page=18]{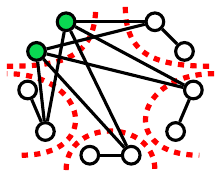}
$\Delta(G)=4$
\caption{
$k=5$ and $n_i=2$
}
\label{fig:mvd_ckp_k5_n2}
\end{subfigure}

\caption{
Some examples of graphs in ${\mathcal O}(K_{n_1,\cdots,n_k})$ which minimize the maximum the vertex degree.
Those vertices achieving this maximum degree are highlighted.
These match the special cases outlined in Table~\ref{tab:special_conditions_minimizing_mvd}.
The strong splits are shown and these match the QASST structure in the corresponding cases described by Table~\ref{tab:mvd_cases_complete_k-partite}.
For simplicity, we assume that $n_1=\cdots=n_k$.
}
\label{fig:ex_special_cases_minimizing_mvd_complete_k_partite}
\end{figure*}

\begin{figure*}[t]
\centering
\begin{subfigure}{0.25\textwidth}
\centering
\includegraphics[width=0.7\textwidth,page=19]{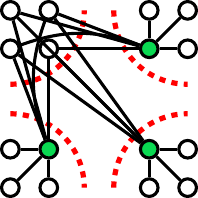}
$\Delta(G)=7$
\caption{
$k=4$ and $n_i=4$
}
\label{fig:mvd_cs_k4_n4}
\end{subfigure}\begin{subfigure}{0.25\textwidth}
\centering
\includegraphics[width=0.7\textwidth,page=20]{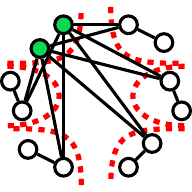}
$\Delta(G)=5$
\caption{
$k=6$ and $n_i=2$
}
\label{fig:mvd_cs_k6_n3}
\end{subfigure}\begin{subfigure}{0.25\textwidth}
\centering
\includegraphics[width=0.7\textwidth,page=21]{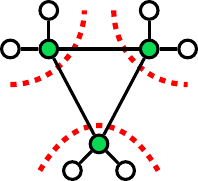}
$\Delta(G)=4$
\caption{
$k=3$ and $n_i=3$
}
\label{fig:mvd_cs_k3_n3}
\end{subfigure}\begin{subfigure}{0.25\textwidth}
\centering
\includegraphics[width=0.7\textwidth,page=22]{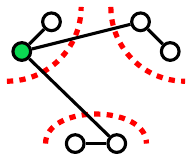}
$\Delta(G)=3$
\caption{
$k=3$ and $n_i=2$
}
\label{fig:mvd_cs_k3_n2}
\end{subfigure}

\caption{
Some examples of graphs in ${\mathcal O}(CS^r_{n_1,\cdots,n_k})$ which minimize the maximum the vertex degree.
Those vertices achieving this maximum degree are highlighted.
These match the special cases outlined in Table~\ref{tab:special_conditions_minimizing_mvd}.
The strong splits are shown and these match the QASST structure in the corresponding cases described by Table~\ref{tab:mvd_cases_clique-star}.
For simplicity, we assume that $n_1=\cdots=n_k$.
}
\label{fig:ex_special_cases_minimizing_mvd_clique_star}
\end{figure*}

\section{Relating Local Equivalence to Isomorphisms}
\label{app:LC_and_isomorphisms}

In many cases, graphs in the same LC orbit will be isomorphic.
In other cases, there exist isomorphic graphs belonging to disjoint LC orbits.
Here we briefly examine the relationship between isomorphic graphs and locally equivalent graphs in some of the special cases we have studied.

\subsection{Isomorphic Graphs in the Same Local Equivalence Class}
\label{app:complte_k-partite_ismorphic_LC_equivalent_graphs}

If two graphs are isomorphic, then they will have isomorphic QASST decompositions. More concretely, a graph isomorphism will induce an isomorphism on the strong split tree and quotient graphs.
Since a graph can always be reconstructed from its QASST by merging together quotient graphs, the reverse is also true.
An isomorphism of QASST decompositions will induce an isomorphism of graphs.
These observations also show that isomorphic graphs must be QASST equivalent, although they need not be locally equivalent.

In the context of graphs QASST equivalent to the complete $k$-partite graph, we would like to characterize those locally equivalent graphs in ${\mathcal O}(K_{n_1,\cdots,n_k})$ or ${\mathcal O}(CS^r_{n_1,\cdots,n_k})$ which are also isomorphic.
Since any graphs in these LC orbits have a QASST decomposition described by the rules in Tables~\ref{tab:complete_k-partite_quasst_equivalences} and \ref{tab:clique-star_quasst_equivalences}, graphs can only be isomorphic if they are contained in the same broad symmetry class.
However, the converse need not hold in general; determining whether two such graphs are isomorphic will also depend on matching the quotient graphs based on their symmetry class.

\begin{figure}[t]
\centering
\includegraphics[width=0.8\linewidth,page=13]{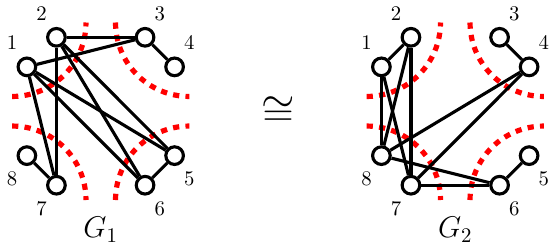}
{\small
\begin{tabular}{ccccccccccccccc}
1&$\mapsto$&7&&3&$\mapsto$&4&&5&$\mapsto$&1&&7&$\mapsto$&5\\
2&$\mapsto$&8&&4&$\mapsto$&3&&6&$\mapsto$&2&&8&$\mapsto$&6
\end{tabular}
}
\caption{
Example of two locally equivalent graphs in ${\mathcal O}(K_{2,2,2,2})$ which are isomorphic via the explicit bijection of vertices given above.
Each belongs to the second symmetry class of Table~\ref{tab:complete_k-partite_quasst_equivalences}.
}
\label{fig:example_isomorphic_graphs_in_LC_orbit}
\end{figure}

This goal is difficult in general, and so we will consider a simplification by restricting to the special case where $n_1=\cdots=n_k$.
Figure~\ref{fig:example_isomorphic_graphs_in_LC_orbit} shows an example of a graph isomorphism between locally equivalent graphs.
Intuitively, identifying such an isomorphism reduces to matching the quotient graphs from the QASST decompositions of each graph.
We formalize this essential idea in the following lemma.

\begin{lemma}\label{thm:complete_k-partite_orbit_isomorphic_graphs}
Let $G,G'\in{\mathcal O}(K_{n_1,\cdots,n_k})$, where $n_1=\cdots=n_k$.
Suppose that $G$ and $G'$ have QASST decompositions $QASST(G)=\{Q_0,Q_1\cdots,Q_k\}$ and $QASST(G')=\{Q_0',Q_1',\cdots,Q_k'\}$, respectively.
Let $I\subseteq[k]$ denote the subset of indices for which the quotient graphs $Q_1,\cdots,Q_k$ of $G$ are star-spoke. Define $I'\subseteq[k]$ similarly for $G'$.
If $G$ and $G'$ belong to the same symmetry class defined in Table~\ref{tab:complete_k-partite_quasst_equivalences} and if $|I|=|I'|$, then $G$ and $G'$ are isomorphic.
\end{lemma}

\begin{proof}
Suppose that $G$ and $G'$ both belong to the first symmetry class of Table~\ref{tab:complete_k-partite_quasst_equivalences}: $Q_0$ and $Q_0'$ are complete graphs; $Q_i$ is star-spoke for each $i\in I$ and $Q_i'$ is star-spoke for each $i\in I'$; $Q_i$ is star-center for each $i\in[k]\setminus I$ and $Q_i'$ is star-center for each $i\in[k]\setminus I'$.
In this case, $|I|=|I'|$ is an even number.
To prove the claim, we will construct an explicit isomorphism $A:G\rightarrow G'$ by patching together isomorphisms on quotient graphs.

Since $I$ and $I'$ have the same size, there exists a bijection of index sets $\alpha:[k]\rightarrow[k]$ for which $\alpha(I)=I'$.
Hence, if $i\in I$, then $Q_i$ is star-spoke in $QASST(G)$ and $Q_{\alpha(i)}'$ is star-spoke in $QASST(G')$.
Since we assumed that $n_1=\cdots=n_k$, we know that $Q_i$ and $Q_{\alpha(i)}'$ have the same number of vertices and are hence isomorphic.
Hence, there exists a quotient graph isomorphism $\alpha_i:Q_i\rightarrow Q_{\alpha(i)}'$, and we may assume that $\alpha_i$ preserves the split-node.
That is, for the split-node $s_i^0\in V(Q_i)$, we have that $\alpha(s_i^0)=s_{\alpha(i)}^0\in V(Q_{\alpha(i)}')$ is the split-node in $Q_{\alpha(i)}'$.

For the central quotient graphs, use the isomorphisms $\alpha_1,\cdots,\alpha_k$ defined above to define a map $\alpha_0:Q_0\rightarrow Q_0'$ in the following way.
For each split-node $s_0^i\in V(Q_0)$, define $\alpha_0(s_0^i)=s_0^{\alpha(i)}\in Q_0'$.
Since $Q_0$ and $Q_0'$ both consist of $k$ split-nodes of this form, this accounts for all of the vertices, and shows that $\alpha_0$ is a bijection of vertex sets.
Since we assumed $G$ and $G'$ are in the first symmetry class of Table~\ref{tab:complete_k-partite_quasst_equivalences}, both $Q_0$ and $Q_0'$ are complete graphs. Thus, $\alpha_0$ is an isomorphism since any bijection of vertex sets between two complete graphs is an isomorphism.

Finally, patch these quotient graph isomorphisms $\alpha_0,\alpha_1,\cdots,\alpha_k$ into a map $A:G\rightarrow G'$.
Since each vertex in $v\in V(G)$ corresponds to a unique leaf-node in some quotient graph $Q_{i>0}$, define $A(v)=\alpha_i(v)\in V(G')$ to be the vertex in $G'$ associated to the corresponding leaf-node $\alpha_i(v)$ in $Q_{\alpha(i)}'$.
It remains to show that $A$ respects edges.

Let $(v,w)\in E(G)$ be an edge. If $v,w\in V(G)$ correspond to leaf-nodes in the same quotient graph $Q_{i>0}$, then we know there exists a corresponding edge $(v,w)\in E(Q_i)$. Since $\alpha_i:Q'\rightarrow Q_{\alpha(i)}'$ is a graph isomorphism, $(\alpha_i(v),\alpha_i(w))\in E(Q_{\alpha(i)}')$ is an edge between leaf-nodes in $Q_{\alpha(i)}'$ and hence corresponds to an edge $(\alpha(v),\alpha(w))\in E(G')$.

Now suppose that $(v,w)\in E(G)$ is an edge for which $v,w\in V(G)$ correspond to leaf-nodes in different quotient graphs, say $v\in V(Q_i)$ and $w\in V(Q_{\ell})$. Knowing that $v$ and $w$ are connected in $G$ implies that we have quotient graph edges $(v,s_i^0)\in E(Q_i)$, $(s_0^i,s_0^{\ell})\in E(Q_0)$, and $(s_{\ell}^0,w)\in E(Q_{\ell})$.
However, under the maps $\alpha_i$, $\alpha_0$, and $\alpha_{\ell}$, these correspond to edges $(\alpha_i(v),\alpha_i(s_i^0))=(\alpha_i(v),s_{\alpha(i)}^0)\in E(Q_{\alpha(i)}')$, $(\alpha_0(s_0^i),\alpha_0(s_0^{\ell}))=(s_0^{\alpha(i)},s_0^{\alpha(\ell)})\in E(Q_0')$, and $(\alpha_{\ell}(s_{\ell}^0),\alpha_{\ell}(w))=(s_{\alpha(\ell)}^0,\alpha_{\ell}(w))\in E(Q_{\alpha(\ell)}')$.
These edges imply that $(\alpha_i(v),\alpha_{\ell}(w))=(\alpha(v),\alpha(w))\in E(G')$ is an edge.
This establishes that $A:G\rightarrow G'$ is an edge-preserving bijection and thus an isomorphism.

Therefore, we may conclude that $G$ and $G'$ are indeed isomorphic graphs when they belong to the first symmetry class of Table~\ref{tab:complete_k-partite_quasst_equivalences}.
If they belong to the second or third symmetry class, the proof is essentially the same with a slight modification.
In this case, $Q_0$ and $Q_0'$ are star graphs pointing towards quotient graphs $Q_j$ and $Q_{j'}'$, respectively. However, $j,j'\notin[k]$ and hence $j\notin I$ and $j'\notin I'$.
Thus, we may require that the isomorphism $\alpha:[k]\rightarrow[k]$ is such that $\alpha(j)=j'$ without affecting the rest of the construction.
This guarantees that $\alpha_j:Q_j\rightarrow Q_{\alpha(j)}'=Q_{j'}'$ and that $\alpha_0:Q_0\rightarrow Q_0'$ is an isomorphism.
\end{proof}

This result was proven for graphs locally equivalent to a complete $k$-partite graph, but the same proof also works for the case of a clique-star. Hence, we state this fact here as an additional lemma.

\begin{lemma}\label{thm:clique-star_orbit_isomorphic_graphs}
Let $G,G'\in{\mathcal O}(CS^r_{n_1,\cdots,n_k})$, where $n_1=\cdots=n_k$.
Suppose that $G$ and $G'$ have QASST decompositions $QASST(G)=\{Q_0,Q_1\cdots,Q_k\}$ and $QASST(G')=\{Q_0',Q_1',\cdots,Q_k'\}$, respectively.
Let $I\subseteq[k]$ denote the subset of indices for which the quotient graphs $Q_1,\cdots,Q_k$ of $G$ are star-spoke. Define $I'\subseteq[k]$ similarly for $G'$.
If $G$ and $G'$ belong to the same symmetry class defined in Table~\ref{tab:clique-star_quasst_equivalences} and if $|I|=|I'|$, then $G$ and $G'$ are isomorphic.
\end{lemma}

Although we have only proven these results for the special case when $n_1=\cdots=n_k$, one can see how the proof idea could be generalized.
In this more general case, the quotient graphs can be divided up based on how many vertices they contain. For those quotient graphs with a fixed number of vertices $n_i$, the number of these quotient graphs which are star-spoke in $G$ must match the number of these quotient graphs which are star-spoke in $G'$. In essence, isomorphic graphs in this particular LC orbit must have the same proportion of quotient graphs of a given size and structure.
An explicit isomorphism can only send one such quotient graph in $G$ to another such quotient graph in $G'$.

In this more general case, it becomes challenging to identify the total number of isomorphism classes within the local equivalence class. Hence, we neglect this in favor of the special case where all $n_i$ have the same size.

\subsection{Local Equivalence Classes up to Isomorphism}
\label{app:complete_k-partite_LC_orbits_up_to_isomorphism}

Finally, we take some time to discuss counting the number of distinct isomorphism classes in the LC orbit of complete $k$-partite graphs and clique-stars in the special case when $n_1=\cdots=n_k$.
Since we work with labeled graphs, every member of the local equivalence class is considered to be distinct. However, many of the graphs in the LC orbit will be isomorphic via a relabeling of vertices, as we showed in Figure~\ref{fig:example_isomorphic_graphs_in_LC_orbit}.
The question of whether two locally equivalent graphs are isomorphic has previously been considered in connection with enumerating the LC orbit~\cite{adcock2020mapping}; Adcock et al. enumerated both the size of the full LC orbit as well as the size of LC orbits up to isomorphism, wherein isomorphic graphs are considered equal.

In the case of graphs locally equivalent to $K_{n_1,\cdots,n_k}$ or $CS^r_{n_1,\cdots,n_k}$, Lemmas~\ref{thm:complete_k-partite_orbit_isomorphic_graphs} and \ref{thm:clique-star_orbit_isomorphic_graphs} show that graph isomorphism classes are determined by their symmetry class as defined in Tables~\ref{tab:complete_k-partite_quasst_equivalences} or \ref{tab:clique-star_quasst_equivalences} and by the number of quotient graphs which are star-spoke (as defined by a choice of index set $I\subseteq[k])$. Hence, the total number of distinct isomorphism classes in ${\mathcal O}(K_{n_1,\cdots,n_k})$ or ${\mathcal O}(CS^r_{n_1,\cdots,n_k})$ corresponds to the number of combinations of symmetry classes and index set sizes $|I|$.

For each of the three broad symmetry classes, this question reduces to counting the number of even or odd integers between 0 and $k$ (for Case 1) or between 0 and $k-1$ (for Cases 2 and 3).
Since we assume that $n_1=\cdots=n_k$, the choice $j\in[k]$ for Cases 2 and 3 is unimportant since all quotient graphs are symmetric, but we must exclude a single index from $[k]$ when counting the size of possible subsets.
For Case 1, the number of even integers between 0 and $k$ is $\lfloor\frac{k}{2}\rfloor+1$ and the number of odd integers is $\lceil\frac{k}{2}\rceil$.
The same formulas work for Cases 2 and 3 using $k-1$ instead of $k$.
Combining these and simplifying gives the formulas obtained in the following theorems.

\begin{theorem}\label{thm:complete_k-partite_LC_orbit_up_to_isomorphism}
Consider the LC orbit of the complete $k$-partite graph $K_{n_1,\cdots,n_k}$, where $n_1=\cdots=n_k\geq2$ and $k\geq3$. The number of distinct graph isomorphism classes in this orbit is
\begin{eqnarray}\label{eq:complete_k-partite_LC_orbit_size_up_to_isomorphism}
\left|{\mathcal O}(K_{n_1,\cdots,n_k})/\cong\right|&=&\lfloor\tfrac{k}{2}\rfloor+k+1.
\end{eqnarray}
Each isomorphism class corresponds to a choice of symmetry class in Table~\ref{tab:complete_k-partite_quasst_equivalences} and choice of index set $I\subseteq[k]$, where only the size of $I$ is important.
\end{theorem}

\begin{theorem}\label{thm:clique-star_LC_orbit_up_to_isomorphism}
Consider the LC orbit of the clique-star $CS^r_{n_1,\cdots,n_k}$, where $n_1=\cdots=n_k\geq2$ and $k\geq3$. The number of distinct graph isomorphism classes in this orbit is
\begin{eqnarray}\label{eq:clique-star_LC_orbit_up_to_isomorphism}
\left|{\mathcal O}(CS^r_{n_1,\cdots,n_k})/\cong\right|&=&\lceil\tfrac{k}{2}\rceil+k.
\end{eqnarray}
Each isomorphism class corresponds to a choice of symmetry class in Table~\ref{tab:clique-star_quasst_equivalences} and choice of index set $I\subseteq[k]$, where only the size of $I$ is important.
\end{theorem}

\begin{figure}[t]
\centering
\includegraphics[width=\linewidth,page=14]{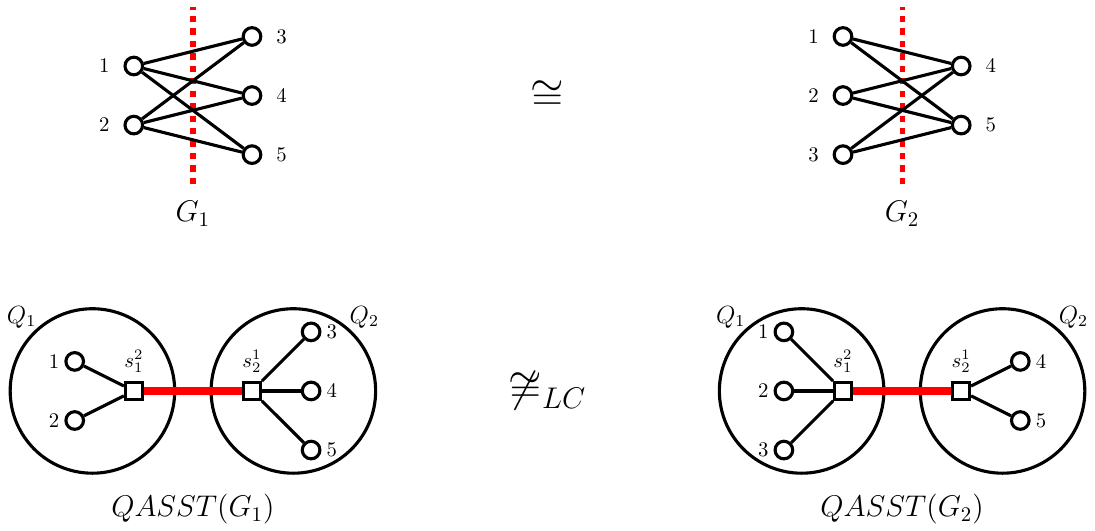}
\caption{
Two isomorphic graphs which are not locally equivalent as labeled graphs (i.e.~these graphs are not contained in the same LC orbit), a fact made clear by comparing the QASST decompositions.
}
\label{fig:example_of_non_LC_equivalent_isomorphic_graphs}
\end{figure}

We end this subsection with a final comment about the relationship between isomorphic graphs and LC orbits.
In general, two isomorphic graphs defined on the same vertex set need not be in the same LC orbit (Figure~\ref{fig:example_of_non_LC_equivalent_isomorphic_graphs}).
However, if $G_1\cong G_2$ are two isomorphic graphs, then the graph-labeled multi-graphs defined by their LC orbits (in the sense of Figure~\ref{fig:example_K4_orbit}) will also be isomorphic, but they may or may not be identical.
If we only consider graphs up to isomorphism (i.e.~if we work with unlabeled graphs rather than labeled graphs), then there is no distinction between these multi-graphs representing the LC orbit.

\end{document}